\documentclass[11pt,a4paper]{article}

\usepackage[left=2.45cm, top=2.45cm,bottom=2.45cm,right=2.45cm]{geometry}

\usepackage{amsmath,amssymb,amsthm,mathrsfs,calc,graphicx,stmaryrd,xcolor}
\usepackage[british]{babel}
\usepackage{amsfonts}              

\newtheoremstyle{definition}
{10pt}
{10pt}
{}
{}
{\bfseries}
{}
{.5em}
{}

\newtheoremstyle{plain}
{10pt}
{10pt}
{\itshape}
{}
{\bfseries}
{}
{.5em}
{}

\theoremstyle{plain}	
\newtheorem{thm}{Theorem}[section]
\newtheorem{lem}[thm]{Lemma}
\newtheorem{prop}[thm]{Proposition}
\newtheorem{cor}[thm]{Corollary}

\theoremstyle{definition}	
\newtheorem{definition}[thm]{Definition}
\newtheorem{remark}[thm]{Remark}

\newcommand{\RR}{\mathbb{R}}      

\newcommand{\vol}{\operatorname{vol}}
\newcommand{\Po}{\operatorname{Po}}
\newcommand{\intt}{\operatorname{int}}
\newcommand{\cl}{\operatorname{cl}}

\def\BB{\mathbb{B}}

\def\EE{\mathbb{E}}

\def\NN{\mathbb{N}}
\def\PP{\mathbb{P}}

\def\RR{\mathbb{R}}
\def\SS{\mathbb{S}}

\def\BBd{\mathbb{B}^d}
\def\SSd{\mathbb{S}^{d-1}}

\def\g{\gamma}

\def\la{\lambda}

\def\bv{\mathbf{v}}
\def\bx{\mathbf{x}}

\def\cB{\mathcal{B}}
\def\cC{\mathcal{C}}

\def\cF{\mathcal{F}}

\def\cM{\mathcal{M}}

\def\cP{\mathcal{P}}

\def\cX{\mathcal{X}}

\def\dint{\textup{d}}

\def\var{{\textup{var}}}

\def\cl{{\rm cl}}

\def\cone{{\rm cone}}
\def\ext{{\rm ext}}
\def\setk{\llbracket k\rrbracket}

\begin{document}

\title{\bfseries Gaussian polytopes:\ a cumulant-based approach}

\author{Julian Grote\footnotemark[1]\ \footnotemark[2]\ \ and Christoph Th\"ale\footnotemark[3]}

\date{}
\renewcommand{\thefootnote}{\fnsymbol{footnote}}
\footnotetext[1]{Ruhr University Bochum, Faculty of Mathematics, D-44780 Bochum, Germany. E-mail: julian.grote@rub.de}

\footnotetext[2]{The first author has been supported by the German Research Foundation (DFG) via Research Training Group RTG 2131 \textit{High dimensional Phenomena in Probability -– Fluctuations and Discontinuity}.}

\footnotetext[3]{Ruhr University Bochum, Faculty of Mathematics, D-44780 Bochum, Germany. E-mail:  christoph.thaele@rub.de}

\maketitle

\begin{abstract}
The random convex hull of a Poisson point process in $\mathbb{R}^d$ whose intensity measure is a multiple of the standard Gaussian measure on $\mathbb{R}^d$ is investigated. The purpose of this paper is to invent a new viewpoint on these Gaussian polytopes that is based on cumulants and the general large deviation theory of Saulis and Statulevi\v{c}ius. This leads to new and powerful concentration inequalities, moment bounds, Marcinkiewicz-Zygmund-type strong laws of large numbers, central limit theorems and moderate deviation principles for the volume and the face numbers. Corresponding results are also derived for the empirical measures induced by these key geometric functionals, taking thereby care of their spatial profiles.
\end{abstract}

\tableofcontents

\section{Introduction and main results}

\subsection{General introduction}

Random polytopes are among the most central objects studied in stochastic geometry, a branch of mathematics at the borderline between convex geometry and probability. One class of random polytopes that has attracted particular interest is that of Gaussian polytopes which arise as convex hulls of a collection of independent random points that are distributed in $\RR^d$ according to the standard Gaussian law, see the survey articles of Hug \cite{HugSurvey} and Reitzner \cite{ReitznerSurvey}. 

Gaussian polytopes are highly relevant in asymptotic convex geometry or the local theory of Banach spaces. Since the breakthrough paper of Gluskin \cite{Gluskin} it has been realized and is nowadays well known that Gaussian polytopes -- and their close relatives -- often serve as extremizers in geometric or analytic problems. For example, consider the two random convex hulls ${\rm conv}([-1,1]^d\cup\{\pm X_1^{(i)},\ldots,\pm X_d^{(i)}\})$, $i\in\{1,2\}$, formed by the union of $[-1,1]^d$ with two independent and symmetrized Gaussian polytopes in $\RR^d$, respectively, that arise from independent standard Gaussian points $X_1^{(i)},\ldots,X_d^{(i)}$, $i\in\{1,2\}$. Then, with probability tending to $1$ exponentially fast as the space dimension $d$ tends to infinity, these two random convex hulls -- or, equivalently, the random $d$-dimensional normed spaces that have these polytopes as their respective unit balls -- have Banach-Mazur distance bounded from below by a constant multiple of $d$. The value $d$ also constitutes an upper bound for this quantity by the classical John's theorem. For a generalization of this result to certain sub-Gaussian polytopes we refer to the work of Lata\l a, Mankiewicz, Oleszkiewicz and Tomczak-Jaegermann \cite{LatalaeEtAl} and also to the survey of Mankiewicz and Tomczak-Jaegermann \cite{MankiewiczHANDBOOK}. Further extremality results in this context are due to Gluskin and Litvak \cite{GluskinLitvak} or Szarek \cite{Szarek}, to name just a few.

In addition, Gaussian polytopes are prototypical examples of random convex sets that are known to satisfy the (probabilistic version of the) celebrated hyperplane conjecture. More precisely, it was shown by Klartag and Kozma \cite{KK} that the isotropic constant of the convex hull of $n\geq d+1$ independent Gaussian random points in $\RR^d$ is bounded by a universal constant with probability at least $1-e^{-cd}$, where $c\in (0,\infty)$ is another universal constant. In other words, Gaussian polytopes satisfy the hyperplane conjecture asymptotically almost surely, as $d\to\infty$. For other random polytope models satisfying this form of the hyperplane conjecture we refer to the works of Alonso-Guti\'errez \cite{AlonsoGutierrez.2008}, Dafnis, Gu\'edon and Giannopoulos \cite{DafnisGiannopGuedon}, H\"orrmann, Hug, Reitzner and Th\"ale \cite{HHRT} and H\"orrmann, Prochno and Th\"ale \cite{HPT}.

Gaussian polytopes are also of interest in some branches of coding theory because of the following interpretation that has been derived by Baryshnikov and Vitale \cite{BaryshnikovVitale}. Fix $n\geq d+1$ and let $\Delta_{n}$ be a regular simplex in $\RR^{n}$. Now, take a random rotation $\varrho$ in $\RR^{n}$ and let ${\rm pr}_d^{n}:\RR^n\to\RR^d$ be the projection onto the first $d$ coordinates. Then, up to an affine transformation, the randomly rotated and projected simplex ${\rm pr}_d^{n}(\varrho(\Delta_{n}))$ has the same distribution as a Gaussian polytope that arises as the convex hull of $n+1$ standard Gaussian random points in $\RR^d$. In the context of coding theory it is of interest whether the randomly rotated and projected simplices ${\rm pr}_d^{n}(\varrho(\Delta_{n}))$ are $k$-neighbourly in that the convex hull of the projection of $k$ vertices of $\Delta_{n}$ is always a $(k-1)$-face of ${\rm pr}_d^{n}(\varrho(\Delta_{n}))$. It turns out that, as $n\to\infty$, this is indeed the case, at least as long as $k$ and $d$ are both proportional to $n$. For more details in this direction we refer to the works of Candes, Rudelson, Tao and Vershynin \cite{CandesRudelsonTaoVershynin}, Candes and Tao \cite{CandesTao}, Donoho and Tanner \cite{DonohoTanner1,DonohoTanner2,DonohoTanner3}, and Vershik and Sporyshev \cite{VershikSporyshev}. 

Bearing in mind that spatial data is often being assumed to follow a Gaussian law, Gaussian polytopes have a clear relevance also in the area of multivariate statistics. For example, the vertices of a Gaussian polytope can be regarded as the multivariate extremes of the underlying sample. More generally, the convex hull of a spatial point configuration is also related to the notion of convex hull peeling data depth that allows an ordering of spatial data. For more information on this point we refer to the survey article of Cascos \cite{Cascos}.

\medskip

The purpose of the present paper is to introduce a new probabilistic viewpoint of Gaussian polytopes in $\RR^d$ that allows to gain new insights into their large scale asymptotic geometry. It is based on sharp bounds for cumulants and the large deviation theory of Saulis and Statulevi\v{c}ius. By means of these techniques we will derive a number of new and powerful results that were not within the reach of other methods available before.  The geometric characteristics we consider are on the one hand related to the metric and on the other hand related to the combinatorial structure of Gaussian polytopes and given by the volume as well as the face numbers, respectively. Using an underlying binomial point process in the construction of Gaussian polytopes, that is, a fixed number of independent random points that are distributed according to the standard Gaussian law in $\RR^d$, the expectation asymptotics of these parameters have been determined in a series of papers. It starts with the classical work of R\'enyi and Sulanke \cite{RenyiSulanke} for the vertex number in the planar case $d=2$ and is continued by the papers of Affentranger \cite{Affentranger} and Affentranger and Schneider \cite{AffentrangerSchneider}, dealing with the face numbers in higher dimensions. An integral-geometric approach to mean values for Gaussian polytopes has been developed by Hug, Munsonius and Reitzner \cite{HugMunsoniusReitzner}. Hug and Reitzner \cite{HugReitzner} also derived variance upper bounds and used them to show laws of large numbers. The central limit problem for Gaussian polytopes has first been treated by Hueter \cite{Hueter94,Hueter99} and finally been settled in the breakthrough paper of B\'ar\'any and Vu \cite{BaranyVu}. 

Throughout the present text we will assume that the random polytopes we deal with are generated as convex hulls of a Poisson point process whose intensity measure is given by a multiple $\la>0$ of the standard Gaussian measure in $\RR^d$. Iniciated in \cite{BaranyVu}, the line of research under this further randomization has recently been taken up in the remarkable work of Calka and Yukich \cite{CalkaYukich}, who computed the precise variance asymptotics and considered scaling limits by means of a scaling transformation that has its origins in previous works of Calka, Schreiber and Yukich \cite{CalkaSchreiberYukich} and Schreiber and Yukich \cite{SchreiberYukich} on random polytopes in the unit ball. We use this scaling transformation and its properties to obtain further probabilistic results for Gaussian polytopes generated by a Poisson point process. In particular, these are
\begin{itemize}
\item[--] concentration inequalities (Theorem \ref{ConcentrationEinleitung} and Theorem \ref{thm:ConcentrationEmpMeas}),
\item[--] bounds for the growth of moments and cumulants of all orders (Theorem \ref{MomentboundEinleitung} and Theorem \ref{cumulantestimate}),
\item[--] Marcinkiewicz-Zygmund-type strong laws of large numbers (Theorem \ref{StarkesGesetzVolumen}),
\item[--] central limit theorems and bounds on the relative error in the central limit theorems (Theorem \ref{DeviationprobabilityEinleitung} and Theorem \ref{thm:DeviationsEmpMeas}),
\item[--] moderate deviation principles (Theorem \ref{ModeratedeviationsEinleitung}, Theorem \ref{moderatedeviations} and Theorem \ref{thm:MDPMeasureGeneral})
\end{itemize}
for the key geometric characteristics discussed above as well as their measure-valued counterparts. We emphasize that the latter have the advantage to capture the spatial profile of the functionals we consider as well and not only their total masses. This naturally continues the recent line of research on the probabilistic analysis of large (or high-dimensional) geometric random systems.

\medbreak

Let us briefly outline the structure and the content of this paper. In Section \ref{subsec:MainResults} we present our main findings for the volume and the face numbers of Gaussian polytopes. Some preliminaries, especially the scaling transformation from \cite{CalkaYukich} as well as the measure-valued versions of our functionals are described in Section \ref{sec:Preliminaries}. We also introduce there a cluster measure representation of their cumulant measures, that will turn out to be crucial for later purposes. The main results for the empirical measures are stated in Section \ref{sec:MainResultsEmpirical}, while their proofs as well as the proofs of the results in Section 1.2 are the content of Section \ref{sec:ProofMain}. Their proofs are based on two auxiliary estimates. The first one is a moment bound that we provide in Section \ref{sec:MomentEstimates}, while the final Section \ref{sec:CumulantProof} contains the proof of the second ingredient, a cumulant bound, which is the main technical device we develop and apply in our text.

\subsection{Statement of the main results}\label{subsec:MainResults}

Fix a space dimension $d\ge 2$ and let $\gamma_d$ be the standard Gaussian measure on $\RR^d$ that has density
\begin{align}\label{DefinitionPhi}
\phi(x) := (2\pi)^{-d/2}\, \exp\big(-\|x\|^2/2\big), \qquad x\in \RR^d,
\end{align}
with respect to the Lebesgue measure on $\RR^d$. Here and in what follows, $\|\,\cdot\,\|$ stands for the Euclidean norm. By $\cP_\la$ we denote a Poisson point process on $\RR^d$ with intensity measure $\lambda\gamma_d$, $\la>0$. That is, if $N(\la)$ is a Poisson random variable with mean $\la$, $\cP_\la$ is a random set that consists of $N(\la)$ points in $\RR^d$ that are independently chosen according to the Gaussian law $\g_d$. By $K_\la$ we denote the Gaussian polytope that arises as the convex hull of $\cP_\la$. The volume and the number of $j$-dimensional faces of $K_\la$, $j\in\{0,\ldots,d-1\}$, are denoted by $\vol(K_\la)$ and $f_j(K_\la)$, respectively.

\medskip

We can now present our main results.  Let us start with a new and powerful concentration inequality for the volume $\vol(K_\la)$ and the number $f_j(K_\la)$ of $j$-dimensional faces. To streamline our presentation, we define the individual weights $z[f_0]:=d$ and $z[f_j]:=0$ if $j\in\{1,\ldots,d-1\}$. Here and in what follows, writing that a statement holds \textit{for sufficiently large $\la$} means that there exists a $\la_0>0$, depending on the dimension $d$ and the geometric functional under consideration, such that the statement is valid for all $\la\ge \la_0$.

\begin{thm}[Concentration inequalities]\label{ConcentrationEinleitung}
	\begin{description}
	\item[(i)] Let $y\geq 0$. Then,
		\begin{align*}
			&\PP\big(|\vol(K_\lambda)-\EE [\vol(K_\la)]|\geq y\, \sqrt{\var [\vol(K_\la)]}\,\big)\\
			&\qquad \qquad \le 2\exp\Big(-{1\over 4}\min\Big\{{y^2\over 2^{3d+5}},c_1\,  (\log \la)^{d-1\over 4(3d+5)}\, y^{1\over 3d+5}\Big\}\Big)
			\end{align*}
			for all sufficiently large $\la$ with a constant $c_1\in(0,\infty)$ only depending on $d$.	
	\item[(ii)] Let $y\geq 0$ and $j\in \{0,\ldots,d-1\}$. Then, 
		\begin{align*}
		&\PP\big(|f_j(K_\lambda)-\EE [f_j(K_\la)]|\geq y\, \sqrt{\var [f_j(K_\la)]}\,\big)\\
		&\qquad \qquad \le 2\exp\Big(-{1\over 4}\min\Big\{{y^2\over 2^{j(3d+1) + 5 + z[f_j]}},c_2\, (\log \la)^{d-1\over 4(j(3d+1) + 5 + z[f_j])}\, y^{1\over j(3d+1) + 5 + z[f_j]}\Big\}\Big)
		\end{align*}
	for all sufficiently large $\la$ with a constant $c_2\in(0,\infty)$ only depending on $d$ and $j$.
	\end{description}
\end{thm}

\begin{remark}
\begin{itemize}
\item[(i)] The concentration inequalities in the previous theorem can be represented in an alternative form that avoids the minimum in the exponential exponent. For example, the volume of $K_\la$ satisfies
\begin{align*}
& \PP\big(|\vol(K_\lambda)-\EE [\vol(K_\la)]|\geq y\, \sqrt{\var [\vol(K_\la)]}\,\big)\\
&\qquad\qquad \leq 2\,\exp\Bigg(-{y^2\over 2\big(2^{3d+5}+c\, y(\log\la)^{(d-1)/(12(2d+3))}\big)^{3(2d+3)/(3d+5)}}\Bigg)
\end{align*}
with a constant $c\in(0,\infty)$ only depending on $d$ for all $y\geq 0$ and sufficiently large $\la$. This follows by combining the estimates we obtain in the proof of Theorem \ref{ConcentrationEinleitung} with \cite[Lemma 2.4]{SaulisBuch}. However, it turns out that this form is less suitable for our purposes compared to that of Theorem \ref{ConcentrationEinleitung}.
\item[(ii)] For small arguments $y$ the Gaussian exponent $-y^2$ is already optimal. To improve the (presumably non-optimal) exponent for larger values of $y$ by our method, which is based on sharp bounds for cumulants, one would need to improve the cumulant bound in Theorem \ref{cumulantestimate}. This point will further be discussed in Remark \ref{rem:Optimality} below.
\end{itemize}
\end{remark}

A strong law of large numbers dealing with the volume of Gaussian polytopes constructed from an underlying binomial point process has been derived by Hug and Reitzner \cite{HugReitzner} using the Chebychev inequality together with an upper bound on the variance obtained in the same paper. Using our concentration inequality from Theorem \ref{ConcentrationEinleitung} (i) we prove a stronger result for our Poisson point process-based model, which has the form of a Marcinkiewicz-Zygmund-type strong law. While the classical strong law of large numbers for the random variables $\vol(K_\la)$ says that
\begin{equation*}
	{\vol(K_{\la_k})-\EE[\vol(K_{\la_k})]\over (\log \la_k)^{d\over 2}}\longrightarrow 0\,,\qquad\text{as } k\to\infty\,,
\end{equation*}
with probability one, along all subsequences $\la_k$ of the form $\lambda_k:=a^k$, $a>1$, our Marcinkiewicz-Zygmund-type strong law makes a statement about the almost sure convergence to zero, as $k\to\infty$, of the random variables
$$
{\vol(K_{\la_k})-\EE[\vol(K_{\la_k})]\over (\log \la_k)^{p{d\over 2}}}
$$
\textit{for all} $p>{d-3\over 2d}$, again along all subsequences of the form $\la_k = a^k$, $a > 1$. While for $p\geq 1$ such a result is a consequence of a classical strong law, the situation for $p<1$ is not covered by such a result. We notice that for $p={d-3\over 2d}$ the denominator in the above expression equals $(\log \la_k)^{\frac{d-3}{4}}$, which is precisely the rescaling that is necessary in the central limit theorem for the volume of $K_\la$, see Corollary \ref{CLTEinleitung} below. Indeed, it holds that $\var[\vol(K_{\la_k})] \sim c (\log \la_k)^{\frac{d-3}{2}}$, as $k\rightarrow \infty$, where $c\in (0,\infty)$ is a constant just depending on $d$, see \cite{CalkaYukich}. This implies that our condition on $p$ is in fact optimal and covers the whole possible range of parameters $p$. In contrast to the volume functional and even in the case of an underlying binomial point process, a strong law of large numbers for the face numbers of Gaussian polytopes does not exist so far. In part (ii) of the next theorem we present the first such result. Again, our condition on the parameter $p$ is best possible.

\begin{thm}[Marcinkiewicz-Zygmund-type strong laws of large numbers]\label{StarkesGesetzVolumen}
Let $(\la_k)_{k\in \NN}$ be a sequence of real numbers defined by $\la_k = a^k$, $a>1$.
	\begin{description}
		\item[(i)] Fix $p > \frac{d-3}{2d}$. Then, as $k \rightarrow \infty$, one has that
		\begin{align*}
			\frac{\vol(K_{\la_k}) - \EE [\vol(K_{\la_k})]}{(\log \la_k)^{p \frac{d}{2}}} \longrightarrow 0 
		\end{align*}
		with probability one.
		\item[(ii)] Fix $j\in\{0,\ldots,d-1\}$ and $p > \frac{1}{2}$. Then, as $k\to\infty$, one has that
		\begin{align*}
			\frac{f_j(K_{\la_k}) - \EE [f_j(K_{\la_k})]}{(\log{\la_k})^{p \frac{d-1}{2}}} \longrightarrow 0 
		\end{align*}
		with probability one.
	\end{description}
\end{thm}

\begin{remark}
\begin{itemize}
\item[(i)] In \cite{HugReitzner} the strong law of large numbers for the volume was proved for the binomial counterpart of the Gaussian random polytopes $K_{\lambda_k}$ along the subsequence $\lambda_k=2^k$ and then extended by monotonicity arguments to $\lambda_k=k$ (see \cite[Corollary 1.4]{HugReitzner}). In our set-up such an extension by monotonicity is not possible, since the centred random variables  $\vol(K_{\la_k}) - \EE [\vol(K_{\la_k})]$ and also $f_j(K_{\la_k}) - \EE [f_j(K_{\la_k})]$ ($j\in\{0,\ldots,d-1\}$) are not monotone in $k$.
\item[(ii)] As already pointed out in \cite{HugReitzner}, a classical strong law of large numbers for the volume of $K_{\la_k}$ can in principle also be deduced from the work of Geffroy \cite{Geffroy}. However, this is not the case for our Marcinkiewicz-Zygmund-type strong law. In addition, this approach also fails for the face numbers of Gaussian polytopes.
\end{itemize}
\end{remark}

As a consequence of the cumulant bound presented in Theorem \ref{cumulantestimate} we obtain upper and lower bounds for the $k$th moments of the volume and the face numbers of $K_\la$ that differ by a factor $c^k k!$, where $c\in (0,\infty)$ is a suitable constant.

\begin{thm}[Moment bounds]\label{MomentboundEinleitung}
\begin{description}
\item[(i)] There are constants $c_1,\ldots,c_4\in(0,\infty)$ only depending on $d$ such that
$$
c_1\, c_2^k\, (\log \la)^{k\frac{d}{2}} \le \EE[\vol(K_\la)^k] \le c_3\,c_4^k\, k!\, (\log \la)^{k\frac{d}{2}}
$$
for all sufficiently large $\la$ and $k\in \NN$.
\item[(ii)] Let $j\in \{0,\ldots,d-1\}$. Then there are constants $c_5,\ldots,c_{8}\in (0,\infty)$ that only depend on $d$ and $j$ such that 
$$
c_5\, c_6^k\, (\log \la)^{k\frac{d-1}{2}} \le \EE [f_j(K_\la)^k] \le c_7\, c_{8}^k\, k!\, (\log \la)^{k\frac{d-1}{2}} 
$$
for all sufficiently large $\la$ and $k\in \NN$.
\end{description}
\end{thm}


One of the main results of the work of B\'ar\'any and Vu \cite{BaranyVu} is a central limit theorem for suitably normalized versions of the random variables $\vol(K_\la)$ and $f_j(K_\la)$. While we are able to recover their result with a weaker rate of convergence (see Theorem \ref{CLTempiricalmeasure} below and the discussion at the end of Section \ref{sec:MainResultsEmpirical}), our technique also delivers an estimate for the relative error in the central limit theorem that was not available before. To formulate our result, let $\Phi(x):=(2\pi)^{-1/2}\int_{-\infty}^x e^{-t^2/2}\,\dint t$, $x\in\RR$, be the distribution function of a standard Gaussian random variable.

\begin{thm}[Bounds on the relative errors in the central limit theorems]\label{DeviationprobabilityEinleitung}
\begin{description}
\item[(i)] Let us assume that $0\leq y\leq c_1\, (\log \la)^{d-1\over 4(6d+9)}$ and that $\la$ is sufficiently large. Then, one has that
	\begin{align*}
	\Bigg|\log{\PP\big(\vol(K_\la)-\EE [\vol(K_\la)]\geq y\, \sqrt{\var [\vol(K_\la)]}\,\big)\over 1-\Phi({y})}\Bigg| &\leq c_2\,(1+y^3)\, (\log \la)^{-{d-1\over 4(6d+9)}}\,,\\
	\Bigg|\log{\PP\big(\vol(K_\la)-\EE [\vol(K_\la)]\leq -y\, \sqrt{\var [\vol(K_\la)]}\,\big)\over \Phi(-{y})}\Bigg| &\leq c_2\,(1+y^3)\, (\log \la)^{-{d-1\over 4(6d+9)}}\,,
	\end{align*}
	with constants $c_1,c_2\in(0,\infty)$ only depending on $d$.
\item[(ii)]	For $0\leq y\leq c_3 (\log \la)^{d-1\over 4(2j(3d + 1) + 9 + 2z[f_j])}$, $j\in \{0,\ldots,d-1\}$ and sufficiently large $\la$ one has that 
	\begin{align*}
	\Bigg|\log{\PP\big(f_j(K_\la)-\EE [f_j(K_\la)]\geq y\, \sqrt{\var [f_j(K_\la)]}\,\big)\over 1-\Phi({y})}\Bigg| &\leq c_4\,(1+y^3)\, (\log \la)^{-{d-1\over 4(2j(3d + 1) + 9 + 2z[f_j])}}\,,\\
	\Bigg|\log{\PP\big(f_j(K_\la)-\EE [f_j(K_\la)]\leq -y\, \sqrt{\var [f_j(K_\la)]}\,\big)\over \Phi(-{y})}\Bigg| &\leq c_4\,(1+y^3)\, (\log \la)^{-{d-1\over 4(2j(3d + 1) + 9 + 2z[f_j])}}
	\end{align*}
	with constants $c_3,c_4\in(0,\infty)$ only depending on $d$ and $j$. 	
\end{description}
\end{thm}

The above theorem immediately implies that the random variables $\vol(K_\la)$ and $f_j(K_\la)$ satisfy a central limit theorem. Indeed, for fixed $y\in\RR$ not depending on $\la$, one has that the quantities on the right hand sides of the inequalities in Theorem \ref{DeviationprobabilityEinleitung} converge to zero, as $\la\to\infty$. Although this is known from \cite{BaranyVu}, as anticipated above, we formulate this observation as a corollary.

\begin{cor}[Central limit theorems]\label{CLTEinleitung}
	As $\la\to\infty$, the random variables
	$$
	{\vol(K_\la)-\EE [\vol(K_\la)]\over\sqrt{\var [\vol(K_\la)]}}\qquad\text{and}\qquad{f_j(K_\la)-\EE [f_j(K_\la)]\over\sqrt{\var [f_j(K_\la)]}}
	$$
	with $j\in\{0,\ldots,d-1\}$ satisfy a central limit theorem, that is, they converge in distribution to a standard Gaussian random variable.
\end{cor} 

After having investigated concentration inequalities, strong laws of large numbers and (the relative errors in) the central limit theorems, we turn now to moderate deviation principles for the volume and the face numbers of Gaussian polytopes. For convenience and to keep the paper self-contained, let us recall from \cite[Chapter III.1]{DenHollander} what this formally means.

\begin{definition}\label{def:LDPMDP}
A family $(\nu_\lambda)_{\lambda > 0}$ of probability measures on a topological space $E$ fulfils a large deviation principle with speed $a_\lambda$ and (good) rate function $I : E \rightarrow [0,\infty]$ if $I$ is lower semi-continuous, has compact level sets and if for every Borel set $B\subseteq E$,
\begin{align*}
-\inf\limits_{x\in \intt(B)} I(x) \leq \liminf\limits_{\lambda \rightarrow \infty} \frac{1}{a_\lambda} \log \nu_\lambda (B) \leq \limsup\limits_{\lambda \rightarrow \infty} \frac{1}{a_\lambda} \log \nu_\lambda (B) \leq -\inf\limits_{x\in \cl(B)} I(x)\,,
\end{align*}
where $\intt(B)$ and $\cl(B)$ stand for the interior and the closure of $B$, respectively. A family $(X_\lambda)_{\lambda > 0}$ of random elements in $E$ satisfies a large deviation principle with speed $a_\lambda$ and rate function $I : E \rightarrow [0,\infty]$, if the family of their distributions does. Moreover, if the involved random elements $(X_\lambda)_{\lambda > 0}$ satisfy a strong law of large numbers and a central limit theorem, and if the rescaling $a_\lambda$ lies between that of a law of large numbers and that of a central limit theorem, one usually speaks about a moderate deviation principle instead of a large deviation principle with speed $a_\la$ and rate function $I$.
\end{definition} 

Our main result in this context is a moderate deviation principle for the volume and one for the number of $j$-dimensional faces of the Gaussian polytopes $K_\la$ for all $j\in\{0,\ldots,d-1\}$. Although moderate (or large) deviations belong to the class of classical limit theorems in probability theory, to the best of our knowledge they have not been investigated in the context of Gaussian polytopes so far.

\begin{thm}[Moderate deviation principles]\label{ModeratedeviationsEinleitung}
\begin{description}
\item[(i)]  Let $(a_\lambda)_{\lambda > 0}$ be a sequence of real numbers with
	\begin{align*}
	\lim\limits_{\lambda \rightarrow \infty} a_\lambda = \infty \quad \text{and} \quad \lim\limits_{\lambda \rightarrow \infty} a_\lambda\,  (\log \la)^{-\frac{d-1}{4(6d + 9)}} = 0\,.
	\end{align*}
	Then, the family $$\left(\frac{1}{a_\la} \frac{\vol(K_\la) - \EE[\vol(K_\la)]}{\sqrt{\var [\vol(K_\la)]}} \right)_{\la> 0}$$ satisfies a moderate deviation principle on $\RR$ with speed $a_\la^2$ and rate function $I(x) = \frac{x^2}{2}$.
\item[(ii)] Let $j\in\{0,\ldots,d-1\}$ and let $(a_\la)_{\lambda > 0}$ be a sequence of real numbers satisfying
	\begin{align*}
	\lim\limits_{\lambda \rightarrow \infty} a_\lambda = \infty \quad \text{und} \quad \lim\limits_{\lambda \rightarrow \infty} a_\lambda\,  (\log \la)^{-\frac{d-1}{4(2j(3d + 1) + 9 + 2z[f_j])}} = 0\,.
	\end{align*}
	Then, the family of random variables $$\left(\frac{1}{a_\la} \frac{f_j(K_\la) - \EE[f_j(K_\la)]}{\sqrt{\var [f_j(K_\la)]}} \right)_{\la> 0}$$ satisfies a moderate deviation principle on $\RR$ with speed $a_\la^2$ and rate function $I(x)={x^2\over 2}$.	
\end{description}
\end{thm}

The results that we have presented in this section have immediate consequences for the model of randomly rotated and projected simplices briefly discussed in the previous section if we randomize the model further. Namely, we let the space dimension $n=N(\la)$ be an independent random integer that is Poisson distributed with parameter $\la$ and think of ${\rm pr}_d^{N(\la)}(\varrho(\Delta_{N(\la)}))$ as already being embedded in $\RR^d$ in the case that $N(\la)\leq d$ (the probability of this event tends to zero, as $\la\to\infty$). Then, we conclude  for the face numbers $f_j({\rm pr}_d^{N(\la)}(\varrho(\Delta_{N(\la)})))$ of ${\rm pr}_d^{N(\la)}(\varrho(\Delta_{N(\la)}))$ for all $j\in\{0,\ldots,d-1\}$ from Theorem \ref{ConcentrationEinleitung} a concentration inequality, from Theorem \ref{StarkesGesetzVolumen} a Marcinkiewicz-Zygmund-type strong law of large numbers, from Theorem \ref{MomentboundEinleitung} bounds for the moments of all orders, from Theorem \ref{DeviationprobabilityEinleitung} a bound on the relative error in the central limit theorem that we have from Corollary \ref{CLTEinleitung} as well as a moderate deviation principle from Theorem \ref{ModeratedeviationsEinleitung}. Moreover, Theorem \ref{cumulantestimate} below delivers a bound on the cumulants of these random variables. We refrain from presenting all these results formally, since their statements are literally the same as in the theorems mentioned above with the Gaussian polytope $K_\la$ replaced by the randomly rotated and projected simplex ${\rm pr}_d^{N(\la)}(\varrho(\Delta_{N(\la)}))$.

\medbreak

The theorems we have seen in this section (and also those in Section \ref{sec:MainResultsEmpirical} below) are the clear analogues for Gaussian polytopes of the results recently derived in our paper \cite{GroteThäle}, where we considered random polytopes that arise as convex hulls of a homogeneous Poisson point process in the $d$-dimensional unit ball. Moreover, also the principal technique we use, based on sharp bounds for cumulants in conjunction with the large deviation theory from \cite{SaulisBuch}, parallels that in \cite{GroteThäle}. However, we emphasize at this point that besides of these conceptual similarities, the further details and arguments differ considerably and require much more technical effort as well as a number of new ideas compared to \cite{GroteThäle}. This is basically due to the fact that, in contrast to random polytopes in the unit ball, Gaussian polytopes in $\RR^d$ grow unboundedly in all directions. In particular, for any fixed $\la$ there is no centred ball with radius only depending on $\la$ (or any other deterministic set that depends on the parameter $\la$ only) in which a Gaussian polytope is included with probability one. This in turn implies that the scaling transformation we borrow from \cite{CalkaYukich}, which we recall in Section \ref{sec:Preliminaries} below, maps a Gaussian polytope into a random set in the product space $\RR^{d-1}\times\RR$, while the scaling transformation for random polytopes in the unit ball has $\RR^{d-1}\times[0,\infty)$ as its target space, see \cite{CalkaSchreiberYukich}. Here, the upper half-space $\RR^{d-1}\times[0,\infty)$ corresponds to the image of an appropriate centred ball that contains the Gaussian polytope with high probability, while the lower half-space $\RR^{d-1}\times(-\infty,0)$ corresponds to the image of its complement. The probability that the latter contains points from the underlying Poisson point process is small, but if there are such points, they have a significant influence on the geometry of the Gaussian polytopes. While the rescaled geometric functionals satisfy a weak spatial localization property in the upper half-space, such a behaviour is no longer true in the global set-up. This remarkable but unavoidable phenomenon, explained in detail in \cite{CalkaYukich} and briefly recalled in Section \ref{sec:Preliminaries}, causes considerable technical difficulties that were not present in our previous work \cite{GroteThäle} and makes the analysis of probabilistic properties of Gaussian polytopes a demanding task.

\section{Preliminaries}\label{sec:Preliminaries}

We start this section by introducing some notation.  
By $\|\,\cdot\,\|$ we denote the standard Euclidean norm in $\RR^d$. Furthermore, we put $\BB^d := \{x\in \RR^d: \|x\|\le 1 \}$ and  $\SS^{d-1} := \{x\in \RR^d: \|x\| = 1\}$, and define $\kappa_j := \text{vol}_j(\BB^j)$ as the $j$-volume of $\BB^j$ for $j\in\NN$.  
By $\sigma_{d-1}$ we denote the $(d-1)$-dimensional Hausdorff measure on $\SS^{d-1}$ and by $\gamma_d$ the standard Gaussian measure on $\RR^d$.

Let $\BB^d(x,r)$ be the ball centred at $x\in \RR^d$ with radius $r>0$. If we parametrize points in $\RR^d$ by $(v,h)\in\RR^{d-1}\times\RR$ we write $C_{d-1}(v,r)$ for the infinite vertical cylinder $\BB^{d-1}(v,r) \times \RR$ around $x$ with base radius $r>0$.

Finally, if $E$ is some Polish space, $\cB(E)$ and $\cC(E)$ indicate the spaces of bounded measurable and of bounded continuous functions on $E$, respectively. Furthermore, for a Borel set $B\subseteq E$ we write $\cC(E;B)$ for the collection of functions $f\in\cB(E)$ whose set of continuity points includes $B$. We write $\cM(E)$ for the space of finite signed measures on $E$, and for a function $f\in \mathcal{B}(E)$ and a measure $\nu\in \mathcal{M}(E)$ we will use the symbol $\langle f,\nu \rangle$ to abbreviate the integral of $f$ with respect to $\nu$, that is, $\langle f,\nu\rangle:=\int f\,\dint\nu$.

\subsection{The scaling transformation}

Recall that we denote by $K_\la$ a Gaussian polytope that arises as the convex hull of a Poisson point process $\cP_\la$ with intensity measure $\la\g_d$ for some $\la>0$. An important tool on our way to asymptotic results about geometric characteristics of $K_\la$ is a scaling transformation that has been introduced in \cite{CalkaYukich}. It maps the point process $\mathcal{P}_\la$ into the space $\RR^{d-1}\times \RR$. In what follows, we recall the definition of the scaling transformation and those properties that are needed in our proofs.

Let $u_0:= (0,\ldots,0,1) \in \RR^d$ and $T_{u_0}:= T_{u_0}(\SS^{d-1})$ be the tangent space of $\SS^{d-1}$ at the point $u_0$. We identify $T_{u_0}$ with the $(d-1)$-dimensional Euclidean space $\RR^{d-1}$. The exponential map $\exp := \exp_{u_0}: T_{u_0} \rightarrow \SS^{d-1}$ maps a vector $v\in T_{u_0}$ to the point $u\in \SS^{d-1}$ in such a way that $u$ lies at the end of the unique geodesic ray with length $\|v\|$ starting at $u_0$ and having direction $v$. We denote by $\BB_{d-1}(0,\pi) := \{v\in T_{u_0}: \|v\| < \pi\}$ the region of $T_{u_0}$ on which the exponential map is injective and note that $\exp(\BB_{d-1}(0,\pi)) = \SS^{d-1}\setminus\{-u_0\}$. (Following \cite{CalkaYukich}, we prefer to write $\BB_{d-1}(0,r)$ for a centred ball of radius $r>0$ in $T_{u_0}$ instead of $\BB^{d-1}(0,r)$ to prevent confusions.) 

\begin{definition}
Define
\begin{align}\label{Rlambda}
R_\la:= \sqrt{2\,\log \la- \log(2^{d+1}\pi^d\, \log \la)}
\end{align}
and let $\la>0$ be such that $R_\la\geq 1$. Then the mapping
$T_\la: \RR^d \rightarrow \RR^{d-1} \times \RR$ given by 
\begin{align}\label{Transformation}
T_\la(x) := \left(R_\la\, \exp^{-1}\left(\frac{x}{\|x\|}\right), R_\la^2\, \left(1-\frac{\|x\|}{R_\la}\right)\right)\,,\qquad x\in\RR^d\setminus\{0\},
\end{align}
is called scaling transformation and maps $\RR^d\setminus\{0\}$ into the region $W_\la=R_\la\, \BB_{d-1}(0,\pi) \times (-\infty, R_\la^2]\subseteq \RR^{d-1} \times \RR$, see Figure \ref{Transformation.}.
\end{definition}    

\begin{figure}[t] 
	\centering
	\includegraphics[width=\textwidth]{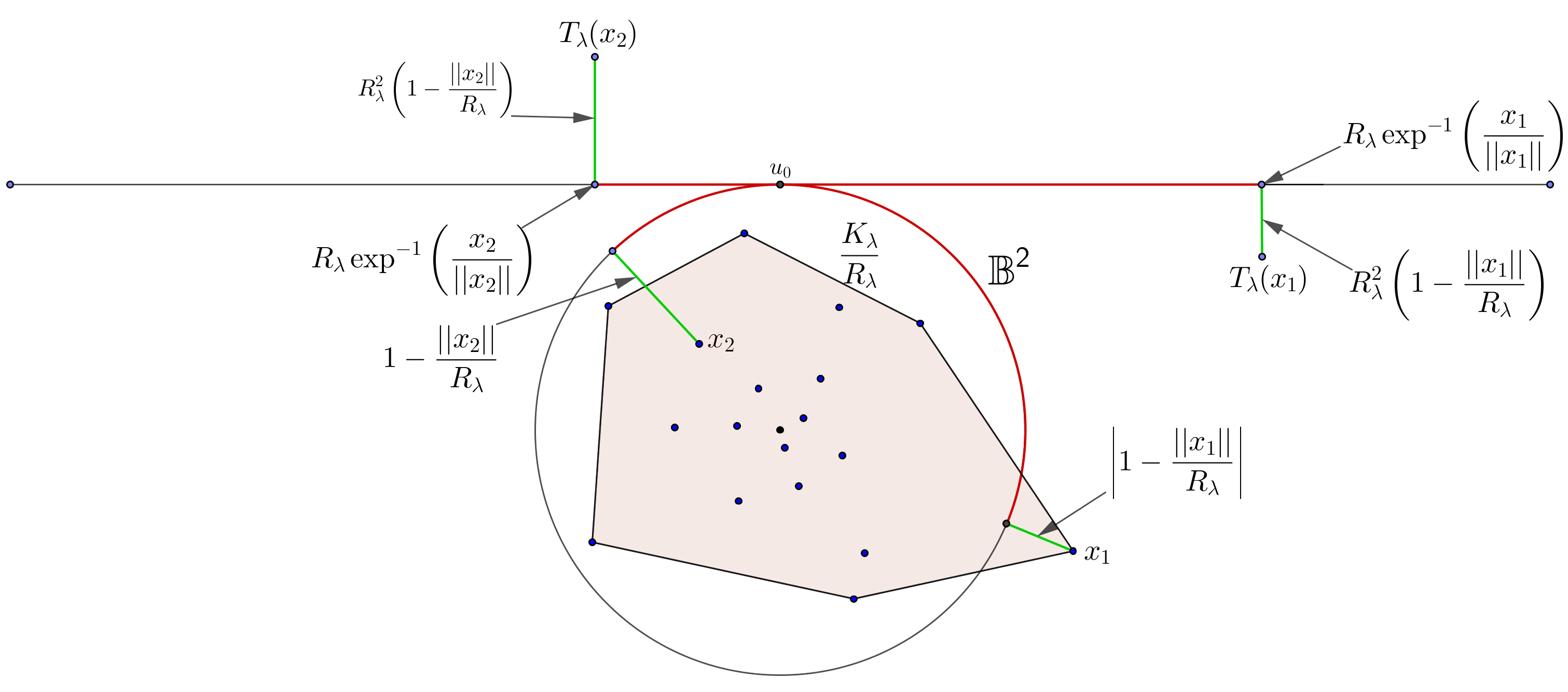}
	\caption{The scaling transformation $T_\lambda$.}
	\label{Transformation.}
\end{figure}

In the rest of this paper we will implicitly assume that $\la$ is such that $R_\la\geq 1$. 
The inverse of the exponential map, $\exp^{-1}$, is well defined on $\SS^{d-1}\setminus \{-u_0\}$ and for convenience we also put $\exp^{-1}(-u_0) := (0,\ldots,0,\pi)$ and $T_\la(0) := (0,R_\la^2)$. So defined, the scaling transformation $T_\la$ describes a bijection between $\RR^d$ and $W_\la$. 

The rescaled point process is given by $\mathcal{P}^{(\la)} := T_{\la}(\mathcal{P}_\la)$, see Figure \ref{Transformation..}.
\begin{figure}[t] 
	\centering
	\includegraphics[width=\textwidth]{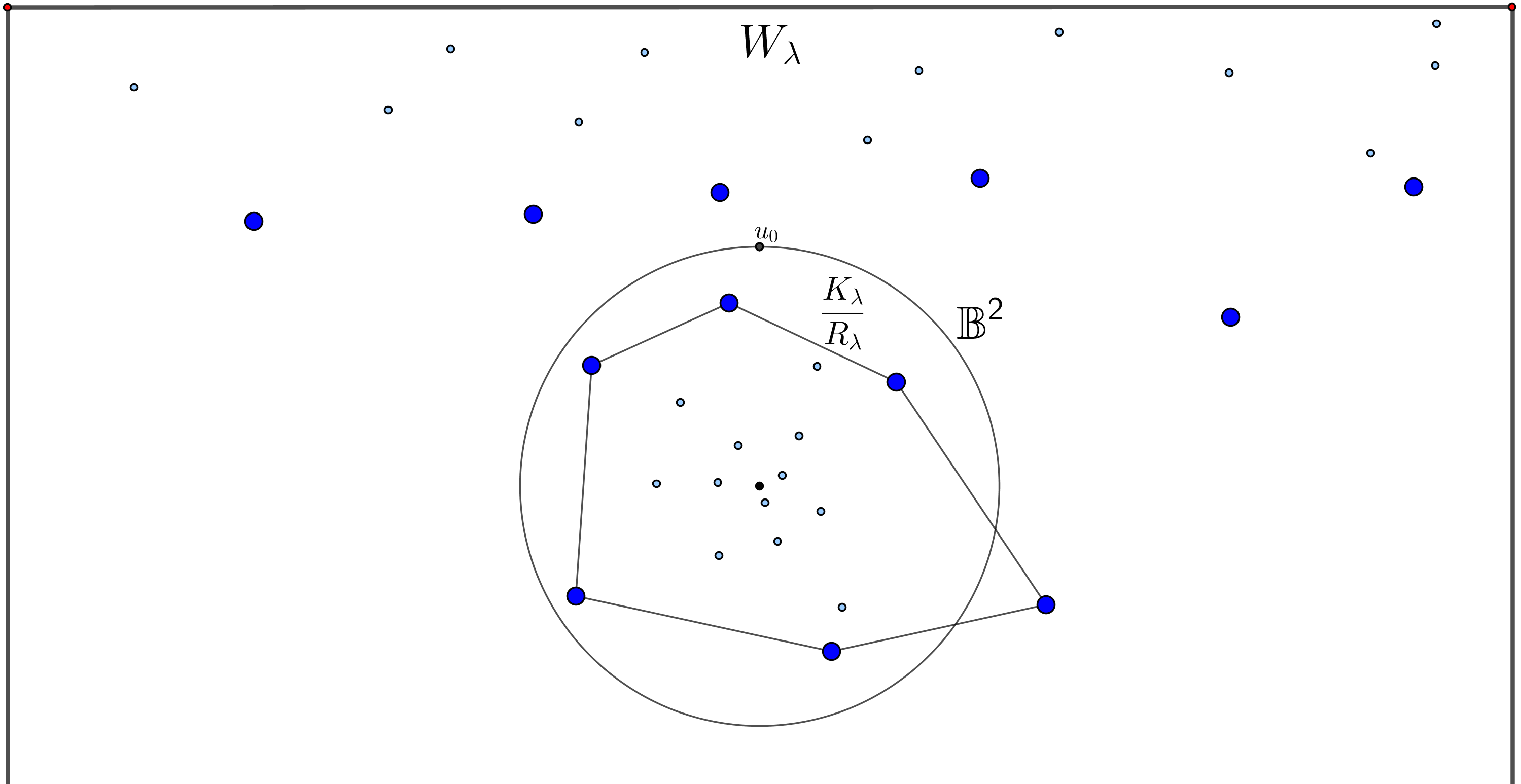}
	\caption{The rescaled Poisson point process $\mathcal{P}^{(\lambda)}$.}
	\label{Transformation..}
\end{figure}
Because of the mapping property of Poisson point processes, the rescaled point process $\mathcal{P}^{(\la)}$ is a Poisson point process on $W_\la$. In particular, Equation (3.18) in \cite{CalkaYukich} shows that the intensity measure of $\cP^{(\la)}$ has density
\begin{align}\label{IntensityP^lambda}
(v,h) \mapsto \frac{\sin^{d-2} (R_\la^{-1}\|v\|)}{\|R_\la^{-1}v\|^{d-2}}\, \frac{\sqrt{2\log \la}}{R_\la}\, \exp\left(h-\frac{h^2}{2R_\la^2}\right) \left(1-\frac{h}{R_\la^2}\right)^{d-1}\,{\bf 1}((v,h)\in W_\la)
\end{align}  
with respect to the Lebesgue measure on $\RR^{d-1}\times\RR$. In other words, \eqref{IntensityP^lambda} is the density of the image measure of $\la\gamma_d$ under the scaling transformation $T_\la$. It is readily seen that, as $\la\to\infty$, this density converges to 
\begin{align}\label{P}
(v,h) \mapsto e^h\,,\qquad  (v,h) \in \RR^{d-1} \times \RR\,,
\end{align}
which in view of Proposition 2.1 and Remark 3.2 (iv) in \cite{DST} (see also Lemma 3.2 in \cite{CalkaYukich}) implies that $\cP^{(\la)}$ converges in distribution to a Poisson point process $\mathcal{P}$ on $\RR^{d-1} \times \RR$ whose intensity measure has density as in \eqref{P} with respect to the Lebesgue measure on $\RR^{d-1} \times \RR$. Similarly, it follows from \cite[Equation (3.19)]{CalkaYukich} that the image measure of $R_\la$ times the Lebesgue measure on $\RR^d$ under the scaling transformation $T_\la$ has density
\begin{align}\label{eq:DensityLebesgueUnderTrafo}
(v,h)\mapsto{\sin^{d-2}(R_\la^{-1}\|v\|)\over\|R_\la^{-1}v\|^{d-2}}\,\left(1-\frac{h}{R_\la^2}\right)^{d-1}\,{\bf 1}((v,h)\in W_\la)
\end{align}
with respect to the Lebesgue measure on the product space $\RR^{d-1}\times\RR$.

\begin{remark}\label{rem:WalVonRl}
Let us briefly comment on the choice of the critical radius $R_\la$, which might look a bit inconvenient at first sight. In the proof of formula \eqref{IntensityP^lambda}, the definition of $R_\la$ is crucial to ensure the equality 
\begin{align}\label{WahlvonRlambda}
\la\, \phi\left(u\, R_\la\left(1-\frac{h}{R_\la^2}\right)\right) = \sqrt{2 \log \la}\, \exp \left(h-\frac{h^2}{2R_\la^2}\right), \qquad u\in \SS^{d-1},
\end{align}
which in turn is needed to show the convergence in distribution of $\mathcal{P}^{(\la)}$ to $\mathcal{P}$ discussed above and which is also used in the proof of our cumulant bound in Section 6 below. Indeed, the definition of $R_\la$ yields that
\begin{align*}
\exp\left(-\frac{R_\la^2}{2}\right) &= \exp\left(-\frac{1}{2}\Big(2\log \la- \log(2^{d+1}\pi^d \log \la)\Big)\right)=\frac{1}{\la} (2\pi)^{d/2} \sqrt{2\log \la}
\end{align*}
and we thus obtain for $u\in\SSd$, by using \eqref{DefinitionPhi}, that
\begin{align*}
\phi\left(u\, R_\la\left(1-\frac{h}{R_\la^2}\right)\right) &= (2\pi)^{-d/2}\, \exp\left(-\frac{1}{2}\, \left\|u\, R_\la\left(1- \frac{h}{R_\la^2}\right)\right\|^2\right)\\
&= (2\pi)^{-d/2}\, \exp\left(-\frac{1}{2}\, R_\la^2\, \left(1-\frac{2h}{R_\la^2} + \frac{h^2}{R_\la^4}\right)\right)\\
&= (2\pi)^{-d/2}\, \exp\left(h-\frac{h^2}{2R_\la^2}\right)\, \exp\left(-\frac{R_\la^2}{2}\right)\\
&= \frac{\sqrt{2\log \la}}{\la}\, \exp\left(h-\frac{h^2}{2R_\la^2}\right)\,, 
\end{align*}
which proves \eqref{WahlvonRlambda}.
\end{remark}

\subsection{The volume and $j$-face functionals}

This section starts with the introduction of the key geometric functionals of $K_\la$ we consider in this paper.  Afterwards, we link these characteristics of $K_\la$ with those of germ-grain processes $\Psi^{(\la)}$ and $\Phi^{(\la)}$ in $\RR^{d-1} \times \RR$ via the scaling transformation $T_\la$ introduced in the previous section.

Given a finite point set $\cX \subseteq \RR^d$, let $[\cX]$ be the convex hull of $\cX$ and $x$ be a vertex of $[\cX]$. Then the symbol $\cF_{j}(x,\cX)$ is used for the collection of all $j$-dimensional faces of $[\cX]$ that contain $x$, $j\in \{0,\ldots,d-1\}$. In particular, $\cF_0(x,\cX)=\{x\}$. We also write $|\cF_{j}(x,\cX)|$ for the cardinality of $\cF_j(x,\cX)$ and define the cone induced by $\cF_{d-1}(x,\cX)$ as $\text{cone}(x,\cX) := \{ry: r\geq 0, y\in \cF_{d-1}(x,\cX)\}$.   

\begin{definition} We define the defect volume functional with respect to the ball $\BB^d(0,R_\la)$ by putting
\begin{align*}
\xi_V(x,\cP_\la):={1\over d}\, R_\la\, \big(\vol(\cone(x,\cP_\la)\cap\BBd(0,R_\la)) - \vol(\cone(x,\cP_\la)\cap K_\la)\big)\,,
\end{align*}
if $x$ is a vertex of $K_\la$ and zero for all other points of $\cP_\la$. Moreover, for $j\in \{0,\ldots,d-1\}$ we define the $j$-face functional of the Gaussian polytope $K_\la$ by
$$
\xi_{f_j}(x,\cP_\la):={1\over j+1} |\mathcal{F}_j(x,\cP_\la)|\,,
$$
if $x$ is a vertex of $K_\la$ and again $\xi_{f_j}(x,\cP_\la):=0$ otherwise, as above. We shall write $\Xi:=\{\xi_V,\xi_{f_0},\ldots,\xi_{f_{d-1}}\}$ for the collection of these geometric functionals and use for $\xi\in \Xi$ the abbreviation
\begin{equation}\label{eq:DefHXi}
H_\la^\xi := \sum_{x\in\cP_\la}\xi(x,\cP_\la)\,.
\end{equation}
\end{definition}

With these definitions it follows that the total number of $j$-faces of $K_\la$, $j\in \{0,\ldots,d-1\}$, almost surely satisfies $f_j(K_\la) =H^{\xi_{f_j}}_\la$,
while the total defect volume of $K_\la$ with respect to the ball $\BBd(0,R_\la)$ fulfills
\begin{align}\label{überzähligesRlambda}
\vol(\BBd(0,R_\la)) - \vol(K_\la)=R_\la^{-1}\, H^{\xi_V}_\la
\end{align}
almost surely, conditioned on the event that the origin is an interior point of $K_\la$. We notice that this event occurs with probability at least $1-e^{-c\lambda}$ for some constant $c\in(0,\infty)$ only depending on $d$. To keep our presentation short in all computations concerning the functional $\xi_V$ that are carried out in Sections \ref{sec:MomentEstimates} and \ref{sec:CumulantProof}, we implicitly condition on this event. In fact, this causes -- up to constants -- no changes in our results, since conditioning on the complementary event only leads to terms that are negligible for sufficiently large $\la$. Also implicitly this convention has already been used in \cite{CalkaSchreiberYukich,CalkaYukich,GroteThäle}.

\begin{remark}
It is known from the work of Geffroy \cite{Geffroy} that the Hausdorff distance between $K_{\la_k}$ and $\BB^d(0,\sqrt{2\, \log \la_k})$ converges to zero almost surely, along all suitable subsequences $\la_k$ tending to infinity, as $k\to\infty$. Furthermore, B\'ar\'any and Vu \cite{BaranyVu} show that for sufficiently large $\la$ the vertices of $K_\la$ concentrate around the boundary of $\BB^d(0,R_\la)$ with overwhelming probability. It is thus natural to choose $\BB^d(0,R_\la)$ as a reference body for the Gaussian polytope $K_\la$ and to compare the volume of $K_\la$ with that of $\BB^d(0,R_\la)$. 
\end{remark}
\medskip

Next, we define for every $\xi \in \Xi$ and sufficiently large $\la$ the rescaled functional $\xi^{(\la)}$ by
\begin{align*}
\xi^{(\la)} (w,\cP^{(\la)}) := \xi (T_\la^{-1}(w), T_\la^{-1}(\cP^{(\la)}))\,,\qquad w\in W_\la\,,
\end{align*} 
and denote by $\Xi^{(\la)} := \{\xi_V^{(\la)}, \xi_{f_0}^{(\la)}, \ldots, \xi_{f_{d-1}}^{(\la)}\}$ the family of rescaled geometric functionals. Here and in the rest of this paper we adopt the following notational convention. If $w\in W_\la$ is not a point of the rescaled point process $\cP^{(\la)}$, we understand $\xi^{(\la)} (w,\cP^{(\la)})$ as $\xi^{(\la)} (w,\cP^{(\la)} \cup \{w\})$ and similarly also $\xi(x,\cP_\la)$ as $\xi(x,\cP_\la\cup \{x\})$ for $\xi\in\Xi$ and $x\in\RR^d$ with $x\notin \cP_\la$.

These functionals are tightly connected to geometric properties of special germ-grain processes that have been introduced in \cite{CalkaYukich}. Namely, let $x\in \RR^d$ and $(v,h):=T_\la(x)$. Then the discussion at the beginning of Section 3.1 in \cite{CalkaYukich} shows at first that $T_\la$ transforms the ball $\BBd(x/2, \|x\|/2)$ into the set
\begin{align*}
[\Pi^{\uparrow}(v,h)]^{(\lambda)} := \left\{(v',h')\in W_\la: h'\geq R_\la^2(1-\cos(d_\lambda(v',v))) + h\, \cos(d_\lambda(v',v))\right\}\,,
\end{align*}
where $d_\lambda(v',v) := d_{\mathbb{S}^{d-1}}(\exp(R_\la^{-1} v'), \exp(R_\la^{-1}v))$ is the geodesic distance between the images of the rescaled points $v'$ and $v$ under the exponential map. The results in \cite{CalkaYukich} then imply that the germ-grain process 
\begin{align*}
\Psi^{(\lambda)}:=  \bigcup\limits_{w\in \mathcal{P}^{(\lambda)}} [\Pi^{\uparrow}(w)]^{(\lambda)},
\end{align*}
see Figure \ref{GermGrain1}, has deep connections to the geometry of the Gaussian polytopes $K_\la$.
\begin{figure}[t] 
	\centering
	\includegraphics[width=\textwidth]{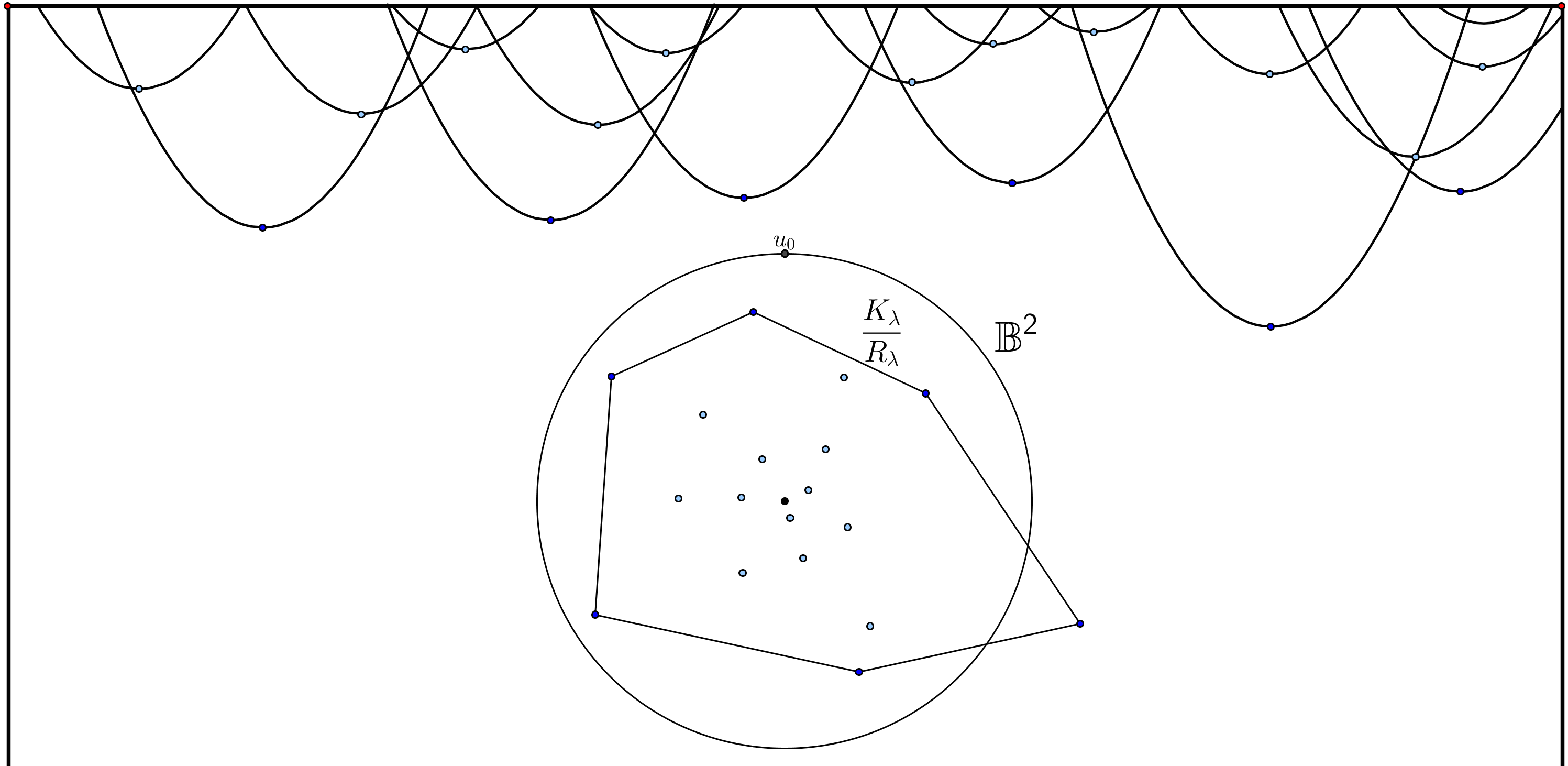}
	\caption{The germ-grain model $\Psi^{(\lambda)}$.}
	\label{GermGrain1}
\end{figure}
For example, it turns out that, for sufficiently large $\la$, the property that a point $x\in\RR^d$ is a vertex of $K_\la$ is equivalent to the statement that $T_\la(x)$ is an extreme point of $\Psi^{(\la)}$, whose collection is denoted by $\ext(\Psi^{(\la)})$ in what follows. The latter property means that $[\Pi^{\uparrow}(T_\la(x))]^{(\lambda)}$ is not covered by other grains from $\Psi^{(\la)}$. This observation has been used extensively in \cite{CalkaYukich} and also our results exploit this fact.

In \cite{CalkaYukich} a germ-grain model $\Phi^{(\la)}$ `dual' to $\Psi^{(\la)}$ has also been introduced. Formally, $\Phi^{(\la)}$ is defined as
$$
\Phi^{(\la)} := \bigcup_{w\in\RR^{d-1}\times\RR\atop \cP^{(\la)}\cap{\rm int}(\Pi^{\downarrow}\oplus w)=\emptyset}(\Pi^{\downarrow}\oplus w)\,,
$$
where $\Pi^{\downarrow}=\{(v,h)\in\RR^{d-1}\times\RR:h\leq -\|v\|^2/2\}$ is the unit downward paraboloid, ${\rm int}(\,\cdot\,)$ denotes the interior of the argument set and $\oplus$ is the usual Minkowski sum. Clearly, the boundary $\partial\Phi^{(\la)}$ of $\Phi^{(\la)}$ is build from piecewise parabolic facets that are glued together at the extreme points of $\Psi^{(\la)}$, see \cite{CalkaYukich}.

\subsection{Properties of the rescaled functionals and the germ-grain models}

The main goal in this section is to give an overview of such properties of the rescaled functionals $\xi^{(\la)}\in\Xi^{(\la)}$ and the germ-grain model $\Psi^{(\la)}$ and its dual $\Phi^{(\la)}$ that will be used in our proofs. These features have been shown in \cite{CalkaYukich} in order to derive expectation and variance asymptotics for the geometric characteristics of $K_\la$ and are summarized in the following lemma. Before we are able formulate it, we need some further notions and notation. 

The collection $\Xi^{(\la)}$ consists of spatial correlated functionals defined on $W_\la$. The purpose of the localization theory developed in the context of random polytopes in \cite{CalkaSchreiberYukich,CalkaYukich,SchreiberYukich} is to quantify these spatial dependencies. For a rescaled functional $\xi^{(\la)} \in  \Xi^{(\la)}$ and $w=(v,h) \in W_\la$ put 
\begin{align*}
\xi_{[r]}^{(\lambda)}(w,\mathcal{P}^{(\la)}) :=  \xi^{(\lambda)}(w,\mathcal{P}^{(\lambda)}\cap C_{d-1}(v,r))\,.
\end{align*} 
In what follows, we shall refer to $h$ as the height (coordinate) of the point $w$.

\begin{definition}
One says that a random variable $L= L(\xi^{(\lambda)},w)$ that is allowed to depend on the rescaled functional $\xi^{(\la)}$ and the point $w$ only is a localization radius for $\xi^{(\lambda)}$ at $w$ if, almost surely,
\begin{align*}
\xi^{(\lambda)}(w,\mathcal{P}^{(\lambda)}) = \xi_{[L]}^{(\lambda)}(w,\mathcal{P}^{(\lambda)})\qquad \text{and}\qquad \xi_{[L]}^{(\lambda)}(w,\mathcal{P}^{(\lambda)}) = \xi_{[s]}^{(\lambda)}(w,\mathcal{P}^{(\lambda)})
\end{align*}
for all $s\geq L$. In the following, it is convenient to denote also the minimum of all such $L$ by the same symbol, and call this random variable \textit{the} radius of localization for $\xi^{(\lambda)}$ at $w$. 
\end{definition}

Let $H:= H(w, \mathcal{P}^{(\la)})$ be the maximal height coordinate of an apex of a downward paraboloid 
which contains a parabolic face in the boundary of $\Phi^{(\la)}$ if $w$ belongs to the extreme points of the germ-grain process $\Psi^{(\la)}$ and zero otherwise. If $w\notin \cP^{(\lambda)}$, we shall use the shorthand notation $w\in \ext(\Psi^{(\lambda)})$ for the event that
\begin{align*}
w\in \ext\bigg(\bigcup\limits_{z\in \mathcal{P}^{(\lambda)} \cup \{w\}} [\Pi^{\uparrow}(z)]^{(\lambda)}\bigg).
\end{align*}
Moreover, let us write $a\lor b$ for the maximum of two real numbers $a,b\in\RR$ and $\|\,\cdot\,\|_\infty$ for the sup-norm of the argument (function).
 
\begin{lem}\label{Eigenschaften}
Let $\xi \in \Xi$. Then for all $w=(v,h)\in W_\la$ and sufficiently large $\la$ there exist constants $c_1,\ldots,c_{10} \in (0,\infty)$ only depending on $\xi$ and on $d$ with the following properties.
\begin{description}
\item[(i)] The localization radius $L = L(\xi^{(\lambda)},w)$ satisfies
\begin{align}\label{Lokalisierung1}
\PP(L \geq t) \le c_1\, \exp\left(-\frac{t^2}{c_2}\right)\,,\qquad t\geq|h|\,,
\end{align}
as well as the weaker estimate 
\begin{align}\label{Lokalisierung2}
\PP(L \geq t) \le c_3\, \exp\left(-\frac{t}{c_4}\right)\,,\qquad t\geq|h|\,.
\end{align}
\item[(ii)] The probability for the event that a point belongs to the extreme points of $\Psi^{(\la)}$ decays exponentially with its height coordinate. In particular, one has that
\begin{align}\label{WahrscheinlichkeitExtrempunkt}
\PP (w\in \ext(\Psi^{(\la)})) \le c_5\, \exp\left(-\frac{e^{h\lor 0}}{c_6}\right)\,.
\end{align}
\item[(iii)] For all $M \in (0,\infty)$ one has that
\begin{align}\label{WahrscheinlichkeitHöhe}
\PP (\|\partial \Psi^{(\lambda)}(\mathcal{P}^{(\lambda)}) \cap C_{d-1}(v,M)\|_\infty \geq  t) \le c_7\, M^{2(d-1)}\exp\left(-\frac{t}{c_8}\right)\,,\qquad t>0\,.
\end{align}
\item[(iv)] It holds that
\begin{align}\label{UngleichungH}
\PP(H\ge t) \le c_9\, \exp\left(-\frac{e^t}{c_{10}}\right)\,,\qquad t\geq h\lor 0\,.
\end{align}
\end{description}
\end{lem}

Let us briefly comment on the statements of the previous lemma. At first, we emphasize that the tail estimates \eqref{Lokalisierung1} and \eqref{Lokalisierung2} are valid only for arguments $t\geq|h|$. Next, also the probability for a point $w=(v,h)\in W_\la$ to belong to the extreme points of $\Psi^{(\la)}$ falls into two cases. Namely, if the height $h$ exceeds $0$, then $\PP (w\in \ext(\Psi^{(\la)}))$ decays super-exponentially fast, while if $h\leq 0$ one only has an estimate independently of $h$ (which is in some sense trivial). Similarly, also the probability for the event that $H\geq t$ can only be estimated in a meaningful way if $t$ or $h$ are not too small. This underlines the effect already discussed at the end of Section \ref{subsec:MainResults} that the spatial localization property of the rescaled geometric functionals $\xi^{(\la)}\in\Xi^{(\la)}$ we consider can effectively only be handled in the upper half-space $\RR^{d-1}\times[0,\infty)$, while in the lower half-space no such spatial localization is available. This phenomenon is new compared to the theory of random polytopes in the unit ball developed in \cite{CalkaSchreiberYukich,GroteThäle,SchreiberYukich} and is in fact the leading cause for the technical complications that arise in the context of Gaussian polytopes.

\subsection{Empirical measures and their cumulants}

It is crucial in the proofs of our main results to have very precise control on the growth of the cumulants of the geometric characteristics $H^\xi_\la$ given by \eqref{eq:DefHXi}. For that purpose, it turns out to be more convenient to work with the measure-valued versions of $H^\xi_\la$ and for this reason we define for all $\la>0$ with $R_\la\geq 1$ the empirical random measures
\begin{align}\label{empiricalmeasure}
\mu_\la^{\xi} :=  \sum\limits_{x\in \mathcal{P}_\la} \xi(x,\mathcal{P}_\la)\delta_x =  \sum\limits_{w\in \mathcal{P}^{(\la)}} \xi^{(\la)}(w,\mathcal{P}^{(\la)})\delta_{T_\la^{-1}(w)} \,,\qquad\xi\in\Xi\,,
\end{align}
where $\delta_x$ is the Dirac measure at $x$. The corresponding centred versions are given by $\bar{\mu}_\lambda^\xi:= \mu_\lambda^\xi - \EE[\mu_\lambda^\xi]$. Moreover, for a function $f\in \cB(\RR^d)$ and $r\in\RR\setminus\{0\}$ define $f_r(x) := f(x/r)$. 
We recall from Theorem 2.1 in \cite{CalkaYukich} that if $f\in\cC(\RR^d;\SSd)$, $\sigma_\lambda^\xi(f_{R_\la}):= (\text{var}[\langle f_{R_\la},\mu_\lambda^\xi\rangle])^{1/2}$ fulfils for all  $\xi\in\Xi$ and sufficiently large $\la$ the estimate 
\begin{align}\label{SchrankefürVarianz}
\sigma_\lambda^\xi(f_{R_\la}) \geq c\, \langle f^2, \sigma_{d-1}\rangle^{\frac{1}{2}} (\log \la)^{\frac{d-1}{4}}
\end{align}
with a constant $c\in (0,\infty)$ that depends only on the space dimension $d$ and the functional $\xi$.

The method of expanding the so-called cumulant measures associated with $\mu_\la^{\xi}$ in terms of cluster measures has been developed and successfully applied in \cite{BaryshnikovYukich} in the context of the central limit theorem. We use a refined version from \cite{EichelsbacherSchreiberRaic,GroteThäle} to deduce sharp bounds for the cumulants of $\langle f_{R_\la},\mu_\la^\xi \rangle$. To present the main formulas, let us write $M_\la^k$ for the $k$th order moment measure of $\mu_\la^k$, see \cite{BaryshnikovYukich,EichelsbacherSchreiberRaic,GroteThäle} for a formal definition. (Here and in what follows, we think of $\xi\in\Xi$ being fixed and hence suppress the dependence on $\xi$ in our notation.) To appropriately handle the moment measures, we define for $g\in\cB(\RR^d)$ the singular differential $\bar{\dint }[g]$ by the relation
\begin{align}\label{singulardifferential}
\int\limits_{(\mathbb{R}^d)^k} F(x_1,\ldots,x_k)\, \bar{\dint}[g](x_1,\ldots,x_k) := \int\limits_{\mathbb{R}^d} F(y,\ldots,y)\ g(y) \,\dint y\,,\qquad F\in\cB((\RR^d)^k)\,,
\end{align}  
and put, for $\text{\bf x} := (x_1,\ldots,x_k)\in(\RR^d)^k$,
\begin{align}\label{singulardifferential2}
\tilde{\dint}[g](\text{\bf x}) := \sum\limits_{L_1,\ldots,L_p\, \preceq\setk} \bar{\dint}[g](\text{\bf x}_{L_1})\ldots \bar{\dint}[g](\text{\bf x}_{L_p})\,,
\end{align}
where $\text{\bf x}_{L_i} := (x_\ell)_{\ell\in L_i}$ for $i\in \{1,\ldots,p \}$ and the sum runs over all unordered partitions $\{L_1,\ldots,L_p\}$ of $\{1,\ldots,k\}=:\setk$. This is indicated by the symbol $L_1,\ldots,L_p\, \preceq\setk$ in what follows. From Proposition 3.1 in \cite{EichelsbacherSchreiberRaic} one deduces that the density of $M_\la^k$ with respect to $\tilde{\dint}[\lambda\phi]$, with $\phi$ being the Gaussian density, equals
\begin{align*}
m_\la(\text{\bf x}) = m_\lambda(x_1,\ldots,x_k):= \EE\left[\prod\limits_{i=1}^{k} \xi^{(\lambda)}\left(T_\la(x_i),\mathcal{P}^{(\lambda)} \cup\bigcup_{i=1}^k\{T_\la(x_i)\}\right)\right]\,.
\end{align*} 

\begin{definition}
The $k$th (signed) cumulant measure $c_\la^k$ associated with $\mu_\la^\xi$ is defined as 
\begin{align}\label{cumulantmeasure}
c_\lambda^k := \sum\limits_{L_1,\ldots,L_p\, \preceq\setk} (-1)^{p-1} (p-1)!\ M_\lambda^{|L_1|}\otimes \ldots \otimes M_\lambda^{|L_p|}\,,
\end{align}
where $M_\lambda^{|L_1|}\otimes \ldots \otimes M_\lambda^{|L_p|}$ denotes the product measure of $M_\lambda^{|L_1|},\ldots,M_\lambda^{|L_p|}$.
\end{definition}

The cumulant measures can alternatively be expressed as a sum of so-called cluster measures. For non-empty and disjoint sets $S,T \subseteq \NN$ the cluster measure $U_\la^{S,T}$ on $(\mathbb{R}^d)^{|S|}\times (\mathbb{R}^d)^{|T|}$ is defined by
\begin{align*}
 U_\la^{S,T}(A\times B) := M_\la^{|S\cup T|}(A\times B) - M_\la^{|S|}(A)\, M_\la^{|T|}(B)
\end{align*}
for Borel sets $A\subseteq (\mathbb{R}^d)^{|S|}$ and $B\subseteq (\mathbb{R}^d)^{|T|}$. Loosely speaking, the cluster measures capture the spatial correlations of the re-scaled functionals $\xi^{(\la)}$ and their measure-valued counterparts. To proceed, define for $\text{\bf x} = (x_1,\ldots,x_k)\in (\mathbb{R}^d)^k$ and their rescaled images $(v_i,h_i):=T_\la(x_i)$, $i=1,\ldots,k$, the quantity
 \begin{align}
 \begin{split}\label{delta}
 \delta(\text{\bf x}) &:= \delta\left(v_1,\ldots,v_k\right) := \max\left\{d(\bv_S,\bv_T): \{S,T\} \preceq\setk \right\}\,,
 \end{split}
 \end{align}
 where $\bv_S=(v_s)_{s\in S}$ and $\bv_T=(v_t)_{t\in S}$, and $d(\bv_S,\bv_T) := \min_{s\in S, t\in T} \|v_s-v_t\|$ is the separation for the partition $\{S,T\}$ of $\setk$. Moreover, let
$\Delta := \{(x,\ldots,x) \in (\RR^d)^k: x\in \mathbb{R}^d\}$ be the diagonal in $(\mathbb{R}^d)^k$. Similar to what has been explained in \cite{BaryshnikovYukich,EichelsbacherSchreiberRaic,GroteThäle} one can decompose the space $(\mathbb{R}^d)^k \setminus \Delta$ into a disjoint union of sets $\delta(\{S,T\})$ with non-trivial partitions $\{S,T\} \preceq\setk$ such that $\text{\bf x}\in \delta(\{S,T\})$ implies that $d(\bv_S,\bv_T) = \delta(\text{\bf x})$. This leads to the following cluster measure representation, in which we write $f^k$ for the $k$th tensor power of a function $f$, cf.\ \cite{BaryshnikovYukich,EichelsbacherSchreiberRaic,GroteThäle} for further details and explanations.

\begin{lem}\label{lem:CumulantExpressionCluster}
Fix $k\in\{2,3,\ldots\}$ and let $f\in\cB(\RR^d)$. Then,
\begin{equation}\label{ZerlegungKumulanten}
\begin{split}
\langle f_{R_\la}^k,c_\lambda^k \rangle  = \int\limits_\Delta f_{R_\la}^k\, \dint c_\lambda^k \ + \sum_{S,T\, \preceq\setk}\,  \int\limits_{\delta(\{S,T\})} & \sum\limits_{S^{'},T^{'},K_1,\ldots,K_s\, \preceq\setk} a_{S^{'},T^{'},K_1,\ldots,K_s}\\
&\qquad\times f_{R_\la}^k\, \dint(U_\lambda^{S^{'},T^{'}} \otimes M_\lambda^{|K_1|}\otimes \cdots \otimes M_\lambda^{|K_s|})\,,
\end{split}
\end{equation} 	
where in every summand $S^{'},T^{'},K_1,\ldots,K_s$ is a partition of $\setk$ with $S^{'}\subseteq S$, $T^{'}\subseteq T$ and the constants $a_{S^{'},T^{'},K_1,\ldots,K_s}$ satisfy the estimate
\begin{align}\label{AbschätzungSummanden}
\sum\limits_{S^{'},T^{'},K_1,\ldots,K_s\, \preceq\setk} |a_{S^{'},T^{'},K_1,\ldots,K_s}| \leq 2^k\, k!\,.
\end{align} 
\end{lem}

\begin{remark}
The proof in \cite{EichelsbacherSchreiberRaic} shows that the bound \eqref{AbschätzungSummanden} cannot be improved.
\end{remark}

\section{Main results for empirical measures}\label{sec:MainResultsEmpirical}

In this section we present a series of results for the empirical measures introduced at \eqref{empiricalmeasure}. 
The advantage of working with empirical measures instead of just their total masses is that they allow to capture also the spatial profile of the geometric functionals we consider.

Let us briefly recall the set-up. By $\cP_\la$ we denote a Poisson point process in $\RR^d$ whose intensity measure is a multiple $\la>0$ of the standard Gaussian measure $\gamma_d$. The Gaussian polytope $K_\la$ is the random convex hull generated by $\cP_{\la}$. The class of key geometric functionals associated with $K_\la$ is abbreviated by the symbol $\Xi$ and for $\xi\in\Xi$ we let $\mu_\la^\xi$ be the corresponding empirical measure defined by \eqref{empiricalmeasure}. To present our results in a unified way, let us define for $\xi\in\Xi$ the weights 
\begin{align}\label{eq:WEIGHTS}
u[\xi]:= \begin{cases}
0 &: \xi = \xi_V\\
 j  &: \xi = \xi_{f_j}\,,
\end{cases} \quad
v[\xi]:= \begin{cases}
1 &: \xi = \xi_V\\
j  &: \xi = \xi_{f_j}
\end{cases} \quad \text{and} \quad
w[\xi]:= \begin{cases}
2 &: \xi = \xi_V\\
j  &: \xi = \xi_{f_j}\,,
\end{cases}
\end{align}
 where $j\in \{0,\ldots,d-1\}$. Moreover, extending the definition of the weights $z[f_j]$ from the introduction, we put
\begin{align}\label{zzz}
z[\xi]:= \begin{cases}
d &: \xi = \xi_{f_0}\\
0 &: \xi \in \{\xi_{f_1},\ldots \xi_{f_{d-1}}, \xi_{V}\}\,.
\end{cases} 
\end{align}
We start with the following concentration bound.

\begin{thm}[Concentration inequality]\label{thm:ConcentrationEmpMeas}
Let $\xi \in \Xi$ and $f\in\cC(\RR^d;\SSd)$ with $\langle f^2, \sigma_{d-1}\rangle > 0$. Then, for all $y\geq 0$ and sufficiently large $\la$,
\begin{align*}
\PP\big(|\langle f_{R_\la},\bar{\mu}_\lambda^\xi\rangle| \geq y\, \sigma_\lambda^\xi(f_{R_\la})\big) & \leq 2 \exp\bigg(- \frac{1}{4} \min\bigg\{\frac{y^2}{2^{3dv[\xi] + u[\xi] + 5 + z[\xi]}}, \\
&\qquad\qquad c\, (\log \la)^{\frac{d-1}{4(3dv[\xi] + u[\xi] + 5 + z[\xi])}}\, y^{\frac{1}{3dv[\xi] + u[\xi] + 5 + z[\xi]}}\bigg\} \bigg)
\end{align*}
with a constant $c \in (0,\infty)$, that only depends on $d$, $\xi$ and $f$. 
\end{thm}

The next result is a generalization of Theorem \ref{DeviationprobabilityEinleitung} in the introduction and assesses the relative error in the central limit theorem on a logarithmic scale. We remark that this is a simplified version of a non-logarithmic estimate in terms of the so-called Cram\'er-Petrov series for which we refer to \cite{SaulisBuch}. For clarity and to keep the presentation more transparent, we have decided to use the simplified version that we took from \cite[Corollary 3.2]{EichelsbacherSchreiberRaic}.

\begin{thm}[Bounds on the relative error in the central limit theorem]\label{thm:DeviationsEmpMeas}
Let $\xi \in \Xi$ and $f\in\cC(\RR^d;\SSd)$ with $\langle f^2, \sigma_{d-1}\rangle>0$. Then, for all $y$ with $0\leq y\leq c_1\, (\log \lambda)^{\frac{d-1}{4(2(3dv[\xi] + u[\xi] + z[\xi]) + 9)}}$ and sufficiently large $\la$ one has that 
\begin{align*}
\left|\log\ \frac{\PP(\langle f_{R_\la},\bar{\mu}_\lambda^\xi\rangle \geq y\, \sigma_\lambda^\xi(f_{R_\la}))}{1 - \Phi(y)} \right| &\leq c_2\, (1 + y^3)\, (\log \la)^{-\frac{d-1}{4(2(3dv[\xi] + u[\xi] + z[\xi]) + 9)}}\quad\text{and}\\
\left|\log\ \frac{\PP(\langle f_{R_\la},\bar{\mu}_\lambda^\xi\rangle \leq -y\, \sigma_\lambda^\xi(f_{R_\la}))}{\Phi(-y)} \right| &\leq c_2\, (1 + y^3)\, (\log \la)^{-\frac{d-1}{4(2(3dv[\xi] + u[\xi] + z[\xi]) + 9)}}
\end{align*}
with constants $c_1,c_2\in (0,\infty)$ only depending on $d$, $\xi$ and $f$.
\end{thm}

For the sake of completeness we also include the following central limit theorem that is available from our technique and, as anticipated above, is closely related to the previous theorem. However, we point out that the rate of convergence we obtain is weaker than that derived in \cite{BaranyVu}. On the other hand, our result is more general since we consider integrals with respect to the empirical measures of general functions $f\in\cC(\RR;\SSd)$, while in \cite{BaranyVu} only constant functions were investigated.

\begin{thm}[Central limit theorem with Berry-Esseen bound]\label{CLTempiricalmeasure}
Let $\xi\in \Xi$ and $f\in\cC(\RR^d;\SSd)$ with $\langle f^2, \sigma_{d-1}\rangle>0$. Then, for sufficiently large $\la$,
\begin{align}\label{Berry}
\sup\limits_{y\in \mathbb{R}} \left|\PP \left(\frac{\langle f_{R_\la}, \bar{\mu}_\lambda^\xi\rangle}{\sigma_\lambda^\xi(f_{R_\la})} \leq y\right) - \Phi(y) \right| \leq c\, (\log \la)^{-\frac{d-1}{4(2(3dv[\xi] + u[\xi] + z[\xi]) + 9)}} \,,
\end{align}
where $c\in (0,\infty)$ is a constant that only depends on $d$, $\xi$ and $f$. In particular, as $\lambda \rightarrow \infty$, the sequence $$\left(\frac{\langle f_{R_\la}, \bar{\mu}_\lambda^\xi\rangle}{\sigma_\lambda^\xi(f_{R_\la})}\right)_{\lambda > 0}$$ converges in distribution to a standard Gaussian random variable.
\end{thm}

Now, we turn to moderate deviation principles, recall Definition \ref{def:LDPMDP}. The first one is a moderate deviation principle on $\RR$ for integrals with respect to the empirical measures.
  
\begin{thm}[Moderate deviation principle]\label{moderatedeviations}
	Let $\xi\in \Xi$, $f\in\cC(\RR^d;\SSd)$ with $\langle f^2, \sigma_{d-1}\rangle>0$ and $(a_\lambda)_{\lambda > 0}$ be a sequence of real numbers that satisfies the growth condition
	\begin{align}\label{alambda}
	\lim\limits_{\lambda \rightarrow \infty} a_\lambda = \infty \quad \text{und} \quad \lim\limits_{\lambda \rightarrow \infty} a_\lambda\,  (\log \la)^{-\frac{d-1}{4(2(3dv[\xi] + u[\xi] + z[\xi]) + 9)}} = 0\,.
	\end{align}
	Then, 
	\begin{align*}
	\left(\frac{1}{a_\lambda} \frac{\langle f_{R_\la},\bar{\mu}_\lambda^\xi\rangle}{\sigma_\lambda^\xi(f_{R_\la})}\right)_{\lambda > 0}
	\end{align*}
	fulfils a moderate deviation principle on $\RR$ with speed $a_\lambda^2$ and rate function $I(x) = \frac{x^2}{2}$.
\end{thm}

In a next step we lift the result of Theorem \ref{moderatedeviations} to a moderate deviation principle on $\cM(\SSd)$ for the empirical measures themselves. This clearly goes beyond the results stated in the introduction. For this, we supply the space $\cM(\SSd)$ with the usual weak topology. To present our result, we recall from Theorem 2.1 in \cite{CalkaYukich} that for all $\xi\in\Xi$ there exists a constant $\sigma_\infty^\xi\in(0,\infty)$ such that
\begin{align*}
\lim_{\la\to\infty} (2\, \log \la)^{-{d-1\over 2}}\,\var[\langle f_{R_\la},\mu_\la^\xi\rangle]=(\sigma_\infty^\xi)^2\,\langle f^2, \sigma_{d-1}\rangle
\end{align*}
for $f\in\cC(\RR^d;\SSd)$. The strict positivity of $\sigma_\infty^\xi$ follows thereby from the considerations in \cite{BaranyVu} (note that Theorem 6.3 in \cite{BaranyVu} contains a misprint and the exponent $(d-3)/2$ there has to be replaced by $(d-1)/2$). 

\begin{thm}[Moderate deviation principle for empirical measures]\label{thm:MDPMeasureGeneral}
	Let $\xi\in\Xi$ and let $(a_\la)_{\la>0}$ be such that the growth condition \eqref{alambda} is satisfied. Then the family
	$$
	\left({1\over a_\la}{\bar\mu_\la^\xi\over \sigma_\infty^\xi\,(2\log \la)^{{(d-1)/4}}}\right)_{\la>0}
	$$
	satisfies a moderate deviation principle on $\cM(\SS^{d-1})$ with speed $a_\la^2$ and rate function
	\begin{align*}
	I(\nu) = \begin{cases} {1\over 2}\langle\varrho^2, \sigma_{d-1}\rangle &: \nu\ll \sigma_{d-1}\text{ with density }\varrho={\dint\nu\over\dint \sigma_{d-1}}\\ \infty &: \text{otherwise}\,.
	\end{cases}
	\end{align*}
\end{thm}


\begin{remark}\label{rem:Optimality}
Except of Theorem \ref{StarkesGesetzVolumen}, we do not claim that our findings are best possible. However, in order to improve them by our methods one would have to decrease the exponent at $k!$ in the cumulant bound in Theorem \ref{cumulantestimate} below from $3dv[\xi]+u[\xi]+5+z[\xi]$ to (optimally) $1$. This would then imply that $\gamma=0$ in the application of Lemma \ref{VorbereitungKumulante}, which is best possible. It is unclear to us and seems unlikely that such an improvement is possible in the framework of Gaussian polytopes. We even doubt that the exponent can be chosen independently of $d$.
\end{remark}

\section{Proof of the main results}\label{sec:ProofMain}

\subsection{A cumulant bound and proof of the theorems for empirical measures}

The proof of our results relies on the following cumulant estimate, whose proof is the content of Sections \ref{sec:MomentEstimates} and \ref{sec:CumulantProof}. To present it, recall the definition of the weights $u[\xi]$, $v[\xi]$, $w[\xi]$ and $z[\xi]$ from \eqref{eq:WEIGHTS} and \eqref{zzz}, respectively. In what follows we write $c,c_1,c_2,\ldots$ for positive and finite constants that are allowed to depend only on the dimension $d$ and the geometric functional $\xi$ we consider if not stated otherwise. Their values may change from line to line. Different constants in the same line are numbered consecutively.

\begin{thm}[Cumulant bounds]\label{cumulantestimate}
Let $k\in \left\lbrace 3,4,\ldots \right\rbrace $ and $f\in \cB(\RR^d)$. Then, for sufficiently large $\lambda$,
\begin{align*}
	  | \langle f_{R_\la}^k, c_\lambda^k\rangle | \le \begin{cases}
	 c_1\, c_2^k\, \left\| f\right\|_\infty^k\, R_\la^{d-1}\, (k!)^{3d + 5} &: \xi = \xi_V\\
	 c_1\, c_2^k\, \left\| f\right\|_\infty^k\, R_\la^{d-1}\, (k!)^{d + 5} &: \xi = \xi_{f_0}\\
	 c_1\, c_2^k\, \left\| f\right\|_\infty^k\, R_\la^{d-1}\, (k!)^{j(3d + 1) + 5} &: \xi = \xi_{f_j}\ \text{for some}\ j\in \{1,\ldots,d-1\}
	 \end{cases}
\end{align*}	
with constants $c_1,c_2\in (0,\infty)$ that only depend on $d$ and $\xi$. In a unified form, this means that
\begin{align*}
| \langle f_{R_\la}^k, c_\lambda^k\rangle | \le c_1\, c_2^k\, \left\| f\right\|_\infty^k\, R_\la^{d-1}\, (k!)^{3dv[\xi] + u[\xi] + 5 + z[\xi]}
\end{align*}
for all $\xi\in \Xi$.
\end{thm}

The previous bound is now combined with the following lemma. It summarizes results from \cite{DoeringEichelsbacher}, \cite{EichelsbacherSchreiberRaic} and \cite{SaulisBuch} in a simplified form that is tailored towards our applications. Let us write $c^k[X]$, $k\in\NN$, for the $k$th cumulant of a random variable $X$ with $\EE|X|^k<\infty$, that is,
$$
c^k[X] = (-\mathfrak{i})^{k}\,{\dint^k\over \dint t^k}\log\EE[\exp(\mathfrak{i}tX)]\Big|_{t=0}\,,
$$
where $\mathfrak{i}$ is the imaginary unit.

 \begin{lem}\label{VorbereitungKumulante}
 	Let $(X_\lambda)$ be a family of random variables with $\EE[X_\lambda] = 0$ and $\var[X_\lambda] = 1$ for all $\lambda > 0$. Suppose that, for all $k\in \{3,4,\ldots\}$ and sufficiently large $\la$,
 	\begin{align*}
 	|c^k[X_\lambda]| \leq \frac{(k!)^{1 + \gamma}}{(\Delta_\lambda)^{k-2}}
 	\end{align*}
 	with a constant $\gamma \in [0,\infty)$ not depending on $\la$ and constants $\Delta_\lambda \in (0,\infty)$ that may depend on $\la$.  Then the following assertions are true.
 	\begin{description}
 		\item[(i)] For all $y\geq 0$ and sufficiently large $\la$,
 		\begin{align*}
 		\PP(|X_\lambda| \geq y) \leq 2 \exp\left(-\frac{1}{4} \min\left\{\frac{y^2}{2^{1+\gamma}}, (y\, \Delta_\lambda)^{1/(1+\gamma)}\right\}\right). 
 		\end{align*}
 		\item[(ii)] There exist constants $c_1,c_2 \in (0,\infty)$ only depending on $\gamma$ such that for sufficiently large $\la$ and $0\leq y\leq c_1\,(\Delta_\lambda)^{\frac{1}{1 + 2\gamma}}$,
 		\begin{align*}
 		\left|\log \frac{\PP(X_\lambda \geq y)}{1 - \Phi(y)}\right| &\leq c_2\, (1 + y^3)\, (\Delta_\lambda)^{-\frac{1}{(1+2\gamma)}} \quad and\\
 		\left|\log \frac{\PP(X_\lambda \leq -y)}{\Phi(-y)}\right| &\leq c_2\, (1 + y^3)\,(\Delta_\lambda)^{-\frac{1}{(1+2\gamma)}}.
 		\end{align*}
 		\item[(iii)] Let $(a_\lambda)_{\lambda > 0}$ be a sequence of real numbers such that
 		\begin{align*}
 		\lim\limits_{\lambda \rightarrow \infty} a_\lambda = \infty \quad \text{and}\quad \lim\limits_{\lambda \rightarrow \infty} a_\lambda\, \Delta_\lambda^{-\frac{1}{1 + 2\gamma}} = 0.
 		\end{align*}
 		Then $(a_\lambda^{-1} X_\lambda)_{\lambda > 0}$ satisfies a moderate deviation principle on $\RR$ with speed  $a_\lambda^2$ and rate function $I(x) = \frac{x^2}{2}$.
 		\item[(iv)] One has the Berry-Esseen bound
 		\begin{align*}
 		\sup\limits_{y\in \mathbb{R}} |\PP(X_\lambda \leq y) - \Phi(y)| \leq c\, (\Delta_\lambda)^{-1/(1+2\gamma)}
 		\end{align*}
 		with a constant $c\in (0,\infty)$ that only depends on $\gamma$.
 	\end{description} 
 \end{lem} 

Now, we fix $\xi\in \Xi$ and $f\in\cC(\mathbb{R}^d,\SSd)$ with $\langle f^2, \sigma_{d-1}\rangle>0$. The cumulant bound in Theorem \ref{cumulantestimate} and the variance estimate \eqref{SchrankefürVarianz} imply that, for all $k\in \{3,4,\ldots\}$ and a sufficiently large $\lambda$,
 \begin{align*}
 \frac{|\langle f_{R_\la}^k,c_\lambda^k\rangle|}{(\sigma_\lambda^\xi(f_{R_\la}))^k} \leq c_1\, c_2^k\, \|f\|_\infty^k\, R_\la^{d-1}\, \left(k!\right)^{3dv[\xi] + u[\xi] + 5 + z[\xi]}\, \left(c_3\, \langle f^2, \sigma_{d-1}\rangle^{\frac{1}{2}} (\log \la)^{\frac{d-1}{4}}\right)^{-k}\,,
 \end{align*}
where the constants $c_1,c_2,c_3\in(0,\infty)$ depend on $d$ and $\xi$ only.
By definition of $R_\la$ we easily see that $R_\la\le \sqrt{2\, \log \la}$ and hence
 \begin{align*}
  \frac{|\langle f_{R_\la}^k,c_\lambda^k\rangle|}{(\sigma_\lambda^\xi(f_{R_\la}))^k} &\le  c_1\, c_2^k\, \|f\|_\infty^k\, (\log \la)^{\frac{d-1}{2}}\, \left(k!\right)^{3dv[\xi] + u[\xi] + 5 + z[\xi]}\, \left(\langle f^2, \sigma_{d-1}\rangle^{\frac{1}{2}} (\log \la)^{\frac{d-1}{4}}\right)^{-k}\\ 
  &\leq c_1 \left(\frac{c_2\, \|f\|_\infty}{\langle f^2, \sigma_{d-1}\rangle^{\frac{1}{2}}}\right)^k\, (\log \la)^{\frac{d-1}{4}(2-k)}\, \left(k!\right)^{3dv[\xi] + u[\xi] + 5 + z[\xi]}\\
 &= c_1\, c_2^k\, \|f\|_\infty^2 \left(\frac{(\log \la)^{\frac{d-1}{4}}}{\|f\|_\infty}\right)^{2-k}\, \left(k!\right)^{3dv[\xi] + u[\xi] + 5 + z[\xi]}\\
 &\leq \left(\frac{(\log \la)^{\frac{d-1}{4}}}{c_2\, \|f\|_\infty\, \max\{1,c_1\, c_2^2\, \|f\|_\infty^2\}}\right)^{-(k-2)} \left(k!\right)^{3dv[\xi] + u[\xi] + 5 + z[\xi]}\,,
 \end{align*} 
 where in the last line $c_1\in(0,\infty)$ only depends on $d$ and $\xi$ and $c_2\in(0,\infty)$ additionally depends on the function $f$.

 \begin{proof}[Proof of Theorem \ref{thm:ConcentrationEmpMeas}, Theorem \ref{thm:DeviationsEmpMeas}, Theorem \ref{CLTempiricalmeasure} and Theorem \ref{moderatedeviations}]
 Define
 \begin{align}\label{Wahl}
 \gamma := 3dv[\xi] + u[\xi] + 4 + z[\xi] \quad \text{and} \quad \Delta_\lambda := \frac{(\log \la)^{\frac{d-1}{4}}}{c_2\, \|f\|_\infty\, \max\{1,c_1\, c_2^2\, \|f\|_\infty^2\}} 
 \end{align}
with $c_1$ and $c_2$ as above. Then our computations performed above imply that the random variables 
\begin{align*}
X_\lambda := \frac{\langle f_{R_\la},\bar{\mu}_\lambda^\xi\rangle}{\sigma_\lambda^\xi(f_{R_\la})}
\end{align*}
fulfil the conditions of Lemma \ref{VorbereitungKumulante} with the constants $\gamma$ and $\Delta_\la$ there given by \eqref{Wahl}. This completes the proof.
\end{proof}

\begin{proof}[Proof of Theorem \ref{thm:MDPMeasureGeneral}]
Theorem \ref{thm:MDPMeasureGeneral} in our text is the Gaussian analogue of Theorem 3.5 in \cite{GroteThäle}, dealing with random polytopes in the unit ball, and is proved along the lines of the proof of Theorem 1.5 in \cite{EichelsbacherSchreiberRaic}. For this reason, we leave out the details.
\end{proof}

\subsection{Proof of the theorems for the volume and the face numbers}

\begin{proof}[Proof of Theorem \ref{ConcentrationEinleitung}, Theorem \ref{DeviationprobabilityEinleitung} and Theorem \ref{ModeratedeviationsEinleitung}]
We start with the face numbers of the Gaussian polytopes $K_\la$. It is clear that in this case Theorems \ref{ConcentrationEinleitung}, \ref{DeviationprobabilityEinleitung} and \ref{ModeratedeviationsEinleitung} presented in the introduction follow directly from the corresponding results in Section 3 by putting $f_{R_\la}\equiv 1$, since
\begin{align*}
\langle 1, \bar\mu_\la^{\xi_{f_j}} \rangle = \sum\limits_{x\in \mathcal{P}_\la} \xi_{f_j}(x,\mathcal{P}_\la) - \EE\Big[ \sum\limits_{x\in \mathcal{P}_\la} \xi_{f_j}(x,\mathcal{P}_\la) \Big] = f_j(K_\la) - \EE[f_j(K_\la)]\,
\end{align*}
for all $j\in \{0,\ldots,d-1\}$.

We turn now to the volume $\vol(K_\la)$ of $K_\la$.  First of all we deduce from Theorem \ref{cumulantestimate} and \eqref{überzähligesRlambda} that
\begin{align*}
|c^k[R_\la(\vol(\BB^d(0,R_\la)) - \vol(K_\la))]| \le c_1\, c_2^k\, R_\la^{d-1}\, (k!)^{3d + 5}.
\end{align*}
Since, for a random variable $X$, $c^k[rX]=r^kc^k[X]$ for all $k\in\NN$ and $r\in\RR$, and $c^k[X-r]=c^k[X]$ for all $k\in\{2,3,\ldots\}$ and $r\in\RR$, we see that
\begin{align*}
|c^k[R_\la(\vol(\BB^d(0,R_\la)) - \vol(K_\la))]| &= |(-1)^kR_\la^k\, c^k[ \vol(K_\la)-\vol(\BB^d(0,R_\la)) ]| \\
&= R_\la^k\, |c^k[\vol(K_\la)]|
\end{align*}
and we thus get
\begin{align}\label{KumulanteVolumen}
\begin{split}
|c^k[\vol(K_\la)]| &\le c_1\, c_2^k\, R_\la^{d-1-k}\, (k!)^{3d + 5} \le c_3\, c_4^k\, (\log \la)^{\frac{d-1-k}{2}}\, (k!)^{3d + 5}
\end{split}
\end{align}
for $k\in\{3,4,\ldots\}$. In combination with the lower variance bound $\var[\vol(K_\la)] \ge c\, (\log \la)^{\frac{d-3}{2}}$ from \cite{BaranyVu}, where $c\in (0,\infty)$ is a constant that depends only on $d$, we get similarly as above
\begin{align*}
\frac{|c^k[\vol(K_\la)] |}{(\sqrt{\var[\vol(K_\la)]})^k} &\le c_1\, c_2^k\, (\log \la)^{\frac{d-1-k}{2}}\, (k!)^{3d + 5} \big(c_3\, (\log \la)^{\frac{d-3}{2}} \big)^{-{k\over 2}}\\
&\le  c_1\, c_2^k\, (\log \la)^{\frac{2(d-1-k) - k(d - 3)}{4}}\, (k!)^{3d + 5}\\
&\le  c_1\, c_2^k\, (\log \la)^{\frac{2(d-1)- k(d - 1)}{4}}\, (k!)^{3d + 5}\\
&\le \big(c\, (\log \la)^{\frac{d-1}{4}}\big)^{-(k-2)}\, (k!)^{3d + 5}
\end{align*} 
with constants $c,c_1,c_2,c_3\in (0,\infty)$ only depending on $d$. Thus, the random variables
\begin{align*}
X_\la:= \frac{\vol(K_\la) - \EE [\vol(K_\la)]}{\sqrt{\var [\vol(K_\la)]}}
\end{align*}
 fulfil the conditions of Lemma \ref{VorbereitungKumulante} with $\gamma := 3d + 4$ and $\Delta_\la:= c\, (\log \la)^{\frac{d-1}{4}}$ with a constant $c\in(0,\infty)$ depending on $d$ only. This completes the proof for the volume functional. 
 \end{proof}
 
\begin{remark}
Let us briefly draw attention to the following interesting observation. Namely, while the $k$th moment of the volume of the Gaussian polytopes $K_\la$ grows like $(\log \la)^{k\frac{d}{2}}$ for fixed $k$ as a function of $\la$, see Theorem \ref{MomentboundEinleitung}, the behaviour of the corresponding $k$th order cumulant is completely different. Indeed, as can be seen from Equation \eqref{KumulanteVolumen}, the $k$th order cumulant of $\vol(K_\la)$ is bounded from above by a constant multiple of $(\log \la)^{\frac{d-1-k}{2}}$ (this obviously holds also for the case that $k=2$). Thus if $d\geq 3$, the $k$th cumulant is, as a function of $\la$, growing as long as $k<d-1$, constant for $k=d-1$ and tends to zero for all $k>d-1$. Moreover, in the special case $d=2$, only the first cumulant (that is, the expectation) grows like a constant multiple of $\sqrt{\log\la}$, while all other cumulants, including the variance, tend to zero, as $\la\to\infty$.
\end{remark}

\begin{proof}[Proof of Theorem \ref{StarkesGesetzVolumen}]
	Define the sequence $(\la_k)_{k\in\NN}$ by putting $\la_k := a^k$, $a>1$, for all $k\in\NN$. The concentration inequality stated in Theorem \ref{ConcentrationEinleitung} and the lower variance bound $\var[\vol(K_\la)] \ge c\, (\log \la)^{\frac{d-3}{2}}$ from \cite{BaranyVu} imply, together with the elementary inequality $\exp(-\min\{a,b\}) \le \exp(-a) + \exp(-b)$, $a,b\ge 0$, for sufficiently large $k$, $p\in \RR$ and $\varepsilon > 0$ that 
	\begin{align*}
		&\PP\big(|\vol(K_{\la_k})-\EE [\vol(K_{\la_k})]|\geq \varepsilon\, (\log \la_k)^{p\frac{d}{2}}\big)\\
		&\qquad \le 2\exp\Big(-{1\over 4}\min\Big\{{c_1\, \left(\frac{\varepsilon\, (\log \la_k)^{p\frac{d}{2}}}{\sqrt{\var [\vol(K_{\la_k})]}}\right)^2},c_2\,  (\log \la_k)^{d-1\over 4(3d+5)}\, \left(\frac{\varepsilon\, (\log \la_k)^{p\frac{d}{2}}}{\sqrt{\var [ \vol(K_{\la_k})]}}\right)^{1\over 3d+5}\Big\}\Big)\\
		&\qquad  \le 2\exp\Big(-{1\over 4}\min\Big\{{c_1\, \left(\frac{\varepsilon\, (\log \la_k)^{p\frac{d}{2}}}{\sqrt{(\log \la_k)^{\frac{d-3}{2}}}}\right)^2},c_2\,  (\log \la_k)^{d-1\over 4(3d+5)}\, \left(\frac{\varepsilon\, (\log \la_k)^{p\frac{d}{2}}}{\sqrt{(\log \la_k)^{\frac{d-3}{2}}}}\right)^{1\over 3d+5}\Big\}\Big)\\
		&\qquad  = 2\exp\Big(-{1\over 4}\min\Big\{{c_1\, \varepsilon^2\, (\log \la_k)^{\frac{d(2p - 1) + 3}{2}}},c_2\, \varepsilon^{\frac{1}{3d + 5}} (\log \la_k)^{pd +1 \over 2(3d+5)}\,\Big\}\Big)\\
	    &\qquad \le 2\exp\left(- c_1\, \varepsilon^2\, (\log \la_k)^{\frac{d(2p-1) + 3}{2}}\right)  +2 \exp\left(- c_2\, \varepsilon^{\frac{1}{3d+5}}  (\log \la_k)^{\frac{pd + 1}{2(3d+5)}}\right).
	\end{align*}
	Next, we notice that
	\begin{align*}
		\sum_{k=1}^\infty \exp \big(- (\log \la_k)^{\frac{pd + 1}{2(3d + 5)}}\big)  = \sum_{k=1}^\infty \exp \big(- (k\log a)^{\frac{pd + 1}{2(3d + 5)}}\big) < \infty
	\end{align*}
	for all $p> -\frac{1}{d}$. Since $-\frac{1}{d} \le \frac{d-3}{2d}$ for $d\ge 2$, the series converges for all $p>\frac{d-3}{2d}$. Similarly, one has that
	\begin{align*}
		\sum_{k=1}^\infty \exp \big(- (\log \la_k)^{\frac{d(2p-1) + 3}{2}}\big) = \sum_{k=1}^\infty \exp\big(-(k\log a)^{d(2p-1)+3\over 2}\big)
	\end{align*}
	is finite as long as $d(2p-1) + 3 > 0$, which is equivalent to $p > \frac{d-3}{2d}$. Thus, the series
	$$
	\sum_{k=1}^\infty \PP\big(|\vol(K_{\la_k})-\EE [\vol(K_{\la_k})]|\geq \varepsilon\, (\log \la_k)^{p\frac{d}{2}}\big)
	$$
	converges for all $p>{d-3\over 2d}$ and the Borel-Cantelli lemma implies that 
	\begin{align}\label{eq:SLLNBorelCantelli}
		\frac{\vol (K_{\la_k}) - \EE [\vol(K_{\la_k})]}{(\log \la_k)^{p \frac{d}{2}}} \longrightarrow 0
	\end{align}
	with probability 1, as $k \rightarrow \infty$, for all $p> \frac{d-3}{2d}$. This completes the proof of (i). Part (ii) follows in the same way and is therefore omitted.
\end{proof}

\begin{proof}[Proof of Theorem \ref{MomentboundEinleitung}]
It is well known, see \cite{Leonov}, that the $k$th moment of a random variable $X$ equals the $k$th complete Bell polynomial evaluated in $c^1[X],\ldots,c^k[X]$, that is, 
$$
\EE[X^k] = B_k(c^1[X],\ldots,c^k[X])
$$
with 
\begin{align*}
B_k(x_1,\ldots,x_k) := \sum_{i=1}^{k} B_{k,i}(x_1,\ldots,x_{k-i+1})\,,
\end{align*}
where the Bell polynomials $B_{k,i}$ are given by
\begin{align}\label{Bell}
B_{k,i}(x_1,\ldots,x_{k-i+1}) := \sum \frac{k!}{j_1!\cdots j_{k-i+1}!} \left(\frac{x_1}{1!}\right)^{j_1} \left(\frac{x_2}{2!}\right)^{j_2} \cdots \left(\frac{x_{k-i+1}}{(k-i+1)!}\right)^{j_{k-i+1}}
\end{align}
and the sum runs over all $k-i+1$ tuples $(j_1,\ldots,j_{k-i+1})$ of natural numbers including zero satisfying $\sum_{\ell=1}^{k-i+1} j_\ell = i$ and $\sum_{\ell=1}^{k-i+1} \ell\, j_\ell = k$.
So, the $k$th moment can be written as a polynomial of the cumulants up to order $k$ that in particular contains the term $(c^1[X])^k$ ($i = k$ in the above sum).

Let us first consider the volume of $K_\la$. As anticipated in the introduction, it is known from the literature that $c^1[\vol(K_\la)] = \EE[ \vol(K_\la)] \le c_1\, (\log \la)^{\frac{d}{2}}$ and $c^2[\vol(K_\la)] = \var[ \vol(K_\la)] \le c_2\, (\log \la)^{\frac{d-3}{2}}$ with constants $c_1,c_2\in (0,\infty)$ that only depend on $d$, see \cite{ReitznerSurvey}, for example. Because of the estimate \eqref{KumulanteVolumen} all higher-order cumulants are also bounded by a constant times $(\log \la)^{\frac{d-3}{2}}$. Hence, $(c^1[\vol(K_\la)])^k$ is the dominating term in the above sum. Moreover, it is known that the total number of terms appearing in the $k$th complete Bell polynomial is the same as the total number of integer partitions of $k$, which in turn is bounded by $\exp(\pi\sqrt{2k/3})$, see \cite{Pribitkin}. Clearly, \eqref{Bell} shows that each coefficient in the sum is bounded from above by $k!$. Combining these facts we get for all sufficiently large $\la$ and $k\in \NN$ that 
\begin{align*}
\EE [\vol(K_\la)^k] \le c_1^k\, k!\, (c^1 [\vol(K_\la)])^k \le c_2\, c_3^k\, k!\,  (\log \la)^{k \frac{d}{2}}
\end{align*}
with a universal constant $c_1\in(0,\infty)$ and constants $c_2,c_3\in (0,\infty)$ that only depend on $d$.

In a second step we prove the upper moment bound for the number $f_j(K_\la)$ of $j$-dimensional faces of $K_\la$ for $j\in \{0,\ldots,d-1\}$.  Since each cumulant of $f_j(K_\la)$ of order higher than two is bounded by a constant multiple of $(\log \la)^{\frac{d-1}{2}}$ in view of Theorem \ref{cumulantestimate} and the expectation and variance growth is of the same order,  it follows again that $(c^1[f_j(K_\la)])^k$ is the term of leading order and hence we achieve similarly as above for all sufficiently large $\la$ and $k\in \NN$ that
\begin{align*}
\EE [f_j(K_\la)^k] \le c_1^k\, k!\, (c^1[f_j(K_\la)])^k \le c_2\, c_3^k\, k!\, (\log \la)^{k \frac{d-1}{2}}
\end{align*}
again with a universal constant $c_1\in(0,\infty)$ and constants $c_2,c_3\in (0,\infty)$ that only depend on $d$ and $j$. 

Finally, Jensen's inequality implies that the $k$th moment of $\vol(K_\la)$ and the $k$th moment of $f_j(K_\la)$ for all $j\in \{0,\ldots,d-1\}$ are bounded from below by $(\EE[\vol(K_\la)])^k$ and $(\EE[f_j(K_\la)])^k$ respectively. Using the estimates for the respective expectations completes the proof.
\end{proof}

\section{Moment estimates}\label{sec:MomentEstimates}

This section contains the first step of the proof of the cumulant estimate presented in Theorem \ref{cumulantestimate} above. Again, recall the definition of the weights $u[\xi]$, $v[\xi]$, $w[\xi]$ and $z[\xi]$ from \eqref{eq:WEIGHTS} and \eqref{zzz}, respectively.

Lemma 4.4 in \cite{CalkaYukich} shows that the geometric functionals $\xi\in\Xi$ have finite moments of all orders. In fact, this easily follows from the properties summarized in Lemma \ref{Eigenschaften}. However, in our context it is crucial to have control on the precise growth of these moments. This considerably refines the previous result from \cite{CalkaYukich}. As already discussed above, deriving such bounds in the context of Gaussian polytopes is a  much more delicate task compared to random polytopes in the unit ball studied in \cite{GroteThäle}, although we could follow at the beginning the principal idea from \cite{CalkaYukich}.

\begin{prop}\label{Momente}
Let $\xi\in \Xi$.
	\begin{description}
		\item[(i)] For all $p\in \NN$, $x=(v,h)\in W_\lambda$ and sufficiently large $\lambda$ one has that 
		\begin{align*}
		\EE\big| \xi^{(\lambda)} (x,\mathcal{P}^{(\lambda)})\big|^p &\le c_1\, c_2^p\, (pu[\xi])!\, ((pdv[\xi])!)^{2}\, (1 + \left| h\right| )^{p(d-1)v[\xi] + d}\, \exp\left(-\frac{e^{h\lor 0}}{c_3}\right)
		\end{align*}
		with constants $c_1,c_2,c_3\in (0,\infty)$ only depending on $\xi$ and $d$. 
		\item[(ii)] For all $k\in\NN$, $x_1=(v_1,h_1), \ldots, x_k=(v_k,h_k)\in W_\lambda$ and sufficiently large $\lambda$ one has that 
		\begin{align*}
		&\EE\left[\left(\prod_{i=1}^{k} \xi^{(\lambda)} \left(x_i, \left(\mathcal{P}^{(\lambda)} \cup\bigcup_{i=1}^k\{x_{i}\}\right)\cap C_{d-1}\left(x_i,\frac{\delta}{2}\right) \right)\right)^2\right]\\
		&\le \EE\left[\left(\prod_{i=1}^{k} \xi^{(\lambda)} \left(x_i, \mathcal{P}^{(\lambda)} \cup\bigcup_{i=1}^k\{x_{i}\}\right)\right)^2\right]\\
		&\le c_1\, c_2^k\, ((ku[\xi])!)^{2}\, ((kdv[\xi])!)^{4}\, \prod_{i=1}^{k}\,\left[ (1+\left|h_i\right|)^{2dw[\xi]}\, \exp\left(-\frac{e^{h_i\lor 0}}{c_3\, k}\right)\right]
		\end{align*}
		with $\delta := \min_{i\neq j=1,\ldots,k} \left\|v_i - v_j\right\|$ and constants $c_1,c_2,c_3\in (0,\infty)$ only depending on $\xi$ and $d$.
	\end{description}
\end{prop}

We start with some inequalities that will repeatedly be applied below. Their elementary proofs are left to the reader. 

\begin{lem} 
	\begin{description}
		\item[(i)] For $d,j,p\in\NN$ one has that
		\begin{align}\label{UngleichungFakultät}
		(3pj)!\leq 27^{pj}\,((pj)!)^3\,,\quad (2pd)! \le 4^{pd}\, ((pd)!)^2 \quad \text{and}\quad (2p)! \le 4^p\, (p!)^2\,.
		\end{align}
		\item[(ii)] For all $d,k,j \in \NN$ we have
		\begin{align}\label{UngleichungFakultät2}
		(d(k+j))! \le c_1\, c_2^k\, (dk)!\,.
		\end{align}
		with constants $c_1,c_2\in(0,\infty)$ that only depend on $d$ and $j$.
		\item[(iii)] For all $a_1,\ldots,a_n \in \NN$, $n \in \NN$, and $b = a_1 + \ldots + a_n$ one has that
		\begin{align}\label{UngleichungFakultät3}
		(a_1!)(a_2!)\cdots(a_n!) \le b!\,.
		\end{align}
	\end{description}
\end{lem}

\begin{proof}[Proof of Proposition \ref{Momente} (i) for $\xi=\xi_V$]
	We start with the first assertion and choose $\xi=\xi_V$. Because of rotation invariance of the underlying point process, we may assume that the point $x$ has representation $(0,h)$ with $h\in (-\infty,R_\la^2]$.  For all $M\in (0,\infty)$ and sufficiently large $\lambda$ put
	\begin{align*}
	D^{(\lambda)}(M) := \| \partial \Psi^{(\lambda)}(\mathcal{P}^{(\lambda)}) \cap C_{d-1}(0,M)\|_\infty \quad \text{and let} \quad L:= L(\xi^{(\lambda)}, (0,h)) 
	\end{align*} 
	be the radius of localization of the functional $\xi^{(\lambda)}$ at $(0,h)$. Then, for sufficiently large $\lambda$, $\left| \xi^{(\lambda)} ((0,h),\mathcal{P}^{(\lambda)})\right|$ is bounded by the Lebesgue measure of the set
	$$\BB_{d-1}(0,L) \times [ -D^{(\lambda)}(L), D^{(\lambda)}(L)]$$ 
	times
	$$
	c_1\, (1+D^{(\la)}(L)/R_\la^2)^{d-1} \leq c_1\,(1+D^{(\la)}(L))^{d-1} \leq c_1\,2^{d-1}\,D^{(\la)}(L)^{d-1} = c_2\,D^{(\la)}(L)^{d-1}
	$$
	in view of \eqref{eq:DensityLebesgueUnderTrafo}.
    Hence,
	\begin{equation}\label{eq:NeueGleichungXXX}
	\begin{split}
	\EE\big| \xi^{(\lambda)} ((0,h),\mathcal{P}^{(\lambda)})\big|^p &\le c_1\,c_2^p\, \EE\big| L^{d-1} D^{(\lambda)}(L)^d\big|^p\\
	& \le c_1\,c_2^p\, \big(\EE \big[L^{2p(d-1)}\big]\big)^{1/2}\big(\EE \big[D^{(\lambda)}(L)^{2pd}\big]\big)^{1/2}
	\end{split}
	\end{equation}
	by the Cauchy-Schwarz inequality. Using \eqref{Lokalisierung2}, the substitution $t/c_2 = s$ and the definition of the gamma function we get
	\begin{align*}
	\EE[L^r] &= r \int\limits_0^\infty \PP(L>t)\, t^{r-1}\, \dint t \le r\, c_1 \int\limits_0^\infty \exp\left(-\frac{t}{c_2}\right)\, t^{r-1}\, \dint t + r \int\limits_0^{\left|h \right| } t^{r-1}\, \dint t \le c_3\, c_4^r\, r! + \left| h \right|^r 
	\end{align*}
	for all $r\in \left[1,\infty \right) $. Hence,
	\begin{align*}
	\EE\big[L^{2p(d-1)}\big] &\le c_1\, c_2^p\, (2p(d-1))! + \left| h \right|^{2p(d-1)} \le c_3\, c_4^p\, (2pd)! + \left| h \right|^{2p(d-1)}\\
	&\le c_1\, c_2^p\, ((pd)!)^2 + \left| h\right|^{2p(d-1)} \le  c_3\, c_4^p\, ((pd)!)^2\, (1 + \left| h\right|^{2p(d-1)})\\
	&\le c_1\, c_2^p\, ((pd)!)^2\, (1 + \left| h\right|)^{2p(d-1)}
	\end{align*}
	because of relation \eqref{UngleichungFakultät} and thus
	\begin{align}\label{Teil1}
	\big(\EE \big[L^{2p(d-1)}\big]\big)^{1/2}  \le \big( c_1\, c_2^p\, ((pd)!)^2\, (1+| h| )^{2p(d-1)}\big)^{1/2} \le c_3\, c_4^p\, (pd)!\, (1+ | h| )^{p(d-1)}\,.
	\end{align}
	On the other hand, we have that
	\begin{align}
	\begin{split}\label{L}
	\EE\big[D^{(\lambda)}(L)^r\big] &= \sum_{i=0}^{\infty} \EE\big[ D^{(\lambda)}(L)^r\, \text{\bf 1}(i\le L< i+1)\big]\le \sum_{i=0}^{\infty} \EE\big[ D^{(\lambda)}(i+1)^r\, \text{\bf 1}(L\geq i)\big]\\ 
	& \le \sum_{i=0}^{\infty} \big(\EE \big[D^{(\lambda)}(i+1)^{2r}\big]\big)^{1/2}\,   \PP(L>i)^{1/2}
	\end{split}
	\end{align}
	for all $r\in \left[ 1,\infty\right)$ by using the Cauchy-Schwarz inequality in the last step. Using \eqref{WahrscheinlichkeitHöhe} and \eqref{UngleichungFakultät} we obtain that
	\begin{align*}
	\EE\big[D^{(\lambda)}(i+1)^{2r}\big] &= 2r \int\limits_0^\infty \PP(D^{(\lambda)}(i+1) > t)\, t^{2r-1} \dint t\\
	&\le 2r\, c_1\, (i+1)^{2(d-1)} \int\limits_0^\infty \exp\left( -\frac{t}{c_2}\right)\, t^{2r-1}\, \dint t\\
	& = c_1\,c_2^r\, (i+1)^{2(d-1)}\, (2r)! \leq c_3\,c_4^r\,(i+1)^{2(d-1)}\,(r!)^2\,.
	\end{align*}
Combining this with \eqref{Lokalisierung2} and the fact that $\sqrt{a+b} \le \sqrt{a} + \sqrt{b}$ for all $a,b \ge 0$, it follows from \eqref{L} that
	\begin{align*}
	\EE\big[D^{(\lambda)}(L)^r\big] 
	& \le \sum_{i=0}^{\infty} c_1\, c_2^r\, (i+1)^{d-1}\, r!\, \left( c_3\, \exp\left(-\frac{i}{c_4}\right) + \text{\bf 1}(i\le \left| h\right| )\right)^{1/2}\\
	& \le c_1\, c_2^r\, r!\, \left( \sum\limits_{i=0}^{\infty} (i+1)^{d-1}\, \exp\left(-\frac{i}{c_3}\right) + \sum_{i=0}^{\infty} (i+1)^{d-1}\, \text{\bf 1}(i\le \left| h\right| )\right)\\
	&\le c_1\, c_2^r\, r!\, (1+\left|h\right|)^d\,,
	\end{align*}
	since the first sum is bounded by a constant only depending on $d$ and for the second one we have 
	\begin{align*}
	\sum_{i=0}^{\infty} (i+1)^{d-1}\, \text{\bf 1}(i\le \left| h\right|) \le (1 + |h|)^{d-1}\, \sum_{i=0}^{\infty} \text{\bf 1}(i\le \left| h\right| )\le (1 + |h|)^{d}\,.
	\end{align*}
	This shows that
	\begin{align}
	\begin{split}
	\label{Teil2}
	\big(\EE \big[D^{(\lambda)}(L)^{2pd}\big]\big)^{1/2} &\le \big(c_1\, c_2^p\, (2pd)!\, (1+|h|)^d\big)^{1/2}
	\le c_3\,c_4^p\, (pd)!\, (1+|h|)^d\,,
	\end{split}
	\end{align} 
	again by \eqref{UngleichungFakultät}. Summarizing, we conclude from \eqref{eq:NeueGleichungXXX}, \eqref{Teil1} and \eqref{Teil2} the bound
	\begin{align*}
	\EE\big| \xi^{(\lambda)} ((0,h),\mathcal{P}^{(\lambda)})\big|^p 
	& \le c_1\, c_2^p\, ((pd)!)^2\, (1+\left|h\right|)^{p(d-1) + d}\,.
	\end{align*}
	
	In a next step we improve this by an exponential term. Namely, if the point $(0,h)$ does not belong to the extreme points of the germ-grain model $\Psi^{(\lambda)}$, the functional $\xi^{(\la)}$ evaluated at this point is automatically equal to zero. This means that we can condition on this event without changing the value of the expression. By \eqref{WahrscheinlichkeitExtrempunkt} and the Cauchy-Schwarz inequality this leads to
	\begin{align*}
	\EE\big| \xi^{(\lambda)} ((0,h),\mathcal{P}^{(\lambda)})\big|^p &= \EE\big| \xi^{(\lambda)} ((0,h),\mathcal{P}^{(\lambda)})\, \text{\bf 1}((0,h) \in \text{ext}(\Psi^{(\la)}))\big|^p\\ 
	&\le  \left(c_1\, c_2^p\, ((2pd)!)^2\, (1+\left|h\right|)^{2p(d-1) + d}\,\right)^{1/2}\, \left(\exp\left(-\frac{e^{h\lor 0}}{c_3}\right)\right)^{1/2}\\
	&\le c_1\, c_2^p\, ((pd)!)^2\, (1+\left|h\right|)^{p(d-1)+d}\, \exp\left(-\frac{e^{h\lor 0}}{c_3}\right)\,,
	\end{align*}
	where we used \eqref{UngleichungFakultät} in the last step. This completes the proof of (i) for the volume functional $\xi_V$, since $u[\xi_V]=0$ and $v[\xi_V]=1$ by \eqref{eq:WEIGHTS}.
\end{proof}	
	
\begin{proof}[Proof of Proposition \ref{Momente} (i) for $\xi=\xi_{f_j}$]
Next, we turn to the $j$-face functional $\xi = \xi_{f_j}$ with $j\in\{0,\ldots,d-1\}$. Because of rotation invariance it is again enough to proof the assertion for the point $(0,h)$ with $h\in (-\infty,R_\la^2]$. Let $N^{(\la)}$ be the number of extreme points of $\Psi^{(\la)}$ that are contained in the cylinder $C_{d-1}(0,L)$, where $L$ is the radius of localization of the functional $\xi^{(\la)}$ at the point $(0,h)$. If $j=0$, then clearly $\xi_{f_0}\leq 1$ and hence
	\begin{align}\label{j=0}
	\EE\big|\xi_{f_0}^{(\la)}(x, \mathcal{P}^{(\la)})\big|^p\leq 1
	\end{align}
	for all $p\geq 1$, and it remains to consider the case that $j\in\{1,\ldots,d-1\}$. It is clear that
	$$
	\xi^{(\lambda)} ((0,h),\mathcal{P}^{(\lambda)}) \le \frac{1}{j+1}\, \binom{N^{(\la)}}{j} \leq (N^{(\la)})^j
	$$
	and hence
	$$
	\EE|\xi^{(\lambda)} ((0,h),\mathcal{P}^{(\lambda)})|^p \leq \EE[(N^{(\la)})^{pj}]\,.
	$$
	It is therefore enough to find a bound for $\EE[(N^{(\la)})^{pj}]$. We write $\nu_\la$ for the intensity measure of the rescaled Poisson point process $\cP^{(\la)}$ and observe that in view of \eqref{IntensityP^lambda} one has for sufficiently large $\la$, all $r\in [0,\pi R_\la]$ and $\ell \in (-\infty, R_\la^2 ]$ the inequality 
	\begin{align}\label{UngleichungDichte}
	\nu_\la(C_{d-1}(0,r) \cap (-\infty,\ell)) \le c\, r^{d-1}\, (-(\ell - 1) \lor 1)^{d-1}\, e^\ell\,.
	\end{align}
	Indeed, to verify this inequality we notice at first that for all such $r$ and $\ell$ we get by using the density function in \eqref{IntensityP^lambda} and all sufficiently large $\la$ that
	\begin{align*}
	&\nu_\la(C_{d-1}(0,r) \cap (-\infty,\ell))\\
	&\qquad = \int\limits_{C_{d-1}(0,r) \cap (-\infty,\ell)} \frac{\sin^{d-2} (R_\la^{-1}\|v\|)}{\|R_\la^{-1}v\|^{d-2}}\, \frac{\sqrt{2\log \la}}{R_\la}\, \exp\left(h-\frac{h^2}{2R_\la^2}\right) \left(1-\frac{h}{R_\la^2}\right)^{d-1} \dint v \dint h\\
	&\qquad \le c \int\limits_{C_{d-1}(0,r) \cap (-\infty,\ell)} e^h\, \left(1-\frac{h}{R_\la^2}\right)^{d-1} \dint v \dint h\\
	&\qquad = c \int\limits_{-\infty}^{\ell} e^h\, \left(1-\frac{h}{R_\la^2}\right)^{d-1} \dint h\, \int\limits_{\BB^{d-1}(0,r)} \dint v\\
	&\qquad = c\, r^{d-1}\, \int\limits_{-\infty}^{\ell} e^h\, \left(1-\frac{h}{R_\la^2}\right)^{d-1} \dint h\,,
	\end{align*}
	where for the inequality we have used that the fraction $\frac{\sin^{d-2} (R_\la^{-1}\|v\|)}{\|R_\la^{-1}v\|^{d-2}}$ involving the sinus term as well as the fraction $\frac{\sqrt{2\log \la}}{R_\la}$ involving the critical radius $R_\la$ are bounded from above by $1$ and a constant $c\in (0,\infty)$, respectively. Applying integration by parts $d-1$ times and using the fact that $(1- \ell/R_\la^2)\le (-(\ell - 1) \lor 1)$ for all $\ell \in (-\infty,R_\la^2]$ whenever $R_\la\geq 1$ yields
	\begin{align*}
	&\int\limits_{-\infty}^{\ell} e^h\, \left(1-\frac{h}{R_\la^2}\right)^{d-1} \, \dint h = e^\ell\,\left(1-\frac{\ell}{R_\la^2}\right)^{d-1} + {d-1\over R_\la^2}\int_{-\infty}^\ell e^h\, \left(1-\frac{h}{R_\la^2}\right)^{d-2} \, \dint h\\
	& \leq e^\ell\,\left(1-\frac{\ell}{R_\la^2}\right)^{d-1} + (d-1)\,e^\ell\,\,\left(1-\frac{\ell}{R_\la^2}\right)^{d-2} + (d-1)(d-2)\int\limits_{-\infty}^{\ell} e^h\, \left(1-\frac{h}{R_\la^2}\right)^{d-3} \, \dint h\\
	&\leq \ldots \leq (d-1)!\, e^{\ell}\, \sum\limits_{i=0}^{d-1} \left(1-\frac{\ell}{R_\la^2}\right)^{i} \le (d-1)!\, e^{\ell}\, \sum\limits_{i=0}^{d-1} \left(-(\ell - 1) \lor 1\right)^{i}\\
	&\leq  d!\, e^{\ell}\, \left(-(\ell - 1) \lor 1\right)^{d-1}\,,
	\end{align*}
	where we additionally used that $(-(\ell - 1) \lor 1)\ge 1$. Combining the last two calculations shows the bound claimed in \eqref{UngleichungDichte}, correcting thereby also a misprint in the display after Equation (4.18) in \cite{CalkaYukich}.
	
	Thus, writing $\Po(\alpha)$ for a Poisson random variable  with mean $\alpha>0$ and recalling the definition of the random variable $H$ (see the paragraph before Lemma \ref{Eigenschaften}) we see that, for sufficiently large $\la$,
	\begin{align*}
	& \EE[(N^{(\la)})^{pj}] \le \EE|\mathcal{P}^{(\la)} \cap (C_{d-1}(0,L) \cap (-\infty,H))|^{pj} \\
	&\le \sum_{i=0}^{\infty} \sum_{m=h}^{\infty} \EE[\Po(\nu_\la(C_{d-1}(0,i+1) \cap (-\infty,m+1)))^{pj}\, \text{\bf 1}(i\le L< i+1, m\le H< m+1) ]\\
	&\le \sum_{i=0}^{\infty} \sum_{m=h}^{\infty} \EE[\Po(c\, (i+1)^{d-1}\, (-m \lor 1)^{d-1}\, e^{m+1})^{pj}\, \text{\bf 1}(L\ge i, H\ge m) ]\,.
	\end{align*}   
	 (Here and below, $h$ has to be interpreted as the integer $\lfloor h\rfloor$, but we refrain from such a notation for simplicity. Moreover, from now on we interpret sums like $\sum_{i=h}^0a_i$ as zero, if $h>0$.)
	The moments of $\Po(\alpha)$ are well known and given by the Touchard polynomials. More precisely,
	$$\EE[\Po(\alpha)^k] = \sum_{i=1}^{k} \alpha^i\, \begin{Bmatrix}
	k\\
	i
	\end{Bmatrix}\,,\qquad k\in\NN\,,$$                                 
	where $\begin{Bmatrix}
	k\\
	i
	\end{Bmatrix}$ denotes the Stirling number of second kind. Thus,
	 \begin{align}\label{PoissonMómente}
	 \EE[\Po(\alpha)^k] = \sum_{i=1}^{k} \alpha^i\, \begin{Bmatrix}
	 k\\
	 i
	 \end{Bmatrix} \le \alpha^k \sum_{i=1}^{k} \begin{Bmatrix}
	 k\\
	 i
	 \end{Bmatrix} \text{\bf 1}(\alpha \ge 1) + \sum_{i=1}^{k} \begin{Bmatrix}
	 k\\
	 i
	 \end{Bmatrix} \text{\bf 1}(\alpha < 1) \le \alpha^k\, k! + k!\,,
	 \end{align} 
	 since 
	 \begin{equation}\label{eq:BoundStirlingPartitions}
	 \sum_{i=1}^k\begin{Bmatrix}
	 	 k\\
	 	 i
	 	 \end{Bmatrix}
	 	 =\sum_{L_1,\ldots,L_p\preceq\setk}1\leq k!\,.
	 \end{equation}
	 In the last step we use that the number of unordered partitions of $\{1,\ldots,k\}$ is known as the $k$th Bell number, which can be optimally bounded by $k!$ (meaning that there is no inequality of the type `$k$th Bell number $\le c(k!)^\alpha$' for some constant $c\in (0,\infty)$ and $\alpha<1$), see \cite{Diekert}. 
	
	 Now, H\"older's inequality, \eqref{PoissonMómente}, \eqref{UngleichungDichte}, \eqref{UngleichungFakultät} and the fact that $(a+b)^{1/3} \le a^{1/3} + b^{1/3}$ for $a,b>0$ imply
	 \begin{align*}
	 &\EE[(N^{(\la)})^{pj}]\\ 
	 &\le \sum_{i=0}^{\infty} \sum_{m=h}^{\infty}\, (\EE[\Po(c\, (i+1)^{d-1}\, (-m \lor 1)^{d-1}\, e^{m+1})^{3pj}])^{1/3}\, \PP(L\ge i)^{1/3}\, \PP(H\ge m)^{1/3} \\
	 &\le c_1\,c_2^p\, \sum_{i=0}^{\infty} \sum_{m=h}^{\infty}\, ((3pj)!\, (i+1)^{3p(d-1)j}\, (-m \lor 1)^{3p(d-1)j}\, e^{3p(m+1)j} + (3pj)!)^{1/3}\\\
	 &\hspace{5cm}\times\PP(L\ge i)^{1/3}\, \PP(H\ge m)^{1/3}\\
	 &\le c_1\, c_2^p\, \sum_{i=0}^{\infty} \sum_{m=h}^{\infty}\, (pj)!\, (i+1)^{p(d-1)j}\, (-m \lor 1)^{p(d-1)j}\, e^{p(m+1)j}\, \PP(L\ge i)^{1/3}\, \PP(H\ge m)^{1/3}\\
	 &\qquad\qquad + c_3\,c_4^p\, \sum_{i=0}^{\infty} \sum_{m=h}^{\infty}\, (pj)!\, \PP(L\ge i)^{1/3}\, \PP(H\ge m)^{1/3}\\
	 & =: T_1 + T_2\,.
	 \end{align*}   
	 We bound both terms $T_1$ and $T_2$ separately.  For $T_2$ we get by splitting the summation over $i$ into $i\leq|h|$ and $i>|h|$, and by using \eqref{Lokalisierung1} that it equals
	 \begin{align*}
	 & c_1\, c_2^p\, (pj)! \sum_{i=0}^{|h|} \PP(L\ge i)^{1/3}\sum_{m=h}^{\infty}\, \PP(H\ge m)^{1/3} + c_3\, c_4^p\, (pj)! \sum_{i=|h|}^{\infty} \PP(L\ge i)^{1/3}\sum_{m=h}^{\infty}\, \PP(H\ge m)^{1/3}\\
	 &\leq c_1\, c_2^p\, (pj)! \sum_{i=0}^{|h|}\sum_{m=h}^{\infty}\, \PP(H\ge m)^{1/3} + c_3\, c_4^p\, (pj)! \sum_{i=|h|}^{\infty} \exp(-i^2/c_5)\sum_{m=h}^{\infty}\, \PP(H\ge m)^{1/3}\\
	 &\le c_1\, c_2^p\, (pj)!\, |h|\, (|h| + c_3) + c_4\, c_5^p\, (pj)!\, (|h| + c_6) \le c_7\, c_8^p\, (pj)!\, (1 + |h|)^2\,,
	 \end{align*}
	 since
	 \begin{align*}
	 \sum_{m=h}^{\infty}\, \PP(H\ge m)^{1/3} &= \sum_{m=h}^{0}\, \PP(H\ge m)^{1/3}+ \sum_{m=0}^{\infty}\, \PP(H\ge m)^{1/3}\\
	 &\leq \sum_{m=h}^{0}1+ \sum_{m=0}^{\infty}\, c_1\exp(-e^m/c_2)\le |h| + c_3
	 \end{align*}
	 by \eqref{UngleichungH}. Now, we turn to $T_1$, where we first notice that
	 \begin{equation}\label{1}
	 \begin{split} 
	 \sum_{i=0}^{|h|} (i+1)^{p(d-1)j}\, \PP(L\ge i)^{1/3}
	 &\le \sum_{i=0}^{|h|} (1+|h|)^{p(d-1)j}\\
	 &\le |h|\, (1+|h|)^{p(d-1)j} \le  (1+|h|)^{p(d-1)j + 1}\,.
	 \end{split}
	 \end{equation}
	 Next, we need the inequality $\exp\left(-x\right) \le n!/x^n$, valid for all $n\in \NN$ and $x>0$. In particular,
	 \begin{align}\label{InequalityExponential1}
	 	\exp\left(-x\right) \le \frac{(pj + 1)!}{x^{pj+1}} \quad \text{and} \quad \exp\left(-x\right) \le \frac{(p(d-1)j + 2)!}{x^{p(d-1)j+2}}\,.
	 \end{align} 
	 Using this in conjunction with \eqref{UngleichungH} implies, together with the observation that $(-m\, \lor\, 1)^{p(d-1)j}=1$ if $m\geq 0$,
	 \begin{align*}
	 &\sum_{m=0}^{\infty} (-m \lor 1)^{p(d-1)j}\,e^{p(m+1)j}\, \PP(H\ge m)^{1/3} = e^{pj}\, \sum_{m=0}^{\infty} e^{mpj}\, \PP(H\ge m)^{1/3}\\
	 & \le c_1\,c_2^p\, \sum_{m=0}^{\infty} e^{mpj}\,  \exp\left(-\frac{e^m}{c_3}\right)\le c_3\, c_4^p\, (pj+1)! \sum_{m=0}^{\infty} e^{mpj} e^{-m(pj+1)} \\
	 &\le c_1\, c_2^p\, (pj+1)! \sum_{m=0}^{\infty} e^{-m}\le c_3\, c_4^p\, (pj+1)!\,,
	 \end{align*}
	 and by observing that $mpj \ge m$ for $m,p,j \ge 1$ we get
	 \begin{align*}
	 &\sum_{m=h}^{-1} (-m \lor 1)^{p(d-1)j}\, e^{p(m+1)j}\, \PP(H\ge m)^{1/3} = e^{pj} \sum_{m=h}^{-1}  |m|^{p(d-1)j}\, e^{mpj}\\
	 &\le e^{pj} \sum_{m=1}^{\infty}  m^{p(d-1)j}\, e^{-mpj}
	 \le e^{pj} \sum_{m=1}^{\infty}  m^{p(d-1)j}\, e^{-m}\\
	 &\le e^{pj}\, (p(d-1)j + 2)! \sum_{m=1}^{\infty} m^{p(d-1)j - p(d-1)j -2}\\
	 &= e^{pj}\, (p(d-1)j + 2)! \sum_{m=1}^{\infty}  m^{-2} \leq c_1\,c_2^p\, (p(d-1)j + 2)!\,,
	 \end{align*}
	 where in the first step we used that $(-m \lor 1)^{p(d-1)j} = |m|^{p(d-1)j}$.
	 Together with \eqref{UngleichungFakultät2} this leads to 
	 \begin{align}
	 \begin{split}\label{2}
	 &\sum_{m=h}^{\infty} (-m \lor 1)^{p(d-1)j}\,e^{p(m+1)j}\, \PP(H\ge m)^{1/3}\\
	 &= \sum_{m=h}^{-1} (-m \lor 1)^{p(d-1)j} \,e^{p(m+1)j}\, \PP(H\ge m)^{1/3}\\
	 &\qquad\qquad\qquad + \sum_{m=0}^{\infty} (-m \lor 1)^{p(d-1)j}\,e^{p(m+1)j}\, \PP(H\ge m)^{1/3}\\
	 &\leq c_1\, c_2^p\, (p(d-1)j + 2)! + c_3\, c_4^p\, (pj+1)!\le  c_5\, c_6^p\, (pdj)!\,.
	 \end{split}
	 \end{align}
	 Moreover, with \eqref{Lokalisierung1}, \eqref{UngleichungFakultät2}, \eqref{InequalityExponential1} and the fact that $i^2\ge i+1$ for all $i\ge 2$ we get
	 \begin{align}
	 \begin{split}\label{3}
	 &\sum_{i=|h|}^{\infty} (i+1)^{p(d-1)j}\, \PP(L\ge i)^{1/3} \le c_1 \sum_{i=|h|}^{\infty} (i+1)^{p(d-1)j}\, \exp\left(-\frac{i^2}{c_2}\right)\\
	 &\le c_1 \sum_{i=0}^{\infty} (i+1)^{p(d-1)j}\, \exp\left(-\frac{i+1}{c_2}\right)
	 \le c_3\, c_4^p\, (p(d-1)j + 2)! \sum_{i=0}^{\infty} (i+1)^{p(d-1)j - p(d-1)j -2}\\
	 &\le c_1\, c_2^p\, (p(d-1)j + 2)! \sum_{i=0}^{\infty} (i+1)^{-2}
	 = c_3\, c_4^p\, (p(d-1)j + 2)! \le c_5\, c_6^p\, (pdj)!\,.
	 \end{split}
	 \end{align}
	 Combining \eqref{1}, \eqref{2} and \eqref{3} we see that $T_1$ is bounded as follows:
	 \begin{align*}
	 T_1 &= c_1c_2^p \sum_{i=0}^{\infty} \sum_{m=h}^{\infty}\, (pj)!\, (i+1)^{p(d-1)j}\, (-m \lor 1)^{p(d-1)j}\, e^{p(m+1)j}\, \PP(L\ge i)^{1/3}\, \PP(H\ge m)^{1/3}\\
	 &= c_1\, c_2^p\, (pj)!\sum_{i=0}^{|h|} (i+1)^{p(d-1)j}\, \PP(L\ge i)^{1/3} \sum_{m=h}^{\infty} (-m \lor 1)^{p(d-1)j}\, e^{p(m+1)j}\, \PP(H\ge m)^{1/3}\\
	 & \quad + c_3\, c_4^p\, (pj)!\sum_{i=|h|}^{\infty} (i+1)^{p(d-1)j}\, \PP(L\ge i)^{1/3} \sum_{m=h}^{\infty} (-m \lor 1)^{p(d-1)j}\, e^{p(m+1)j}\, \PP(H\ge m)^{1/3}\\
	 &\le c_1\, c_2^p\, (pj)!\, (1 + |h|)^{p(d-1)j + 1}\, (pdj)! + c_3\, c_4^p\, (pj)!\, (pdj)!\,  (pdj)!\\
	 &\le  c_1\, c_2^p\, ((pdj)!)^2\,(pj)!\, (1 + |h|)^{p(d-1)j + 1}\,.
	 \end{align*}
	 Combining the estimates for $T_1$ and $T_2$ yields
	 \begin{equation}\label{eq:NMoment1}
	 \begin{split}
	 \EE[(N^{(\la)})^{pj}] &\le  c_1\, c_2^p\, ((pdj)!)^2\, (pj)!\, (1 + |h|)^{p(d-1)j + 1} + c_3\,c_4^p\, (pj)!\, (1 + |h|)^2\\
	 &\le  c_1\, c_2^p\, ((pdj)!)^{2}\, (pj)!\, (1 + |h|)^{p(d-1)j + d}\,.
	 \end{split}
	 \end{equation}
	 Because of \eqref{j=0} this bound clearly holds for the functional $\xi_{f_0}$, too. Finally, the additional exponential term appears in the same way as for the defect volume functional by conditioning on the event that the point $x=(0,h)$ belongs to the extreme points of $\Psi^{(\la)}$. This completes the proof of the first part of the proposition.
\end{proof}	 
	 
\begin{proof}[Proof of Proposition \ref{Momente} (ii)]	 
	Next, we turn to the second assertion and we consider the defect volume functional $\xi_V$ first. The first inequality is trivial, since without the constraint to the cylinder the value of the functional can only increase. Moreover, we get with H\"older's inequality and \eqref{UngleichungFakultät} that the second expectation is bounded from above by
	\begin{align*}
	&\prod_{i=1}^{k} \left(\EE\left[\xi_V^{(\lambda)}\left(x_i, \mathcal{P}^{(\lambda)}\right)\right]^{2k}\right)^{1/k}\\
	&\le \prod_{i=1}^{k} \left[c_1\, c_2^k\, ((2kd)!)^2\, (1+\left|h_i\right|)^{2k(d-1)+d}\, \exp\left(-\frac{e^{h_i\lor 0}}{c_3}\right) \right]^{1/k}\\
	&\le \prod_{i=1}^{k} \left[c_1\, c_2^k\, ((kd)!)^4\, (1+\left|h_i\right|)^{2k(d-1)+d}\, \exp\left(-\frac{e^{h_i\lor 0}}{c_3}\right) \right]^{1/k}\\
	&\le c_1\,c_2^k\, ((kd)!)^4\, \prod_{i=1}^{k} \left[(1+\left|h_i\right|)^{2(d-1)+{d\over k}}\, \exp\left(-\frac{e^{h_i\lor 0}}{c_3\, k}\right)\right]\\
	&\le c_1\,c_2^k\, ((kd)!)^4\, \prod_{i=1}^{k} \left[(1+\left|h_i\right|)^{2(d-1)+d}\, \exp\left(-\frac{e^{h_i\lor 0}}{c_3\, k}\right)\right]\\
	&\le c_1\,c_2^k\, ((kd)!)^4\, \prod_{i=1}^{k} \left[(1+\left|h_i\right|)^{4d}\, \exp\left(-\frac{e^{h_i\lor 0}}{c_3\, k}\right)\right]\,.
	\end{align*}
	The trivial estimate $2(d-1) + d \le 4d$ implies the last inequality. Finally, we consider the functionals $\xi_{f_j}$ with $j\in\{0,\ldots,d-1\}$. Instead of using the bound \eqref{eq:NMoment1}, it is more convenient to work with
	 \begin{align}\label{2.Versionjface}
	 \begin{split}
	 \EE[(N^{(\la)})^{pj}] &\le  c_1\, c_2^p\, ((pdj)!)^2\, (pj)!\, (1 + |h|)^{p(d-1)j + 1} + c_3\, (pj)!\, (1 + |h|)^2\\
	 &\le  c_1\, c_2^p\, ((pdj)!)^{2}\, (pj)!\, (1 + |h|)^{pdj}\,,
	 \end{split}
	 \end{align}
	which holds because of $p(d-1)j + 1 \le pdj$ if $j\ge 1$ and trivially also for $j=0$. Then a similar computation as for the volume functional completes the proof.
\end{proof}

\begin{remark}
	In the step before the last step in the proof of the previous proposition one could also use the better estimate $2(d-1) + d \le 3d$. But since we need an exponent out of the natural numbers after taking the square-root, we directly work with the upper bound $4d$.    
\end{remark}

The next clustering lemma is the analogue of Lemma 5.4 in \cite{GroteThäle}. The main difference and what makes it much more complicated is that in view of \eqref{Lokalisierung2} we do not have a localization property on the whole space for rescaled Gaussian polytopes. This leads to an additional indicator function in the bound for the correlation function, which also makes the analysis later more involved.

\begin{lem}\label{kluster}
	Let $\{S,T\}$ be a non-trivial partition of $\setk$ and $\xi\in \Xi$. Then, for all $x_1=(v_1,h_1), \ldots, x_k=(v_k,h_k)\in W_\lambda$ and sufficiently large $\la$ one has that 
	\begin{align*}
	&\left| m_\lambda(\bx_{S \cup T}) - m_\lambda(\bx_S) m_\lambda(\bx_T) \right|\\ 
	& \qquad \le c_1\, c_2^k\, k\, (ku[\xi])!\, ((kdv[\xi])!)^{2}\, \left(\exp(-c_3\, \delta) + \text{\bf 1}(\delta \le 2 \max\limits_{r\in S\cup T} \{\left|h_r\right|\})\right)\\
	&\qquad\qquad\qquad\qquad\times \prod_{r\in S\cup T} \left[(1+\left|h_r\right|)^{dw[\xi]} \exp\left(-\frac{e^{h_r\lor 0}}{c_4\, k}\right)\right] 
	\end{align*}
	with $	k=|S\cup T|$, $\delta:= d(\bv_S,\bv_T):= \min_{s\in S, t\in T} \|v_s-v_t\| $ and
	\begin{align*}
	m_\lambda(\bx_S) := \EE\left[\prod_{s\in S} \xi^{(\lambda)}\left(x_s, \mathcal{P}^{(\lambda)} \cup\bigcup_{s\in S}\{x_{s}\}\right)\right]\,,
	\end{align*} 
	and with $m_\la(\bx_T)$ and $m_\la(\bx_{S\cup T})$ defined similarly. 
\end{lem}
\begin{proof}
	Let us define the random variables 
	\begin{align*}
	X&:= \prod_{s\in S} \xi^{(\lambda)}\Big(x_s, \mathcal{P}^{(\lambda)} \cup\bigcup_{j\in S}\{x_{j}\}\Big)\,,\quad 
	Y:= \prod_{t\in T} \xi^{(\lambda)}\Big(x_t, \mathcal{P}^{(\lambda)} \cup\bigcup_{j\in T}\{x_{j}\}\Big)\,,\\
	W&:= \prod_{r\in S\cup T} \xi^{(\lambda)}\Big(x_r, \mathcal{P}^{(\lambda)} \cup\bigcup_{j\in S\cup T}\{x_{j}\}\Big)
	\end{align*}
	and
	\begin{align*}
	X_\delta&:= \prod_{s\in S} \xi^{(\lambda)}\Big(x_s, \Big(\mathcal{P}^{(\lambda)} \cup\bigcup_{j\in S}\{x_{j}\}\Big) \cap C_{d-1}\Big(x_s, \frac{\delta}{2}\Big)\Big)\,,\\
	Y_\delta&:= \prod_{t\in T} \xi^{(\lambda)}\Big(x_t, \Big(\mathcal{P}^{(\lambda)} \cup\bigcup_{j\in T}\{x_{j}\}\Big) \cap C_{d-1}\Big(x_t, \frac{\delta}{2}\Big)\Big)\,,\\
	\ W_\delta&:= \prod_{r\in S\cup T} \xi^{(\lambda)}\Big(x_r, \Big(\mathcal{P}^{(\lambda)} \cup\bigcup_{j\in S\cup T}\{x_{j}\}\Big) \cap C_{d-1}\Big(x_r, \frac{\delta}{2}\Big)\Big)\,.
	\end{align*}
	Similarly as in the proof of Lemma 5.4 in \cite{GroteThäle} we have that
	\begin{align}\label{eq:ZerlegungXYW}
	&m_\lambda(\bx_{S\cup T}) - m_\lambda(\bx_{S}) m_\lambda(\bx_{T})
	= \EE[W-W_\delta] - \EE[X_\delta]\EE[Y-Y_\delta] - \EE[Y]\EE[X-X_\delta]\,.
	\end{align}
	Now, observe that Proposition \ref{Momente} implies the estimates
	\begin{align*}
	\EE|X_\delta| 
	&\le c_1\, c_2^k\, (|S|u[\xi])!\, ((|S|dv[\xi])!)^{2}\, \prod_{s\in S} \left[(1+|h_s|)^{dw[\xi]}\, \exp\left(-\frac{e^{h_s\lor 0}}{c_3\, k}\right)\right]
	\intertext{and}
	\EE|Y| 
	&\le c_1\, c_2^k\, (|T|u[\xi])!\, ((|T|dv[\xi])!)^{2}\, \prod_{t\in T} \left[(1+|h_t|)^{dw[\xi]}\, \exp\left(-\frac{e^{h_t\lor 0}}{c_3\, k}\right)\right]\,,
	\end{align*}
	where we also used that $|S|, |T| \le k$. 
	Next, let $N_S$ be the event that at least one $x_s$, $s\in S$, has a radius of localization bigger than or equal to $\delta /2$. On the complement $N_S^c$ of $N_S$ we have $X_\delta = X$. 
	Thus, the Cauchy-Schwarz inequality, the fact that $\sqrt{a+b}\leq\sqrt{a}+\sqrt{b}$ for all $a,b\ge0$, Proposition \ref{Momente} and \eqref{Lokalisierung2} show that
	\begin{align*}
	&\EE|X-X_\delta|	= \EE|X_\delta\,  \text{\bf 1}(N_S)| \leq (\EE[X_\delta^2])^\frac{1}{2} (\PP(N_S))^\frac{1}{2}\\
	&\leq \left(  c_1\, c_2^k\, ((|S|u[\xi])!)^2\, ((|S|dv[\xi])!)^{4}\, \prod_{s\in S} \left[(1+|h_s|)^{2dw[\xi]}\, \exp\left(-\frac{e^{h_s\lor 0}}{c_3\, |S|}\right)\right]\right)^{1/2}\\
	& \qquad\qquad \times \left( c_4\ |S|\, \exp\left(-c_5\, \delta \right) + \text{\bf 1}(\delta \le 2 \max\limits_{s\in S} \{\left|h_s\right|\}) \right)^\frac{1}{2}\\
	&\leq c_1\, c_2^k\,  (|S|u[\xi])!\,  ((|S|dv[\xi])!)^{2}\, k\, \left(\exp\left(-c_3\, \delta \right) + \text{\bf 1}(\delta \le 2 \max\limits_{r\in S\cup T} \{\left|h_r\right|\})\right) \\
	&\qquad\qquad\times \prod_{s\in S} \left[(1+|h_s|)^{dw[\xi]}\,  \exp\left(-\frac{e^{h_s\lor 0}}{c_4\, k}\right)\right] \,,
	\end{align*}  
	since
	$$
	\text{\bf 1}(\delta \le 2 \max\limits_{s\in S} \{\left|h_s\right|\}) \le \text{\bf 1}(\delta \le 2 \max\limits_{r\in S\cup T} \{\left|h_r\right|\})
	$$
	and $|S|\le k$. From \eqref{UngleichungFakultät3} we see that
	$$
	(|S|dv[\xi])!\,  (|T|dv[\xi])! \le (kdv[\xi])!\qquad\text{and}\qquad(|T|u[\xi])!\, (|S|u[\xi])! \le (ku[\xi])!\,,
	$$
	which leads to
	\begin{align*}
	\EE|Y|\, \EE|X-X_\delta|
	&\le  c_1\, c_2^k\, (ku[\xi])!\, ((kdv[\xi])!)^{2}\, k\, \left(\exp\left(-c_3\, \delta \right) + \text{\bf 1}(\delta \le 2 \max\limits_{r\in S\cup T} \{\left|h_r\right|\})\right) \\
	&\qquad\times \prod_{r\in S\cup T} \left[(1+|h_r|)^{dw[\xi]}\, \exp\left(-\frac{e^{h_r\lor 0}}{c_4\, k}\right)\right] 
	\end{align*}
	for sufficiently large $\lambda$. Similar estimates hold for $\EE|X_\delta|\, \EE|Y-Y_\delta|$ and $\EE|W-W_\delta|$ and this completes the proof in view of \eqref{eq:ZerlegungXYW}.
\end{proof}

\section{Proof of the cumulant bound}\label{sec:CumulantProof}

This section contains the most technical part of the proof of our main results, that is, the proof of Theorem \ref{cumulantestimate}. Before going into the details, let us briefly describe the main steps. The starting point is the cluster measure representation of the cumulant measures presented in Lemma \ref{lem:CumulantExpressionCluster} and in a first step we will deal with the diagonal term (Lemma \ref{Diagonal}). Using \eqref{cumulantmeasure}, we see that for this we just need to control the moments of $\xi$. We have prepared such a bound in the first part of Proposition \ref{Momente}. 

In a second and considerably more involved step of the proof we deal with the off-diagonal term. The cluster measure representation of the cumulant measures presented in Lemma \ref{lem:CumulantExpressionCluster}, the description of spatial correlations from the clustering Lemma \ref{kluster} as well as the second part of our moment estimates in Proposition \ref{Momente} allow us to derive a first integral representation for the individual terms (Lemma \ref{lem:FirstBoundOffDiag}). These integrals are then estimated next, starting with the inner integral (Lemma \ref{lem:IntDV}). The fact that the geometric functionals $\xi\in\Xi$ are not globally localizing implies that this estimate has two terms that need to be investigated further (Lemma \ref{I} and Lemma \ref{II}). Combining all bounds finally results in a bound for the off-diagonal term (Lemma \ref{off-diagonal}).

\medskip

As described above, we start by dealing with the diagonal term. From now on we fix $k\in\{3,4,\ldots\}$.

\begin{lem}\label{Diagonal}
Let $\xi\in \Xi$ and $f\in \mathcal{B}(\RR^d)$. Then,
\begin{align*}
\Bigg| \int\limits_\Delta f_{R_\la}^k\, \dint c_\lambda^k \Bigg| \le c_1\, c_2^k\, \|f\|_\infty^k\, R_\lambda^{d-1}\,  (ku[\xi])!\,((kdv[\xi])!)^{3}
\end{align*}
for sufficiently large $\lambda$.
\end{lem}
\begin{proof}
As in the proof of Lemma 5.6 in \cite{GroteThäle} we obtain from the definition of the cumulant measures and the fact that we only integrate over the diagonal $\Delta$ that
\begin{align*}
\Bigg|\int\limits_\Delta f_{R_\la}^k\; \dint c_\lambda^k\Bigg| \leq \|f\|_\infty^k \int\limits_{\mathbb{R}^d} \EE \big|\xi( x, \mathcal{P}_\la)\big|^k\,\lambda\, \phi(x)\, \dint x\,.
\end{align*}
By rotation invariance of the point process one has that
\begin{align*}
\EE \big|\xi( x, \mathcal{P}_\la)\big|^k = \EE\big| \xi^{(\lambda)}( (0,h), \mathcal{P}^{(\lambda)})\big|^k\,,
\end{align*}
where $h$ is defined by $\|x\| = R_\lambda (1 - h/R_\lambda^2)$ in view of the scaling transformation $T_\la$. Writing $u=x/\|x\|$ we can rewrite $\dint x$ as
\begin{align*}
\dint x = \left[R_\la\left(1-\frac{h}{R_\la^2}\right)\right]^{d-1} \frac{1}{R_\la}\, \dint h\, \dint \sigma_{d-1}(u)\,.
\end{align*}
So, the above integral is bounded by 
\begin{align*}
\|f\|_\infty^k\, R_\la^{d-1} \int\limits_{\SS^{d-1}} \int\limits_{-\infty}^{R_\la^2}\, \EE\left| \xi^{(\lambda)}\left( (0,h), \mathcal{P}^{(\lambda)} \right)\right|^k \phi_\la(u,h)  \left(1-\frac{h}{R_\la^2}\right)^{d-1} \dint h \dint \sigma_{d-1}(u)
\end{align*}
with 
\begin{align}\label{eq:DefPhiLambda2Variablen}
\phi_\la(u,h) := \frac{\la}{R_\la}\, \phi\left(u\, R_\la\left(1-\frac{h}{R_\la^2}\right)\right) = \frac{\sqrt{2 \log \la}}{R_\la}\, \exp \left(h-\frac{h^2}{2R_\la^2}\right)\,,
\end{align}
see Remark \ref{rem:WalVonRl}.
The expression for $\phi_\la(u,h)$ is obviously bounded from above by a constant times $e^h$ for all $h\in \RR$, $u\in \SS^{d-1}$ and sufficiently large $\la$.  
Furthermore, from Proposition \ref{Momente} we deduce that 
\begin{align*}
&\EE \big| \xi^{(\lambda)}( (0,h), \mathcal{P}^{(\lambda)} ) \big|^k \\
&\qquad\le c_1\, c_2^k\, (ku[\xi])!\, ((kdv[\xi])!)^{2}\, (1+|h|)^{k(d-1)v[\xi] + d}\, \exp\left(-\frac{e^{h\lor 0}}{c_3}\right)\,.
\end{align*}
Thus, the integral we started with is bounded by
\begin{align*}
& c_1\, c_2^k\, \|f\|_\infty^k\, R_\la^{d-1}\,  (ku[\xi])!\, ((kdv[\xi])!)^{2}\,\int\limits_{\SS^{d-1}} \int\limits_{-\infty}^{R_\la^2}  (1+|h|)^{k(d-1)v[\xi] + d} \exp\left(-\frac{e^{h\lor 0}}{c_3}\right)\\
& \qquad\qquad\times e^h \left(1-\frac{h}{R_\la^2}\right)^{d-1} \dint h \dint \sigma_{d-1}(u)\,.
\end{align*}
We decompose the inner integral into the two parts
\begin{align*}
&  \int\limits_{-\infty}^{0}  (1+|h|)^{k(d-1)v[\xi] + d} \exp\left(-\frac{e^{h\lor 0}}{c_1}\right) e^h \left(1-\frac{h}{R_\la^2}\right)^{d-1} \dint h\\ 
&\qquad \qquad +\int\limits_{0}^{R_\la^2}  (1+|h|)^{k(d-1)v[\xi] + d} \exp\left(-\frac{e^{h\lor 0}}{c_2}\right) e^h \left(1-\frac{h}{R_\la^2}\right)^{d-1} \dint h =: T_3 + T_4\,,
\end{align*}
which will be treated separately. For the first one we have that $\exp\left(-\frac{e^{h\lor 0}}{c}\right) \le 1$ since $h<0$. Using that $R_\la^2 \geq  1$ for sufficiently large $\la$, the inequality $k(d-1)v[\xi] +d+(d-1) \le kdv[\xi] + 2d$ and the substitution $h+1=s$ we obtain
\begin{align*}
T_3 &\le \int\limits_{-\infty}^{0}  (1+|h|)^{k(d-1)v[\xi] + d}\, e^h \left(1-\frac{h}{R_\la^2}\right)^{d-1} \dint h = \int\limits_{0}^{\infty}  (1+h)^{k(d-1)v[\xi] + d}\,  e^{-h} \left(1+\frac{h}{R_\la^2}\right)^{d-1} \dint h\\
&\le \int\limits_{0}^{\infty}  (1+h)^{kdv[\xi] + 2d}\,  e^{-h} \dint h
\le  c_1 \int\limits_{0}^{\infty}  s^{kdv[\xi] + 2d}\,  e^{-s} \dint s = c_2\, (kdv[\xi] + 2d)! \le c_3\, c_4^k\, (kdv[\xi])! \,,
\end{align*}
where we also used \eqref{UngleichungFakultät2} in the last step.

To bound the term $T_4$ we need the inequality 
\begin{align}
\label{InequalityExponential}
\exp\left(-x\right) \le \frac{2}{x^2}\,,\qquad x>0\,.
\end{align} 
Furthermore, we observe that $\left(1- h/R_\la^2\right)^{d-1} \le 1$ for all $h\in \left[0,R_\la^2\right]$. With the help of \eqref{InequalityExponential} and \eqref{UngleichungFakultät2} we get the estimate
\begin{align*}
T_4 &= \int\limits_{0}^{R_\la^2}  (1+|h|)^{k(d-1)v[\xi] + d}\, \exp\left(-\frac{e^{h\lor 0}}{c}\right) e^h \left(1-\frac{h}{R_\la^2}\right)^{d-1} \dint h\\
&\le \int\limits_{0}^{R_\la^2}  (1+h)^{kdv[\xi] + d}\, \frac{2}{\left(e^h/c\right)^2}\, e^h\,  \dint h\\
&\le c_1 \int\limits_{0}^{\infty}  (1+h)^{kdv[\xi] + d}\, e^{-h}\, \dint h
\le c_2\, (kdv[\xi] + d)!
\le c_3\, c_4^k\, (kdv[\xi])!\,.   
\end{align*}
Putting together the bounds for $T_3$ and $T_4$ and using that $\int\limits_{\SS^{d-1}} \dint \sigma_{d-1}(u) = d\kappa_d$ implies 
\begin{align*}
\Bigg| \int\limits_\Delta f_{R_\la}^k\, \dint c_\lambda^k \Bigg| &\le c_1\, c_2^k\, \|f\|_\infty^k\, R_\la^{d-1}\,   (ku[\xi])!\, (kdv[\xi])!\,((kdv[\xi])!)^{2}\, \int\limits_{\SS^{d-1}} \dint \sigma_{d-1}(u)\\
&\le c_1\, c_2^k\, \|f\|_\infty^k\, R_\la^{d-1}\,  (ku[\xi])!\,((kdv[\xi])!)^{3}\,,
\end{align*}
which proves the claim.
\end{proof}

\noindent An upper bound for the off-diagonal term in \eqref{ZerlegungKumulanten} is derived along the following four Lemmas.  

\begin{lem}\label{lem:FirstBoundOffDiag}
	Let $\xi\in \Xi$ and $f\in\cB(\RR^d)$. Then, we have that, for sufficiently large $\la$,
	\begin{align*}
	&\Bigg|\sum_{S,T\, \preceq\setk}\,\int\limits_{\sigma(\{S,T\})} f_{R_\la}^k\, \dint c_\la^k\, \Bigg|\\
	 &\qquad \leq c_1\, c_2^k\, \|f\|_\infty^k\,R_\la^{d-1}\,  k\, k!\,(ku[\xi])!\,  ((kdv[\xi])!)^{2}\\
	 &  \qquad \quad   \times \sum_{L_1,\ldots,L_p\, \preceq\setk}\,  \int\limits_{\SSd} \int\limits_{-\infty}^{R_\la^2} \ldots \int\limits_{-\infty}^{R_\la^2} \int\limits_{(\RR^{d-1})^{p-1}} \left(\exp(-c_3\, \delta(0,\text{\bf v})) + \text{\bf 1}(\delta(0,\text{\bf v}) \le 2 \max\limits_{i=1,\ldots,p} \{\left|h_i\right|\})\right)\\
	& \qquad \quad  \times \prod_{i=1}^p \left[(1+\left|h_i\right|)^{dw[\xi]} \exp\left(-\frac{e^{h_i\lor 0}}{c_4\, k}\right) e^{h_i} \left(1-\frac{h_i}{R_\la^2}\right)^{d-1}\right] \dint \text{\bf v} \dint h_1 \ldots \dint h_p \dint \sigma_{d-1}(u)\,,
	\end{align*}
	where ${\bf v}=(v_2,\ldots,v_p)$.
\end{lem}
\begin{proof}
The definition of the cluster measures and the description of the densities of the moment measures as explained in Section \ref{sec:Preliminaries} imply that for a fixed and non-trivial partition $\{S,T\}$ of $\setk$ and for fixed $S',T',K_1,\ldots,K_s$ as in Lemma \ref{lem:CumulantExpressionCluster},
 \begin{align}
 \begin{split}
 \label{Integrand}
 &\int\limits_{\delta(\{S,T\})} f_{R_\la}^k\ \dint (U_\lambda^{S^{'},T^{'}}\otimes M_\lambda^{|K_1|}\otimes\ldots \otimes M_\lambda^{|K_s|})\\
 &= \int\limits_{\delta(\{S,T\})} f_{R_\la}^k(\text{\bf x})\ \left(m_\lambda(\bx_{S^{'}\cup T^{'}}) - m_\lambda(\bx_{S^{'}})m_\lambda(\bx_{T^{'}})\right) m_\lambda(\bx_{K_1})\ldots m_\lambda(\bx_{K_s})\ \tilde{\dint}[\lambda \phi](\bf x)\,.
 \end{split}
 \end{align} 
 In what follows we will use the parametrization $T_\la(x_i) = (v_i,h_i)$ for all $i\in\setk$. Using this notation, Lemma \ref{kluster} shows that
 \begin{align*}
 &|m_\lambda(\bx_{S^{'}\cup T^{'}}) - m_\lambda(\bx_{S^{'}})m_\lambda(\bx_{T^{'}})|\\ 
 &\qquad \le c_1\, c_2^k\, k\, (|S^{'}\cup T^{'}|u[\xi])!\, ((|S^{'}\cup T^{'}|dv[\xi])!)^{2} \\
 &\qquad\qquad\times \left(\exp(-c_3\, d(\bv_{S^{'}}, \bv_{T^{'}})) + \text{\bf 1}(d(\bv_{S^{'}}, \bv_{T^{'}}) \le 2 \max\limits_{r\in S^{'}\cup T^{'}} \{\left|h_r\right|\})\right)\\
 &\qquad \qquad \times  \prod_{r\in S^{'}\cup T^{'}} \left[(1+\left|h_r\right|)^{dw[\xi]} \exp\left(-\frac{e^{h_r\lor 0}}{c_4\, k}\right)\right]\,.
 \end{align*}
 On the other hand, Proposition \ref{Momente} delivers for all $i\in \{1,\ldots,s\}$ the bound
 \begin{align*}
 |m_\lambda(\bx_{K_i})|
 & \le  c_1\, c_2^k\, (|K_i|u[\xi])!\, ((|K_i|dv[\xi])!)^{2}\,\prod_{i\in K_i} \left[(1+\left|h_i\right|)^{dw[\xi]} \exp\left(-\frac{e^{h_i\lor 0}}{c_3\, k}\right)\right]\,,
 \end{align*}
 since $|K_i|\le k$. Now, we notice that $d(\bv_{S'},\bv_{T'})\geq d(\bv_S,\bv_T)$. Together with the observation that $\max\limits_{r\in S^{'}\cup T^{'}} \{\left|h_r\right|\} \le \max\limits_{i=1,\ldots,k} \{\left|h_i\right|\}$ this yields that
 \begin{align*}
 &\exp(-c\, d(\bv_{S^{'}}, \bv_{T^{'}})) + \text{\bf 1}(d(\bv_{S^{'}}, \bv_{T^{'}})\le 2 \max\limits_{r\in S^{'}\cup T^{'}} \{\left|h_r\right|\})\\
 & \qquad \le \exp(-c\, d(\bv_{S}, \bv_{T})) + \text{\bf 1}(d(\bv_{S}, \bv_{T}) \le 2 \max\limits_{i=1,\ldots,k} \{\left|h_i\right|\}) \,,
 \end{align*}
which in turn implies that
 \begin{align*}
 &|\left(m_\lambda(\bx_{S^{'}\cup T^{'}}) - m_\lambda(\bx_{S^{'}})m_\lambda(\bx_{T^{'}})\right) m_\lambda(\bx_{K_1})\cdots m_\lambda(\bx_{K_s})|\\
 & \le c_1\,c_2^k\, k\, (|S^{'}\cup T^{'}|u[\xi])!\, (|K_1|u[\xi])! \cdots (|K_s|u[\xi])!\, \big[(|S^{'}\cup T^{'}|dv[\xi])!\\
 & \qquad \times (|K_1|dv[\xi])! \cdots (|K_s|dv[\xi])!\big]^{2}\,\left(\exp(-c_3\, d(\bv_{S}, \bv_{T})) + \text{\bf 1}(d(\bv_{S}, \bv_{T}) \le 2 \max\limits_{i=1,\ldots,k} \{\left|h_i\right|\})\right)\\
 &\qquad\times \prod_{i = 1}^{k} \left[(1+\left|h_i\right|)^{dw[\xi]} \exp\left(-\frac{e^{h_i\lor 0}}{c_4\, k}\right)\right]\,.
 \end{align*}
 Now, we use the estimate \eqref{UngleichungFakultät3} to see that
 \begin{align*}
 (|S^{'}\cup T^{'}|dv[\xi])!\, (|K_1|dv[\xi])! \cdots (|K_s|dv[\xi])! &\le (kdv[\xi])! \qquad \text{and}\\
 (|S^{'}\cup T^{'}|u[\xi])!\, (|K_1|u[\xi])! \cdots (|K_s|u[\xi])! &\le (ku[\xi])!\,.
 \end{align*}
Thus,
 \begin{align*}
 &|\left(m_\lambda(\bx_{S^{'}\cup T^{'}}) - m_\lambda(\bx_{S^{'}})m_\lambda(\bx_{T^{'}})\right) m_\lambda(\bx_{K_1})\cdots m_\lambda(\bx_{K_s})|\\
 &\qquad  \le c_1\, c_2^k\, k\, (ku[\xi])!\, ((kdv[\xi])!)^{2}\, \left(\exp(-c_3\, d(\bv_{S}, \bv_{T})) + \text{\bf 1}(d(\bv_{S}, \bv_{T}) \le 2 \max\limits_{i=1,\ldots,k} \{\left|h_i\right|\})\right)\\
 &\qquad \qquad \times \prod_{i = 1}^{k} \left[(1+\left|h_i\right|)^{dw[\xi]} \exp\left(-\frac{e^{h_i\lor 0}}{c_4\, k}\right)\right]\,.
 \end{align*}
 Recalling Lemma \ref{lem:CumulantExpressionCluster} and especially \eqref{AbschätzungSummanden}, summing \eqref{Integrand} over all $S',T',K_1,\ldots,K_s\preceq\setk$ and observing that $d(\bv_{S}, \bv_{T}) = \delta(\bf x)$ whenever we are integrating over $\delta(\{S,T\})$, we get
 \begin{align*}
  \Bigg|\,\int\limits_{\delta(\{S,T\})} f_{R_\la}^k\, \dint c_\lambda^k \Bigg|  &\le c_1\,c_2^k\, \|f\|_\infty^k\, k\, k!\,(ku[\xi])!\,  ((kdv[\xi])!)^{2} \\
 &\qquad\times\int\limits_{\delta(\{S,T\})} \left(\exp(-c_3\, \delta(\text{\bf x})) + \text{\bf 1}(\delta(\text{\bf x}) \le 2 \max\limits_{i=1,\ldots,k} \{\left|h_i\right|\})\right) \\
 &\qquad\times\prod_{i = 1}^{k} \left[(1+\left|h_i\right|)^{dw[\xi]} \exp\left(-\frac{e^{h_i\lor 0}}{c_4\, k}\right)\right] \tilde{\dint}[\lambda \phi]({\bf x})
 \end{align*} 
 and finally 
  \begin{align*}
  \Bigg|\, \sum_{S,T\, \preceq\setk}\, \int\limits_{\delta(\{S,T\})} f_{R_\la}^k\ \dint c_\lambda^k \Bigg|  &\le c_1\,c_2^k\, \|f\|_\infty^k\,  k\, k!\,(ku[\xi])!\,  ((kdv[\xi])!)^{2}\\
  &\qquad\times \int\limits_{(\RR^d)^k} \left(\exp(-c_3\, \delta(\text{\bf x})) + \text{\bf 1}(\delta(\text{\bf x}) \le 2 \max\limits_{i=1,\ldots,k} \{\left|h_i\right|\})\right)\\
  &\qquad  \times \prod_{i = 1}^{k} \left[(1+\left|h_i\right|)^{dw[\xi]} \exp\left(-\frac{e^{h_i\lor 0}}{c_4\, k}\right)\right] \tilde{\dint}[\lambda \phi]({\bf x})\,.
  \end{align*}
  
  In a next step, a bound for the integral over $(\RR^d)^k$ is derived and for this we can assume without loss of generality that, after a suitable rotation of the underlying point process, the point $x_1$ is mapped to $(0,h_1) \in W_\la$ under the scaling transformation $T_\la$, where the height coordinate $h_1$ is determined by $\|x_1\| = R_\la(1 - h_1/R_\la^2)$ as in the previous lemma. Together with the definition of the singular differential $\tilde{\dint}[\lambda \phi]({\bf x})$ we conclude that
  \begin{align*}
  &\int\limits_{(\RR^d)^k} \left(\exp(-c_1\, \delta(\text{\bf x})) + \text{\bf 1}(\delta(\text{\bf x}) \le 2 \max\limits_{i=1,\ldots,k} \{\left|h_i\right|\})\right) \prod_{i = 1}^{k} \left[(1+\left|h_i\right|)^{dw[\xi]} \exp\left(-\frac{e^{h_i\lor 0}}{c_2\, k}\right)\right] \tilde{\dint}[\lambda \phi](\text{\bf x})\\
  &= \sum_{L_1,\ldots,L_p\, \preceq\setk}\, \int\limits_{(\RR^d)^k} \left(\exp(-c_1\, \delta(\text{\bf x})) + \text{\bf 1}(\delta(\text{\bf x}) \le 2 \max\limits_{i=1,\ldots,k} \{\left|h_i\right|\})\right)\\
  &\qquad \qquad \times \prod_{i = 1}^{k} \left[(1+\left|h_i\right|)^{dw[\xi]} \exp\left(-\frac{e^{h_i\lor 0}}{c_2\, k}\right)\right] \bar{\dint}[\lambda\phi](\bx_{L_1})\ldots \bar{\dint}[\lambda\phi](\bx_{L_p})\\
  &= \sum_{L_1,\ldots,L_p\, \preceq\setk}\, \la^p \int\limits_{(\RR^d)^p} \left(\exp(-c_1\, \delta(0,v_2,\ldots,v_p)) + \text{\bf 1}(\delta(0,v_2,\ldots,v_p) \le 2 \max\limits_{i=1,\ldots,p} \{\left|h_i\right|\})\right)\\
  &\qquad \qquad \times \prod_{i = 1}^{p} \left[(1+\left|h_i\right|)^{dw[\xi]} \exp\left(-\frac{e^{h_i\lor 0}}{c_2\, k}\right)\right] \phi(x_1) \ldots \phi(x_p)\, \dint x_1 \ldots \dint x_p\,.
  \end{align*}  
  Now, we reparameterize as in the proof of Lemma \ref{Diagonal} and notice that the differential elements transform into
  \begin{align*}
  \la\, \phi(x_1)\, \dint x_1 = R_\la^{d-1}\, \phi_\la(u,h_1)\, \left(1-\frac{h_1}{R_\la^2}\right)^{d-1} \dint h_1 \dint \sigma_{d-1}(u)
  \end{align*}
 and  
  \begin{align*}
  \la\, \phi(x_i)\, \dint x_i &=  \frac{\sin^{d-2} (R_\la^{-1}\|v_i\|)}{\|R_\la^{-1}v_i\|^{d-2}}\, \frac{\sqrt{2\log \la}}{R_\la}\, \exp\left(h_i-\frac{h_i^2}{2R_\la^2}\right) \left(1-\frac{h_i}{R_\la^2}\right)^{d-1} \dint v_i \dint h_i\\
  &= \frac{\sin^{d-2} (R_\la^{-1}\|v_i\|)}{\|R_\la^{-1}v_i\|^{d-2}}\, \phi_\la(u,h_i)\, \left(1-\frac{h_i}{R_\la^2}\right)^{d-1} \dint v_i \dint h_i
  \end{align*}
  for $i=2,\ldots,p$ with $\phi_\la(u,h)$ defined at \eqref{eq:DefPhiLambda2Variablen}, see also \cite{CalkaYukich}. The fractions involving the sinus term are bounded from above by $1$. Moreover, one has that $\phi_\la(u,h_i) \le c\, e^{h_i}$ for all $i\in\{1,\ldots,p\}$ and sufficiently large $\la$ as in the proof of Lemma \ref{Diagonal}. This implies that
  \begin{align*}
  \la^p\, \phi(x_1) \ldots \phi(x_p)\, \dint x_1 \ldots \dint x_p \le c^p\, R_\la^{d-1}\, \prod_{i=1}^{p} \left[e^{h_i} \left(1-\frac{h_i}{R_\la^2}\right)^{d-1}\right] \dint h_1 \ldots \dint h_p \dint v_2 \ldots \dint v_p \dint \sigma_{d-1}(u)
  \end{align*}
  and finally 
  { \allowdisplaybreaks
  \begin{align*}
 & \la^p \int\limits_{(\RR^d)^p} \left(\exp(-c_1\, \delta(0,v_2,\ldots,v_p)) + \text{\bf 1}(\delta(0,v_2,\ldots,v_p) \le 2 \max\limits_{i=1,\ldots,p} \{\left|h_i\right|\})\right)\\
  &\qquad \qquad \times \prod_{i = 1}^{p} \left[(1+\left|h_i\right|)^{dw[\xi]} \exp\left(-\frac{e^{h_i\lor 0}}{c_2\, k}\right)\right] \phi(x_1) \ldots \phi(x_p)\, \dint x_1 \ldots \dint x_p\\
  &\le c_1^p\,R_\la^{d-1} \int\limits_{\SSd} \int\limits_{-\infty}^{R_\la^2} \ldots \int\limits_{-\infty}^{R_\la^2}\, \int\limits_{T_\la(\SSd)} \ldots \int\limits_{T_\la(\SSd)}\\
  &\qquad \qquad \times \left(\exp(-c_2\, \delta(0,v_2,\ldots,v_p)) + \text{\bf 1}(\delta(0,v_2,\ldots,v_p) \le 2 \max\limits_{i=1,\ldots,p} \{\left|h_i\right|\})\right)\\
  &\qquad \qquad \times \prod_{i = 1}^{p} \left[(1+\left|h_i\right|)^{dw[\xi]} \exp\left(-\frac{e^{h_i\lor 0}}{c_3\, k}\right) e^{h_i} \left(1-\frac{h_i}{R_\la^2}\right)^{d-1}\right] \dint v_2 \ldots \dint v_p \dint h_1 \ldots \dint h_p \dint \sigma_{d-1}(u)\\
  &\le c_1^p\,R_\la^{d-1} \int\limits_{\SSd} \int\limits_{-\infty}^{R_\la^2} \ldots \int\limits_{-\infty}^{R_\la^2}\, \int\limits_{(\RR^d)^{p-1}} \left(\exp(-c_2\, \delta(0,\text{\bf v})) + \text{\bf 1}(\delta(0,\text{\bf v}) \le 2 \max\limits_{i=1,\ldots,p} \{\left|h_i\right|\})\right)\\
  &\qquad \qquad \times \prod_{i = 1}^{p} \left[(1+\left|h_i\right|)^{dw[\xi]} \exp\left(-\frac{e^{h_i\lor 0}}{c_3\, k}\right) e^{h_i} \left(1-\frac{h_i}{R_\la^2}\right)^{d-1}\right] \dint \text{\bf v} \dint h_1 \ldots \dint h_p \dint \sigma_{d-1}(u)
  \end{align*}}
  with $\text{\bf v}:= (v_2,\ldots,v_p)$. This yields the desired result. 
\end{proof}

The previous lemma shows that we already have separated the crucial prefactor $R_\la^{d-1}$. In the next steps we will appropriately bound the remaining integrals and we start with the inner integral concerning the integration with respect to the vector $\text{\bf v}$. 

\begin{lem}\label{lem:IntDV}
In the situation of Lemma \ref{lem:FirstBoundOffDiag} it holds that 
\begin{align*}
&\int\limits_{(\RR^d)^{p-1}} \left(\exp(-c\, \delta(0,\text{\bf v})) + \text{\bf 1}(\delta(0,\text{\bf v}) \le 2 \max\limits_{i=1,\ldots,p} \{\left|h_i\right|\})\right) \dint \text{\bf v}\\
&\qquad \qquad  \le c_1\,c_2^p\, p^{p-2} \left((dp)! + \Big(\max\limits_{i=1,\ldots,p} \{|h_i|\}\Big)^{(d-1)(p-1)}\right).
\end{align*}
\end{lem}
\begin{proof}
We divide the proof into two parts. At first we consider the integral regarding the exponential function. Using that
\begin{align*}
\int\limits_{\delta(0,\text{\bf v})}^{\infty} \exp(-c\, t) \,\dint t = \frac{1}{c}\, \exp(- c\, \delta(0,\text{\bf v}))
\end{align*}
together with Fubini`s theorem we obtain
\begin{align*}
\int\limits_{(\RR^d)^{p-1}} \exp(-c\, \delta(0,\text{\bf v})) \, \dint \text{\bf v} &= c \int\limits_{(\RR^d)^{p-1}}\, \int\limits_{\delta(0,\text{\bf v})}^{\infty} \exp(-c\, t) \, \dint t \dint \text{\bf v} \\
&= c  \int\limits_{0}^{\infty} \int\limits_{\{\delta(0,\text{\bf v}) < t\}} \dint \text{\bf v} \exp(-c\, t) \, \dint t\,.
\end{align*}
Similar computations as in the proof of Lemma 5.9 in \cite{GroteThäle} give
\begin{align}\label{Baum}
\int\limits_{\{\delta(0,\text{\bf v}) < t\}} \dint \text{\bf v} \le p^{p-2}\, \kappa_{d-1}^{p-1}\, t^{(d-1)(p-1)}\,.
\end{align}
Furthermore, integration by parts shows that
\begin{align*}
\int\limits_0^\infty t^{(d-1)(p-1)} \exp(-c\, t) \,\dint t = \frac{((d-1)(p-1))!}{c^{(d-1)(p-1) + 1}} \le c_1\, c_2^p\, (dp)!
\end{align*}
and this leads to
\begin{align*}
\int\limits_{(\RR^d)^{p-1}} \exp(-c_1\, \delta(0,\text{\bf v})) \, \dint \text{\bf v} &\le c_2\, p^{p-2}\, \kappa_{d-1}^{p-1}\, \int\limits_0^\infty t^{(d-1)(p-1)} \exp(-c_3\, t) \, \dint t\\
&\le c_1\, c_2^p\, \kappa_{d-1}^{p-1}\, p^{p-2}\, (dp)!\,.
\end{align*}
For the second part of the integral we analyze one has, again by \eqref{Baum}, that
\begin{align*}
\int\limits_{(\RR^d)^{p-1}} \text{\bf 1}(\delta(0,\text{\bf v}) \le 2 \max\limits_{i=1,\ldots,p} \{\left|h_i\right|\})\, \dint \text{\bf v}
&=\int\limits_{\{\delta(0,{\bf v})\leq 2\max_{i=1,\ldots,p}\{|h_i|\}\}}\dint{\bf v}\\
&\le c_1\,c_2^p\, p^{p-2}\, \kappa_{d-1}^{p-1}\, \left(\max\limits_{i=1,\ldots,p} \{\left|h_i\right|\}\right)^{(d-1)(p-1)}\,.
\end{align*} 
Combining both estimates gives the result. 
\end{proof}

The two last lemmas show that we are left with the bound
\begin{align}\label{eq:UpperBoundT4T6}
 \bigg|\, \sum_{S,T\, \preceq\setk}\, \int\limits_{\delta(\{S,T\})} f_{R_\la}^k\ \dint c_\lambda^k \bigg| \le T_5 + T_6\,,
 \end{align}
 where the terms $T_5$ and $T_6$ are given by
 \begin{align}
 \begin{split}\label{Leftover1}
 T_5 &:= c_1\, c_2^k\  \|f\|_\infty^k\, R_\la^{d-1}\, k\, k!\,(dk)!\, (ku[\xi])!\,  ((kdv[\xi])!)^{2} k^{k-2}  \sum_{L_1,\ldots,L_p\, \preceq\setk}\, \int\limits_{\SSd} \int\limits_{-\infty}^{R_\la^2} \ldots \int\limits_{-\infty}^{R_\la^2}\,\\
 &\qquad  \times \prod_{i = 1}^{p} \left[(1+\left|h_i\right|)^{dw[\xi]} \exp\left(-\frac{e^{h_i\lor 0}}{c_3\ k}\right) e^{h_i} \left(1-\frac{h_i}{R_\la^2}\right)^{d-1}\right] \dint h_1 \ldots \dint h_p \dint \sigma_{d-1}(u)\,,
 \end{split}
 \end{align}
 and
 \begin{align}
 \begin{split}\label{Leftover2}
 T_6 &:= c_1\, c_2^k\ \|f\|_\infty^k\, R_\la^{d-1}\, k\, k!\,(ku[\xi])!\,  ((kdv[\xi])!)^{2}\, k^{k-2} \\
 &\qquad\times\sum_{L_1,\ldots,L_p\, \preceq\setk}\, \int\limits_{\SSd} \int\limits_{-\infty}^{R_\la^2} \ldots \int\limits_{-\infty}^{R_\la^2}\,\left(\max\limits_{i=1,\ldots,p} \{\left|h_i\right|\}\right)^{(d-1)(p-1)} \\
 &\qquad \times \prod_{i = 1}^{p} \left[(1+\left|h_i\right|)^{dw[\xi]} \exp\left(-\frac{e^{h_i\lor 0}}{c_3\, k}\right) e^{h_i} \left(1-\frac{h_i}{R_\la^2}\right)^{d-1}\right] \dint h_1 \ldots \dint h_p \dint \sigma_{d-1}(u)\,.
\end{split}
\end{align}
Here, we used in several places that $p\leq k$.

\begin{lem}\label{I}
	For $T_5$ as defined by \eqref{Leftover1} we have that
	\begin{align*}
	|T_5| \le c_1\, c_2^k\, \|f\|_\infty^k\, R_\la^{d-1}\, (dk)!\, (ku[\xi])!\, (k!)^2\, ((kdv[\xi])!)^{2}\,  k^{3k}\,.
	\end{align*}
\end{lem}
\begin{proof}
We start with the integral concerning the coordinate $h_1$. Similarly to the computations performed in the proof of Lemma \ref{Diagonal} and by using \eqref{InequalityExponential} we see that
{ \allowdisplaybreaks
\begin{align*}
 &\int\limits_{-\infty}^{R_\la^2}  (1+|h_1|)^{dw[\xi]}\, \exp\left(-\frac{e^{h_1\lor 0}}{c_1\, k}\right) e^{h_1} \left(1-\frac{h_1}{R_\la^2}\right)^{d-1} \dint h_1\\
 &\qquad =  \int\limits_{-\infty}^{0}  (1+|h_1|)^{dw[\xi]}\, \exp\left(-\frac{e^{h_1\lor 0}}{c_1\, k}\right) e^{h_1} \left(1-\frac{h_1}{R_\la^2}\right)^{d-1} \dint h_1\\
 &\qquad \qquad + \int\limits_{0}^{R_\la^2}\,  (1+|h_1|)^{dw[\xi]} \exp\left(-\frac{e^{h_1\lor 0}}{c_2\, k}\right) e^{h_1} \left(1-\frac{h_1}{R_\la^2}\right)^{d-1} \dint h_1\\
 &\qquad \le  \int\limits_{-\infty}^{0}  (1+|h_1|)^{dw[\xi]}\,  e^{h_1} \left(1-\frac{h_1}{R_\la^2}\right)^{d-1} \dint h_1 + \int\limits_{0}^{R_\la^2}  (1+h_1)^{dw[\xi]}\, \exp\left(-\frac{e^{h_1}}{c_1\, k}\right) e^{h_1}  \dint h_1\\
&\qquad \le  \int\limits_{0}^{\infty}  (1+h_1)^{dw[\xi] + d - 1}\,  e^{-h_1} \dint h_1 + \int\limits_{0}^{\infty}  (1+h_1)^{dw[\xi]}\, \frac{2\, c_1^2\, k^2}{e^{2h_1}}\, e^{h_1}  \dint h_1\\ 
&\qquad = \int\limits_{0}^{\infty}  (1 + h_1)^{dw[\xi] + d - 1}\,  e^{-h_1} \dint h_1 + c_1\, k^2  \int\limits_{0}^{\infty}  (1+h_1)^{dw[\xi]}\,  e^{-h_1}  \dint h_1\\
&\qquad \le c_1\, (dw[\xi] + d)! + c_2\, k^2\, (dw[\xi]+1)!  \\
&\qquad= c_1 + c_2\, k^2 \le c_3\, k^2\,.
\end{align*}}
Now, we have $p-1$ further height coordinates $h_2,\ldots,h_p$. The last computation shows that, up to constants, we get an additional factor $k^2$ for each of these integrals so that the integration with respect to all height coordinates is bounded by a constant times $k^{2p}$. In view of the definition in \eqref{Leftover1} and by \eqref{eq:BoundStirlingPartitions} this leads to 
\begin{align*}
|T_5| &\le c_1\, c_2^k\, \|f\|_\infty^k\, R_\la^{d-1}\, k\, k!\,(dk)!\, (ku[\xi])!\,  ((kdv[\xi])!)^{2}\, k^{k-2}\,  k^{2p} \int\limits_{\SSd} \dint \sigma_{d-1}(u)\, \sum_{L_1,\ldots,L_p\, \preceq\setk}\, 1\\
&\le c_1\, c_2^k\, \|f\|_\infty^k\, R_\la^{d-1}\, (dk)!\, (ku[\xi])!\, (k!)^2\, ((kdv[\xi])!)^{2}\,  k^{3k}\,.
\end{align*}
This completes the proof.
\end{proof}

Next, we turn to the second summand $T_6$ in \eqref{eq:UpperBoundT4T6} involving the max-term. 

\begin{lem}\label{II}
For $T_6$ defined by \eqref{Leftover2} we have that
\begin{align*}
|T_6| \le c_1\, c_2^k\, \|f\|_\infty^k\, R_\la^{d-1}\, (dk)!\, (ku[\xi])!\, (k!)^2\, ((kdv[\xi])!)^{2}\,  k^{3k}\,.
\end{align*}
\end{lem}
\begin{proof}
As in the proof of Lemma \ref{I} we start with the integral with respect to the variable $h_1$. Putting $a:= \max\{|h_2|,\ldots,|h_p|\}$ we can rewrite this integral as
{ \allowdisplaybreaks
\begin{align*}
&\int\limits_{-\infty}^{R_\la^2}  (1+|h_1|)^{dw[\xi]}\, \exp\left(-\frac{e^{h_1\lor 0}}{c_1\, k}\right) e^{h_1} \left(1-\frac{h_1}{R_\la^2}\right)^{d-1} \left(\max \{\left|h_1\right|,\ldots,|h_p|\}\right)^{(d-1)(p-1)}  \dint h_1\\
& = \int\limits_{-\infty}^{0}  (1+|h_1|)^{dw[\xi]}\, \exp\left(-\frac{e^{h_1\lor 0}}{c_1\, k}\right) e^{h_1} \left(1-\frac{h_1}{R_\la^2}\right)^{d-1} \left(\max \{\left|h_1\right|,a\}\right)^{(d-1)(p-1)}  \dint h_1\\
&\qquad  + \int\limits_{0}^{R_\la^2}  (1+|h_1|)^{dw[\xi]}\, \exp\left(-\frac{e^{h_1\lor 0}}{c_2\, k}\right) e^{h_1} \left(1-\frac{h_1}{R_\la^2}\right)^{d-1} \left(\max \{\left|h_1\right|,a\}\right)^{(d-1)(p-1)}  \dint h_1\\
& \le \int\limits_{-\infty}^{0}  (1+|h_1|)^{dw[\xi]}\,  e^{h_1} \left(1-\frac{h_1}{R_\la^2}\right)^{d-1} \left(\max \{\left|h_1\right|,a\}\right)^{(d-1)(p-1)}  \dint h_1\\
&\qquad  + \int\limits_{0}^{R_\la^2}  (1+h_1)^{dw[\xi]}\, \exp\left(-\frac{e^{h_1}}{c_1\, k}\right) e^{h_1} \left(\max \{\left|h_1\right|,a\}\right)^{(d-1)(p-1)}  \dint h_1\\
& = \int\limits_{0}^{\infty}  (1+h_1)^{dw[\xi] + d - 1}\,  e^{-h_1}  \left(\max \{\left|h_1\right|,a\}\right)^{(d-1)(p-1)}  \dint h_1\\
&\qquad + \int\limits_{0}^{\infty}  (1+h_1)^{dw[\xi]}\, \exp\left(-\frac{e^{h_1}}{c_1\, k}\right) e^{h_1} \left(\max \{\left|h_1\right|,a\}\right)^{(d-1)(p-1)}  \dint h_1\\
&  =: T_7 + T_8\,.
\end{align*}}
Both terms, $T_7$ and $T_8$, will be treated separately. Since
\begin{align*}
dw[\xi] + (d-1) + (d-1)(p-1) = dw[\xi] + (d-1)p \le dw[\xi] + dk\,,
\end{align*}
we obtain together with \eqref{UngleichungFakultät2} that
\begin{align*}
T_7 &= \int\limits_{0}^{\infty}  (1+h_1)^{dw[\xi] + d -1}\,  e^{-h_1}  \left(\max \{\left|h_1\right|,a\}\right)^{(d-1)(p-1)}  \dint h_1\\
&= \int\limits_{0}^{a}  (1+h_1)^{dw[\xi] + d -1}\,  e^{-h_1}  a^{(d-1)(p-1)}\,  \dint h_1 + \int\limits_{a}^{\infty}  (1+h_1)^{dw[\xi] + d -1}\,  e^{-h_1}  h_1^{(d-1)(p-1)}\,  \dint h_1\\
&\leq a^{(d-1)(p-1)}  \int\limits_{0}^{a}  (1+h_1)^{dw[\xi] + d -1}\,  e^{-h_1}   \dint h_1 + \int\limits_{a}^{\infty}  (1+h_1)^{dw[\xi] + d -1}\,  e^{-h_1}  (1 + h_1)^{(d-1)(p-1)}\,  \dint h_1\\
&\le a^{(d-1)(p-1)}  \int\limits_{0}^{\infty}  (1+h_1)^{dw[\xi] + d -1}\,  e^{-h_1}   \dint h_1 + \int\limits_{0}^{\infty}  (1+h_1)^{dw[\xi] + dk}\,  e^{-h_1}  \dint h_1\\
&\le c_1\, (dw[\xi] + d)!\, a^{(d-1)(p-1)} + c_2\, (dw[\xi] + dk)! \\
&\le c_3\, a^{(d-1)(p-1)} + c_4\, c_5^k\, (dk)!\,.
\end{align*}
Moreover, applying \eqref{InequalityExponential} in the second, the estimate $dw[\xi] + (d-1)(p-1) \le dw[\xi] + dk$ in the fourth and \eqref{UngleichungFakultät2} in the last step, we conclude that
\begin{align*}
T_8 &= \int\limits_{0}^{\infty}  (1+h_1)^{dw[\xi]}\, \exp\left(-\frac{e^{h_1}}{c_1\, k}\right) e^{h_1} \left(\max \{\left|h_1\right|,a\}\right)^{(d-1)(p-1)}  \dint h_1\\
&\le c_1\, k^2 \int\limits_{0}^{\infty}  (1+h_1)^{dw[\xi]}\,  e^{-h_1} \left(\max \{\left|h_1\right|,a\}\right)^{(d-1)(p-1)}  \dint h_1\\
&= c_1\, k^2 \int\limits_{0}^{a}  (1+h_1)^{dw[\xi]}\, e^{-h_1} a^{(d-1)(p-1)}\,  \dint h_1 + c_2\, k^2 \int\limits_{a}^{\infty}  (1+h_1)^{dw[\xi]}\, e^{-h_1} h_1^{(d-1)(p-1)}\,  \dint h_1\\
&\le c_1\, k^2\, a^{(d-1)(p-1)} \int\limits_{0}^{\infty}  (1+h_1)^{dw[\xi]}\,  e^{-h_1}  \dint h_1 + c_2\, k^2 \int\limits_{0}^{\infty}  (1+h_1)^{dw[\xi] + dk}\,  e^{-h_1}  \dint h_1\\
&\le c_1\, k^2\, (dw[\xi])!\, a^{(d-1)(p-1)} + c_2\, k^2\, (dw[\xi] + dk)! \\
&\le c_3\, k^2\, a^{(d-1)(p-1)} + c_4\, c_5^k\, k^2\, (dk)!\,.
\end{align*}
Combining the two bounds for $T_7$ and $T_8$ we arrive at 
\begin{align*}
&\int\limits_{-\infty}^{R_\la^2}  (1+|h_1|)^{dw[\xi]} \exp\left(-\frac{e^{h_1\lor 0}}{c_1\, k}\right) e^{h_1} \left(1-\frac{h_1}{R_\la^2}\right)^{d-1} \left(\max \{\left|h_1\right|,\ldots,|h_p|\}\right)^{(d-1)(p-1)}  \dint h_1\\
&\qquad \le c_1\, a^{(d-1)(p-1)} + c_2\, c_3^k\, (dk)! + c_4\, k^2\, a^{(d-1)(p-1)} + c_5\, c_6^k\, k^2\, (dk)!\\
&\qquad \le c_1\, k^2\, \left(\max\{|h_2|,\ldots,|h_p|\}\right)^{(d-1)(p-1)} + c_2\, c_3^k\,  k^2\, (dk)!\,.
\end{align*}
For the term $T_6$ defined at \eqref{Leftover2} this implies that
\begin{align*}
&|T_6| \le c_1\, c_2^k\, \|f\|_\infty^k\, R_\la^{d-1}\, k!\,  (dk)!\, (ku[\xi])!\, ((kdv[\xi])!)^{2}\, k^{k-1}\, k^2 \sum_{L_1,\ldots,L_p\,  \preceq\setk}\, \int\limits_{\SSd} \int\limits_{-\infty}^{R_\la^2} \ldots \int\limits_{-\infty}^{R_\la^2}\,\\
&\qquad \qquad \times \prod_{i = 2}^{p} \left[(1+\left|h_i\right|)^{dw[\xi]} \exp\left(-\frac{e^{h_i\lor 0}}{c_3\, k}\right) e^{h_i} \left(1-\frac{h_i}{R_\la^2}\right)^{d-1}\right] \dint h_2 \ldots \dint h_p \dint \sigma_{d-1}(u)\\
&\qquad + c_4\, c_5^k\, \|f\|_\infty^k\, R_\la^{d-1}\, k!\, (ku[\xi])!\,  ((kdv[\xi])!)^{2}\, k^{k-1}\, k^2\\
&\qquad\qquad\times \sum_{L_1,\ldots,L_p\,  \preceq\setk}\, \int\limits_{\SSd} \int\limits_{-\infty}^{R_\la^2} \ldots \int\limits_{-\infty}^{R_\la^2}\,\left(\max\limits_{i=2,\ldots,p} \{\left|h_i\right|\}\right)^{(d-1)(p-1)} \\
&\qquad \qquad \times  \prod_{i = 2}^{p} \left[(1+\left|h_i\right|)^{dw[\xi]} \exp\left(-\frac{e^{h_i\lor 0}}{c_6\, k}\right) e^{h_i} \left(1-\frac{h_i}{R_\la^2}\right)^{d-1}\right] \dint h_2 \ldots \dint h_p \dint \sigma_{d-1}(u)\,.
\end{align*} 
The first summand is almost the same as in Lemma \ref{I}. The only difference is that there are $p-1$ further integrals concerning the height coordinates $h_2,\ldots,h_p$ left. Carrying out the integrations as above, this yields another factor $k^{2(p-1)}$, up to constants. Furthermore, by computations similarly to those performed above we see that
\begin{align*}
&\int\limits_{-\infty}^{R_\la^2}\, (1+\left|h_2\right|)^{dw[\xi]} \exp\left(-\frac{e^{h_2\lor 0}}{c_1\, k}\right) e^{h_2} \left(1-\frac{h_2}{R_\la^2}\right)^{d-1} \left(\max \{\left|h_2\right|,\ldots,|h_p|\}\right)^{(d-1)(p-1)} \dint h_2\\
&\qquad \le c_1\, k^2\, \left(\max\{|h_3|,\ldots,|h_p|\}\right)^{(d-1)(p-1)} + c_2\, c_3^k\, k^2\, (dk)!\,,
\end{align*}
giving the bound 
\begin{align*}
&|T_6| \le c_1\, c_2^k\, \|f\|_\infty^k\, R_\la^{d-1}\, (dk)!\, (ku[\xi])!\, (k!)^2\, ((kdv[\xi])!)^{2}\,  k^{3k}\\
& + c_3\, c_4^k\, \|f\|_\infty^k\, R_\la^{d-1}\, k!\, (dk)!\, (ku[\xi])!\, ((kdv[\xi])!)^{2}\, k^{k-1}\, k^4 \sum_{L_1,\ldots,L_p\, \preceq\setk}\, \int\limits_{\SSd} \int\limits_{-\infty}^{R_\la^2} \ldots \int\limits_{-\infty}^{R_\la^2}\,\\
&\qquad \qquad \times \prod_{i = 3}^{p} \left[(1+\left|h_i\right|)^{dw[\xi]} \exp\left(-\frac{e^{h_i\lor 0}}{c_5\, k}\right) e^{h_i} \left(1-\frac{h_i}{R_\la^2}\right)^{d-1}\right] \dint h_3 \ldots \dint h_p \dint \sigma_{d-1}(u)\\
& + c_6\, c_7^k\, \|f\|_\infty^k\, R_\la^{d-1}\, k!\, (ku[\xi])!\, ((kdv[\xi])!)^{2}\, k^{k-1}\, k^4 \\
&\qquad\qquad\times \sum_{L_1,\ldots,L_p\,\preceq\setk}\, \int\limits_{\SSd} \int\limits_{-\infty}^{R_\la^2} \ldots \int\limits_{-\infty}^{R_\la^2} \left(\max\limits_{i=3,\ldots,p} \{\left|h_i\right|\}\right)^{(d-1)(p-1)} \\
&\qquad \qquad \times  \prod_{i = 3}^{p} \left[(1+\left|h_i\right|)^{dw[\xi]} \exp\left(-\frac{e^{h_i\lor 0}}{c_8\, k}\right) e^{h_i} \left(1-\frac{h_i}{R_\la^2}\right)^{d-1}\right] \dint h_3 \ldots \dint h_p \dint \sigma_{d-1}(u)\,.
\end{align*} 
Repeating this procedure $p-2$ further times yields
 \begin{align*}
& |T_6| \le c_1\, c_2^k\, \|f\|_\infty^k\, R_\la^{d-1}\, (dk)!\, (ku[\xi])!\, (k!)^2\, ((kdv[\xi])!)^{2}\,  k^{3k}\\
& + c_3\, c_4^k\,  \|f\|_\infty^k\, R_\la^{d-1}\, k!\, (dk)!\, (ku[\xi])!\, ((kdv[\xi])!)^{2}\, k^{k-1}\, k^{2(p-1)} \sum_{L_1,\ldots,L_p\, \preceq\setk}\, \int\limits_{\SSd} \int\limits_{-\infty}^{R_\la^2}\\
 & \qquad \qquad \times (1+\left|h_p\right|)^{dw[\xi]} \exp\left(-\frac{e^{h_p\lor 0}}{c_5\, k}\right) e^{h_p} \left(1-\frac{h_p}{R_\la^2}\right)^{d-1} \dint h_p \dint \sigma_{d-1}(u)\\
 & + c_6\, c_7^k\, \|f\|_\infty^k\, R_\la^{d-1}\, k!\, (ku[\xi])!\, ((kdv[\xi])!)^{2}\, k^{k-1}\, k^{2(p-1)} \sum_{L_1,\ldots,L_p\, \preceq\setk}\, \int\limits_{\SSd} \int\limits_{-\infty}^{R_\la^2}\, \left|h_p\right|^{(d-1)(p-1)} \\
 &\qquad \qquad \times (1+\left|h_p\right|)^{dw[\xi]} \exp\left(-\frac{e^{h_p\lor 0}}{c_8\, k}\right) e^{h_p} \left(1-\frac{h_p}{R_\la^2}\right)^{d-1} \dint h_p \dint \sigma_{d-1}(u)\,.
 \end{align*} 
The integral concerning $h_p$ in the second summand is bounded by a constant times $k^2$ as we have already seen in Lemma \ref{I}.
For the third summand it follows by a similar reasoning as above that
\begin{align*}
&\int\limits_{-\infty}^{R_\la^2}\, (1+\left|h_p\right|)^{dw[\xi]} \exp\left(-\frac{e^{h_p\lor 0}}{c_1\, k}\right) e^{h_p} \left(1-\frac{h_p}{R_\la^2}\right)^{d-1} \left|h_p\right|^{(d-1)(p-1)} \dint h_p\\
& \le  \int\limits_{0}^{\infty}\, (1+ h_p)^{dw[\xi] + dk}\,  e^{- h_p} \dint h_p  + c_1\, k^2 \int\limits_{0}^{\infty}\, (1+ h_p)^{dw[\xi] + dk}\,  e^{-h_p} \,\dint h_p\\
& \le c_1\, (dw[\xi] + dk)! + c_2\, k^2\, (dw[\xi] + dk)! \le c_3\, c_4^k\, k^2\, (dk)!\,.
\end{align*}
This completes the proof.
\end{proof}

We can now combine the two previous lemmas into a bound for the off-diagonal term that complements the diagonal bound in Lemma \ref{Diagonal}.

\begin{lem}\label{off-diagonal}
Let $\xi\in\Xi$ and $f\in\cB(\RR^d)$. Then,
\begin{align*}
\begin{split}
&\Bigg|\sum_{S,T\, \preceq\setk}\, \int\limits_{\sigma(\{S,T\})} f_{R_\la}^k\, \dint c_\la^k\, \Bigg| 
\le c_1\, c_2^k\, \|f\|_\infty^k\, R_\la^{d-1}\, (dk)!\, (ku[\xi])!\, (k!)^2\, ((kdv[\xi])!)^{2}\,  k^{3k}\,.
\end{split}
\end{align*}
\end{lem}
\begin{proof}
Recalling \eqref{eq:UpperBoundT4T6}, we have that
\begin{align*}
\Bigg|\sum_{S,T\, \preceq\setk}\, \int\limits_{\sigma(\{S,T\})} f_{R_\la}^k\, \dint c_\la^k\, \Bigg| \le |T_5| + |T_6|
\end{align*}
with the terms $T_5$ and $T_6$ defined by \eqref{Leftover1} and \eqref{Leftover2}, respectively. Applying now Lemma \ref{I} to $T_5$ and Lemma \ref{II} to $T_6$ yields the desired upper bound.
\end{proof}

What is left is to combine the two estimates for the on- and the off-diagonal term to complete the proof of the cumulant bound.

\begin{proof}[Proof of Theorem \ref{cumulantestimate}.]
From Lemma \ref{Diagonal} and Lemma \ref{off-diagonal} we deduce the bound
\begin{equation}\label{eq:FinalProof1}
\begin{split}
\big|\big\langle f_{R_\la}^k,c_\lambda^k\big\rangle\big| &\leq \Bigg|\int\limits_\Delta f_{R_\la}^k\, \dint c_\lambda^k\Bigg| + \Bigg|\sum_{S,T\, \preceq\setk}\, \int\limits_{\delta(\{S,T\})} f_{R_\la}^k\, \dint c_\lambda^k\Bigg|\\
&\le c_1\, c_2^k\, \|f\|_\infty^k\, R_\la^{d-1}\,  (ku[\xi])!\, ((kdv[\xi])!)^{3} \\
&\qquad\qquad+ c_3\, c_4^k\, \|f\|_\infty^k\, R_\la^{d-1}\, (dk)!\, (ku[\xi])!\, (k!)^2\, ((kdv[\xi])!)^{2}\,  k^{3k}\,,
\end{split}
\end{equation}
which holds for sufficiently large $\lambda$. Now, we notice that the elementary inequality $\ell^\ell \leq \ell!\, e^{3\ell}$, valid for all $\ell\in \{3,4,\ldots\}$, leads to
\begin{align*}
k^{3k} \leq (k!)^3\, e^{9k}
\end{align*} 
and we also observe that Stirling's formula gives
\begin{align*}
(kdv[\xi])! &\leq e(kdv[\xi])^{kdv[\xi]+\frac{1}{2}} e^{-kdv[\xi]} = e\sqrt{kdv[\xi]}e^{-kdv[\xi]}(dv[\xi])^{kdv[\xi]}(k^k)^{dv[\xi]}\\
&\leq e\sqrt{kdv[\xi]}e^{-kdv[\xi]}(dv[\xi])^{kdv[\xi]}e^{3kdv[\xi]}(k!)^{dv[\xi]}
\le ((dv[\xi])^{dv[\xi]}e^{4dv[\xi]})^k (k!)^{dv[\xi]}\,.
\end{align*} 
Combined with \eqref{eq:FinalProof1} this implies that
\begin{align*}
\big|\big\langle f_{R_\la}^k,c_\lambda^k\big\rangle\big| &\le c_1\, c_2^k\, \|f\|_\infty^k\, R_\la^{d-1}\, (k!)^{u[\xi] + 3dv[\xi]} + c_3\, c_4^k\, \|f\|_\infty^k\, R_\la^{d-1}\, (k!)^{d + u[\xi] + 5 + 2dv[\xi]}\\
&\le c_1\, c_2^k\, \|f\|_\infty^k\, R_\la^{d-1}\, (k!)^{3dv[\xi] + u[\xi] + 5 + z[\xi]}
\end{align*}
for all sufficiently large $\la$, where we recall that $z[\xi]$ is $d$ if $\xi = \xi_{f_0}$ and zero otherwise. This completes the proof of the cumulant bound.
\end{proof}

\subsubsection*{Acknowledgement}
We would like to thank Pierre Calka (Rouen) and Matthias Reitzner (Osnabr\"uck) for stimulating and enlightening discussions. We also thank two anonymous referees for their comments and remarks that helped us to improve the presentation of our results.


\end{document}